\documentclass[11pt]{article}
\usepackage[hmargin=2.5cm,vmargin=2.2cm]{geometry}
\usepackage[english]{babel}
\usepackage{amsmath,amssymb,amsthm}
\usepackage{mathrsfs}
\usepackage[pdfencoding=auto, psdextra]{hyperref}
\hypersetup{colorlinks=true,allcolors=[rgb]{0,0,0.6}}
\usepackage{enumitem}
\theoremstyle{plain}
\newtheorem{thm}{Theorem}[section]
\newtheorem{cor}[thm]{Corollary}
\newtheorem{prop}[thm]{Proposition}
\newtheorem{lem}[thm]{Lemma}
\newtheorem{defn}[thm]{Definition}

\newtheorem{remark}[thm]{Remark}

\title{On well-posedness and uniqueness for general hierarchy equations of Gross-Pitaevskii and Hartree type}
\author{Z.~Ammari $^\dagger$, Q.~Liard \thanks{LAGA, UMR-CNRS 9345, Universit\'e de Paris 13, av. J.B. Cl\'ement, 93430 Villetaneuse, France.} , C.~Rouffort \thanks{IRMAR, Universit{\'e} de Rennes I, UMR-CNRS 6625, Campus de Beaulieu, 35042 Rennes Cedex, France} }
\date{\today}

\begin{document}

%%%%%%%%%%%%%UserMacro %%%%%%%%%%%%
\def\N{\mathbb{N}}
\def\Z{\mathbb{Z}}
\def\R{\mathbb{R}}
\def\C{\mathbb{C}}
\def\P{\mathbb{P}}
\def\Tr{\mathrm{Tr}}
\def\Id{\mathrm{Id}}
\def\Im{\mathrm{Im}}
\def\Re{\mathrm{Re}}
\newcommand{\zed}{\mathscr{Z}_{0}}
\newcommand{\zeds}{\mathscr{Z}_{s}}
\newcommand{\zedsi}{\mathscr{Z}_{-\sigma}}
%%%%%%%%%%%%%%%%%%%%%%%%%%%%%%%

\maketitle
\begin{abstract}
Gross-Pitaevskii and Hartree hierarchies  are  infinite systems of coupled  PDEs  emerging  naturally from  the mean field theory of Bose gases. Their  solutions are known to be related to   initial value problems, respectively the  Gross-Pitaevskii and Hartree equations.    Due to their physical and mathematical relevance,  the issues of well-posedness and uniqueness for these hierarchies have recently  been studied thoroughly using specific  nonlinear and combinatorial techniques. In this article, we introduce a new approach for the study of such hierarchy equations by firstly establishing  a duality between them and certain  Liouville equations and secondly solving the uniqueness and existence questions for the latter. As an outcome, we  formulate a hierarchy equation starting from any initial value
problem which is $U(1)$-invariant and prove a general principle which can be stated formally as follows:
\begin{itemize}
\item [(i)]  \emph{Uniqueness for weak solutions of an initial value problem implies the uniqueness  of solutions
for the related hierarchy equation.}
\item [(ii)]   \emph{Existence of solutions for an initial value problem  implies  the existence of solutions
for the related hierarchy  equation.}
\end{itemize}
In particular, several new well-posedness results as well as a counterexample to uniqueness for the Gross-Pitaevskii hierarchy equation are proved.  The novelty in our work lies in the  aforementioned duality and the use of Liouville equations with powerful transport techniques  extended to infinite dimensional functional spaces.
\end{abstract}

\vspace{.7in}


\newpage
\setcounter{tocdepth}{2}
\tableofcontents
\newpage

\section{Introduction}
\label{sec:00}
The so-called Hartree and  Gross-Pitaevskii hierarchies are infinite systems of coupled PDEs  taking usually the following form,
\begin{equation}
\label{GPH-cubic}
  \left\{
    \begin{aligned}
    &i \partial_t \gamma^{(k)} = (-\Delta_{X_k}+\Delta_{X'_k}) \gamma^{(k)}  +  \, B_{k}\gamma^{(k+1)}\,, & \\
    &\gamma_{|t=0}^{(k)}=\gamma_0^{(k)}\,,& \\
    \end{aligned}
    \qquad
\forall k\in\N\,,
  \right.
\end{equation}
 where $\Delta_{X_k}$ (or $\Delta_{X'_k}$) denotes  the standard Laplacian on $\R^{dk}$ with  $X_k=(x_1,\cdots,x_k)\in\R^{dk}$, $X'_k=(x'_1,\cdots,x'_k)\in\R^{dk}$ are the configuration space coordinates and  $x_j,x'_j\in \R^d$ for each  $j\in\{1,\cdots, k\}$.  The solutions of these hierarchies are sequences of complex-valued functions $(\gamma^{(k)})_{k\in\N}$ such that $\gamma^{(k)}: (t,X_k,X'_k)\longmapsto \gamma^{(k)}(t,X_k,X'_k)$ is defined over a domain $I\times \R^{dk}\times \R^{dk}$ with  $I$ is a real interval  containing the origin;    and  $(\gamma^{(k)})_{k\in\N}$ fulfills   the equation \eqref{GPH-cubic} with a prescribed initial condition $\gamma_0^{(k)}:\R^{dk}\times \R^{dk}\longrightarrow \C$.
Furthermore, the interaction term in the right hand side of \eqref{GPH-cubic} is given by the following expression:
\begin{equation}
\label{eq.bk}
\begin{aligned}
B_{k}\gamma^{(k+1)}&:=&B_{k}^{+} \gamma^{(k+1)}-B_{k}^{-} \gamma^{(k+1)}\\
&:=&\sum_{j=1}^{k}B^{+}_{j,k}\gamma^{(k+1)}-\sum_{j=1}^{k}B^{-}_{j,k}\gamma^{(k+1)}\,,
\end{aligned}
\end{equation}
where $B^\pm_{j,k}$ are defined for $1\leq j\leq k$  by
\begin{equation}
\label{Bjk-1}
(B^+_{j,k}\gamma^{(k+1)})(t, X_k, X'_k) :=\int_{\R^d} \gamma^{(k+1)}(t,X_k, y, X_k', y)
\,V (x_j - y)\,dy\,,
\end{equation}
and
\begin{equation}
\label{Bjk-2}
(B^-_{j,k} \gamma^{(k)})(t, X_k, X_k') := \int_{\R^d} \gamma^{(k+1)}(t,X_k, y, X_k', y) \,
V (x'_j - y) \,dy\,.
\end{equation}
Here  $V$ is either  an even real-valued measurable function $V:\R^d\to \R$ in the case of the Hartree hierarchy  or a multiple of a Dirac delta function $V=\lambda\delta_0$ in the case of the Gross-Pitaevskii hierarchy. The  parameter $\lambda$  refers to a coupling constant  taking positive or negative values and accounting respectively for a repulsive (defocusing) or an attractive (focusing)  interaction in the Gross-Pitaevskii case.
Additionally, one usually requires two  further physical constraints on the solutions $(\gamma^{k})_{k\in\N}\,$, namely:
\begin{itemize}
\item \underline{Symmetry}: For any permutations $\sigma,\pi$ in the symmetric group $\mathfrak{S}_k$,
  \begin{equation}
  \label{cont1}
\gamma^{(k)}(t, x_{\sigma(1)},\cdots,x_{\sigma(k)};
x'_{\pi(1)},\cdots,x'_{\pi(k)})=\gamma^{(k)}(t, X_k; X'_k).
\end{equation}
\item \underline{Finite density}: Each $\gamma^{(k)}$ is the kernel of a trace class operator on $L^2(\R^{dk})$ satisfying the following inequality in the operator sense,
\begin{equation}
\label{cont2}
0\leq \gamma^{(k)} \leq 1 \,.
\end{equation}
\end{itemize}
The main questions that can be raised  about the equations \eqref{GPH-cubic} as  systems of PDEs are of course uniqueness, existence and stability of solutions. Although these hierarchy equations \eqref{GPH-cubic} are  physically relevant and they  have been known in the physical literature since long time  (see e.g.~ \cite{MR0223148, MR578142} and references therein), their  mathematical investigation started rather recently. Indeed, the  study of the  Hartree and the Gross-Pitaevskii hierarchies have started attracting a wide increasing interest,  only after a remarkable progress was made in the mean field theory of Bose gases (see e.g.~\cite{MR2257859,MR2276262,MR2525781,MR2680421,MR2504864,MR2657816}).  In particular, the  uniqueness
property of solutions for the Gross-Pitaevskii  hierarchy equation was pointed out as a crucial steep for the derivation of the dynamics of Bose-Einstein condensates from many-body quantum mechanics, see \cite{MR2209131,MR2276262,MR2377632}.   Since then, the issues of  well-posedness and uniqueness   for hierarchy equations similar to \eqref{GPH-cubic} are  considered to be   interesting problems combining specific combinatorial  and nonlinear analytic difficulties
\cite{MR3210237,MR3385343,MR3165917,MR3246038,MR3551830,MR3500833,MR3419755}.
Consequently, this stimulated  a current  trend among the PDE and mathematical physics communities, focused on the study of hierarchy equations for their own interest, see e.g.~\cite{MR3360742,MR3293448,MR3013052,MR3395127,MR3466843,MR3170216}.

The mathematical mean field theory of Bose gases was initiated in the pioneering works of Hepp \cite{MR0332046} and Ginibre-Velo \cite{MR530915,MR539736} in the 70s. Afterward, the subject has been revived under the impetus of several contributions, e.g.~ \cite{MR2104130,MR1869286, MR2291792,MR2821235, MR1926667}. One of the purposes of this topic was the rigorous  justification of the mean field approximation and the derivation of  the Hartree, NLS and Gross-Pitaevskii equations from first principles of quantum mechanics. Thus, ultimately  explaining how microscopic effects come at play into the physical macroscopic phenomena of Bose-Einstein condensates. For the sake of clarity, we briefly recall the relationship between the hierarchy equations \eqref{GPH-cubic} and the mean-field approximation of many-body quantum dynamics. For that, consider the many-body   Schr\"odinger operator
\begin{equation}
\label{eq.hn}
H_{n}=\sum_{i=1}^{n}-\Delta_{x_i}+\frac{1}{n}\sum_{1\leq i<j \leq n}W_{n}(x_i-x_j),\,
\end{equation}
where $\,W_{n}(x)=n^{d\beta}W(n^{\beta}x)$ for any $x \in \R^{d}$ and some $\beta \in [0,1]$  with $W$ is  a  real-valued even potential. Under  some reasonable assumptions on $W$, the operator $H_{n}$ is self-adjoint over the space $L_s^{2}(\mathbb{R}^{dn})$ and the quantum dynamics related to \eqref{eq.hn} are well-defined.
Here $L_s^{2}(\mathbb{R}^{dn})$ is the space of symmetric square integrable functions $\Psi\in
  L^{2}(\mathbb{R}^{dn})$ such that $\Psi(x_1,\cdots,x_n)=\Psi(x_{\sigma_1},\dots,x_{\sigma_n})
$
for any permutation $\sigma$ in $\mathfrak{S}_n$. In  particular, if $\varrho_n$ is a normal state or a
density matrix (i.e., a normalized positive trace class operator on $L_s^{2}(\mathbb{R}^{dn})$) describing the quantum system at the initial time $t=0$, then according to the Heisenberg equation the system evolves at time $t$ towards   the state,
\begin{equation}
\label{qstate}
\varrho_n(t)=e^{it H_n} \varrho_n e^{-it H_n}\,.
\end{equation}
 One of the popular methods that explains the mean-field approximation is inspired by classical statistical mechanics and known  as the BBGKY hierarchy approach. In fact, consider all the marginals $\varrho^{(k)}_{n}(t)$, $1\leq k\leq n$,  given by their kernels
\begin{equation}
\label{corr-funct}
\varrho^{(k)}_{n}(t,X_k;X'_k):=\displaystyle\int_{\R^{d(n-k)}} \varrho_n(t,X_k,y;X'_k,y)\,\,dy,
\quad \text{ for all }  \;(X_k,X'_k)\in \R^{2dk},
\end{equation}
then $\varrho^{(k)}_{n}(t)$ are density matrices over $L_s^{2}(\mathbb{R}^{dk})$ usually called
reduced density matrices of the state $\varrho_n(t)$. Thus, the Heisenberg equation on $\varrho_n(t)$  yields the so-called BBGKY hierarchy on the marginals,
\begin{equation}
\label{BBGKY}
\begin{aligned}
i \partial_{t}\varrho^{(k)}_{n}(t)&=
\displaystyle\sum_{j=1}^{k}\big[-\Delta_{x_{j}},\varrho_{n}^{(k)}(t)\big]+\frac{1}{n}\displaystyle \sum_{1\leq i<j\leq k} \big[W_n(x_i-x_j),\varrho_n^{(k)}(t)\big]\\
&+\frac{n-k}{n} \, \sum_{j=1}^{k}\Tr_{k+1}[W_n(x_j-x_{k+1}),\varrho_{n}^{(k+1)}(t)]\,,
\end{aligned}
\end{equation}
where the brackets $[\cdot,\cdot]$ denotes the commutator defined as $[A,B]=AB-BA$ and $\Tr_{k+1}$ denotes the partial trace over the last variable $x_{k+1}$. This means that  $\Tr_{k+1}[W_n(x_j-x_{k+1}),\varrho_{n}^{(k+1)}(t)] $ is an operator who's kernel $K$ is given by the following  formula,
\begin{eqnarray*}
K(X_k,X'_k)=\int_{\R^d} \gamma^{(k+1)}(t,X_k,x_{k+1}, X_k', x_{k+1}) \,W_n (x_j - x_{k+1})\,dx_{k+1} -\hspace{1in}\\
\hspace{2in}\int_{\R^d} \gamma^{(k+1)} (t,X_k, x_{k+1}', X_k', x_{k+1}') \, W_n(x'_j - x_{k+1}') \,dx_{k+1}'\,.
\end{eqnarray*}
 The mean field approximation of quantum dynamics is usually  understood within the BBGKY hierarchy approach as the  two following statements:
\begin{enumerate}[label=\textnormal{(\alph*)}]
\item \label{chaos1} If $\varrho_n$ is an uncorrelated initial state $\varrho_n=|\varphi^{\otimes n}\rangle\langle \varphi^{\otimes n}|$, with $||\varphi||_{L^2(\R^d)}=1$,
then the marginal  $\varrho_n^{(k)}(t)$ satisfying the equation \eqref{BBGKY}
converges as $n\to \infty$ towards a state $|\varphi_t^{\otimes k}\rangle\langle \varphi_t^{\otimes k}|$ for any $k\in\N$ and any time $t\in\R$ in the sense that,
\begin{eqnarray*}
\lim_{n\to\infty}\Tr[\varrho_n^{(k)}(t) A]=\langle\varphi_t^{\otimes k}, A \varphi_t^{\otimes k}\rangle_{L^2(\mathbb{R}^{dk})}, \;
\end{eqnarray*}
for any bounded (or compact) operator $A$ on $L^{2}(\mathbb{R}^{dk})$ ($k$ is kept fixed while $n\to \infty$).
\item \label{chaos2} Furthermore, $\varphi_t$ is the solution of the nonlinear  NLS or  Hartree equation
\begin{eqnarray}
\label{hartree}
\left\{
  \begin{array}[c]{l}
    i\partial_t \varphi_t=-\Delta \varphi_t+(V*\left|\varphi_t\right|^{2})\varphi_t\,,\\
    \varphi_{0}=\varphi\,,
  \end{array}
\right.
\end{eqnarray}
\end{enumerate}
with $V$ is a potential that depends on $W$  and the parameter $\beta$ according to the following dichotomy:
\begin{equation}
\label{eq.lambda}
V=
\begin{cases}
W\,  &\text{ if } \;\beta=0,\\
\lambda\;\delta_0  &\text{ if } \;0<\beta\leq 1.
\end{cases}
\end{equation}
Here $\lambda$ is a coupling constant  which depends on the value of $\beta$. The above statements usually go under the name of propagation of chaos \cite{MR2327286}. For a general overview and more details on the subject we refer to  the recent book  \cite{MR3382225}.  The strategy of the BBGKY approach goes through three steps that are:
\begin{description}
  \item (i) Compactness.
  \item  (ii) Convergence.
  \item (iii) Uniqueness.
\end{description}
Step (ii) and (iii) are the most delicate parts specially when $W$ is an unbounded potential or $\beta>0$
($\beta=1$ seems the most difficult case even if $W$ is smooth and compactly supported). Under favorable assumptions  by using a compactness argument one proves that up to extracting  a subsequence $\varrho^{(k)}_{n}(t)$ converges when $n\to\infty$, with respect to the weak-$*$ topology, to a positive  trace class operator  denoted $\gamma^{(k)}(t)$ for each $k\in\N$ and any time $t$, i.e.,
\begin{eqnarray*}
\lim_{n\to\infty}\Tr[\varrho_n^{(k)}(t) \;A]=\Tr[\gamma^{(k)}(t) \;A],
\end{eqnarray*}
for any $A$ compact operator on $L_s^{2}(\mathbb{R}^{dk})$. This is the step (i) and in particular one remarks that the kernels of the operators $(\gamma^{(k)}(t))_{k\in\N}$ satisfy the constraints \eqref{cont1}-\eqref{cont2}. Step (ii) consists in letting  $n$ tend  to infinity  in the BBGKY hierarchy  \eqref{BBGKY} and proving the convergence towards the Hartree or Gross-Pitaevskii  hierarchy  \eqref{GPH-cubic} written in its equivalent form in terms of trace class operators, i.e.,
\begin{equation}
\label{eq.infbbgky}
i \partial_{t}\gamma^{(k)}=
\displaystyle\sum_{j=1}^{k}\big[-\Delta_{x_{j}},\gamma^{(k)}\big]+B_{k}\gamma^{(k+1)}.
\end{equation}
 Here, the term $B_{k}\gamma^{(k+1)}$ is understood as an operator taking the short writing
  $$
 B_{k}\gamma^{(k+1)}= \displaystyle\sum_{j=1}^{k}\Tr_{k+1}[V(x_j-x_{k+1}),\gamma^{(k+1)}]\,,
 $$
 with its kernel coinciding with the one given in \eqref{eq.bk}-\eqref{Bjk-2}  according to the following identification,
$$
(B^+_{j,k}\gamma^{(k+1)})(t, X_k, X'_k) \;\text{ is the kernel of } \; \Tr_{k+1}[V(x_j-x_{k+1})\gamma^{(k+1)}]
$$
and
$$
(B^-_{j,k}\gamma^{(k+1)})(t, X_k, X'_k) \;\text{ is the kernel of } \; \Tr_{k+1}[\gamma^{(k+1)}V(x_j-x_{k+1})].
$$
Notice that the hierarchy equation  \eqref{GPH-cubic} or \eqref{eq.infbbgky} may not make sense. In fact,  some additional regularity  on $(\gamma^{(k)})_{k\in\N}$ is required in order to make every thing consistent. Nevertheless, at a formal level the limit of the BBGKY hierarchy \eqref{BBGKY} seems coherent  with the equation \eqref{eq.infbbgky}, although the question of convergence is more subtle than a straightforward limit.

 The last step (iii) consists of proving uniqueness of solutions for the hierarchy equations
 \eqref{GPH-cubic} or equivalently \eqref{eq.infbbgky}. In fact, the proof goes as follows.
Consider  an uncorrelated state $\varrho_n=|\varphi^{\otimes n}\rangle\langle \varphi^{\otimes n}|$ as in  \ref{chaos1}, the steps (i)-(ii) yield a solution $(\gamma^{(k)})_{k\in\N}$ of the  hierarchy equation \eqref{eq.infbbgky} satisfying the initial condition $\gamma^{(k)}_0=|\varphi^{\otimes k}\rangle\langle \varphi^{\otimes k}|$ for all $k\in\N$. Moreover,
one easily checks that  if $\varphi_t$ solves the NLS equation \eqref{hartree} then
$|\varphi_t^{\otimes k}\rangle\langle \varphi_t^{\otimes k}|$ is also a solution of the hierarchy equation \eqref{eq.infbbgky} satisfying  the same initial condition as $\gamma^{(k)}(t)$.  Therefore, if one proves that the Gross-Pitaevskii or the Hartree hierarchy \eqref{eq.infbbgky}
 admits  a single solution for each initial datum, then $\gamma^{(k)}(t)=|\varphi_t^{\otimes k}\rangle\langle \varphi_t^{\otimes k}|$ for all times. Since  the assertions  \ref{chaos1}-\ref{chaos2} are proved independently from any extraction of subsequences, then the propagation of chaos is established   in this way.  The above strategy was designed after several contributions that started with the  work of Spohn in \cite{MR578142} for bounded $W$ potentials then improved by Bardos-Golse-Mauser in \cite{MR1869286} (proof of (i)-(ii) for the Coulomb potential)  and in
   \cite{MR1926667} (proof of (iii) for the Coulomb potential) and later on enhanced into a powerful
   method in a series of papers by Erd\H{o}s-Schlein-Yau \cite{MR2680421,MR2525781,MR2276262,MR2257859} in order to tackle the dynamics of Bose-Einstein condensates.   There are of course other  approaches that justify the mean field approximation of quantum dynamics (see e.g.
   \cite{MR2953701,MR3379490,MR2657816,MR2504864,MR2313859,MR3317556,MR2821235}) and there are also other trends that focus for instance on rate of convergence \cite{MR2839064,MR3117522,MR3681700,MR3506807,MR3391830,MR2836427,MR2530155} or on the stationary mean field approximation of ground states for Bose gases (see e.g.~\cite{MR3161107,MR3310520,MR2143817}).
   Notice that it is not necessary to start with exceptional (uncorrelated) states as in \ref{chaos1}-\ref{chaos2}; one can formulate a stronger form for the mean field approximation which is state independent as suggested in the work of Ammari  and Nier (see \cite{MR2465733,MR2513969}).

Our main motivation in this article is the uniqueness or well posedness  problem (iii) for general hierarchy equations of type \eqref{GPH-cubic} as they may emerge from a mean field theory. So, we are not concerned with  (i) and (ii), although if these steps are proved then we expect that our results can be easily used to complete the BBGKY strategy and deduce the validity of the mean field approximation. Before explaining  our contribution, it is useful to highlight some remarkable results concerning the problem (iii). The first uniqueness results in the defocusing case ($\lambda>0$) for the Gross-Pitaeivskii hierarchy were  obtained in \cite{MR2680421,MR2525781,MR2276262,MR2257859} by using some sophisticated Feynman graph expansions. Later on, Klainerman and Machedon \cite{MR2377632} proved a {\it conditional} uniqueness theorem using a board game argument inspired by the Feynman graph expansion with space-time multilinear estimates based on the following a priori condition,
\begin{equation}
\label{klma}
\forall T>0, \exists R>0 \;\text{ s.t. } \quad \int_{0}^{T}\|S^{(k)}\;B_{j,k}^{\pm}\,\gamma^{(k)}(t)\|_{L^{2}(\R^{dk}\times \R^{dk})}\,dt < R^{k},\,\quad \forall k\in \N.
\end{equation}
Here $S^{(k)}$ denotes the following operator,
\begin{align}
\label{eq.tch1}
\mathcal{S}^{(k)}:=\prod_{j=1}^{k}(1-\Delta_{x_{j}})^{1/2}(1-\Delta_{x'_{j}})^{1/2}\,,
\end{align}
which  acts on the kernel of $B_{j,k}^{\pm}\gamma^{(k)}(t)$. A  large number of remarkable results
followed soon after these breakthroughs, see e.g.~ \cite{MR2988730,MR3013052,MR2747009,MR2683760,MR2662450,
MR3500833,MR3116008}. Moreover, other interesting aspects started to be explored like the mean field problem on tori \cite{MR3449225,MR3466843,MR3425265,MR3360742,MR3170216} or the focusing  case ($\lambda<0$)
\cite{MR3641881,MR3674169,MR3488534,MR3425265,MR2683760,MR2600687,MR2662450}.
Furthermore, in \cite{MR3385343}  the authors  gave a new proof, based on a  quantum de Finetti theorem (see \cite{MR2802894,MR3506807,MR3335056,MR3161107} and the discussion in Subsect.~\ref{sec:02}),  leading  to \emph{unconditional} uniqueness results for the Gross-Pitaevskii hierarchy  \eqref{GPH-cubic}. Subsequently, such proof was extended  to show low regularity well-posedness results and also to study the quintic Gross-Pitaevskii hierarchy equation,  see \cite{MR3419755,MR3395127}.

\bigskip
%breach
Beyond the importance and the profound implications  of these aforementioned uniqueness and well-posedness results, they all  relay on the same guiding idea which suggests the extension of nonlinear techniques
 (Strichartz, Morawetz, space-time inequalities, randomization\dots)  to the hierarchy equations  \eqref{GPH-cubic}.  This may seem somewhat surprising since the latter are linear equations. On another side, the hierarchy equations  \eqref{GPH-cubic}  inherit important properties from the BBGKY hierarchy \eqref{BBGKY}  revealing their statistical  nature and  which are usually neglected.   In particular, a  quantum de Finetti theorem, that will be explained in Subsect.~\ref{sec:02}, says that all the physically relevant solutions of  \eqref{GPH-cubic} have the following form for any $k\in \N $,
\begin{equation}
\label{int-dfinet}
\gamma^{(k)}(t)=\int_{L^2(\R^d)}  \,|\varphi^{\otimes k} \rangle \langle \varphi^{\otimes k} | \;d\mu_t,
\end{equation}
 where $\mu_t$ is a Borel probability measure on $L^2(\R^d)$ and $|\varphi^{\otimes k} \rangle \langle \varphi^{\otimes k} | $ is a rank one  projector over the space $L_s^2(\R^{dk})$.    Our first observation is  that any solutions
 of the hierarchy equation \eqref{GPH-cubic}  yields in a natural way  a solution $(\mu_t)_{t\in I}$  of a Liouville (or transport)   equation
 \begin{equation}
 \label{int-liouville}
\partial_{t} \mu_{t} + \nabla^{T} (v \cdot \mu_{t}) = 0\,,
\end{equation}
 with $\mu_t$ is the probability measure defined according to \eqref{int-dfinet}  and vice versa.    Here, $v$ is the vector field that defines the NLS equation \eqref{hartree}  (i.e., $v(u)=
 -\Delta u+ (V*|u|^2)u$)  and $\nabla^{T}$ denotes the transpose operation of the real gradient.
 The writing in \eqref{int-liouville} is formal and it is inspired by the finite dimensional form of
 transport or continuity equations;  but  it should be understood in a weak sense as
 \begin{equation}
\label{int.eq.transport}
      \displaystyle \int_{\R}\int_{L^2(\R^d)}\partial_{t}\varphi(t,u)+{\rm Re}\langle
  v(u),\nabla\varphi(t,u)\rangle \; d\mu_{t}(u)\,dt=0,
\end{equation}
 for all functions $ \varphi:\R\times L^2(\R^d)\to \R $ in a  certain class of  cylindrical smooth test functions
  $\mathscr C_{0,cyl}^{\infty}$  that will be detailed in Subsect.~\ref{fram}.   Thus, by establishing the above duality or equivalence  between hierarchies and transport equations we enter the realm of kinetic theory where powerful ideas have been flourishing for a while.
Our second observation is that the uniqueness problem for  \eqref{int.eq.transport} can be solved by a general
and powerful argument that is model independent with respect to the initial value problem \eqref{hartree}.  The idea
of this argument is in some sense related to the  well-known method of characteristics in kinetic theory and  the related recent advances in the case of non-smooth vector fields  \cite{MR2759545,MR1022305,MR2400257,MR2439520,MR2668627,MR2129498,MR2839299,MR2335089}.   Indeed, inspired by these methods
the authors in \cite{MR3721874,MR3379490} developed   a uniqueness theory for continuity equations defined over
arbitrary rigged Hilbert spaces. A further improvement of  these  results is needed  in this article.
So, thanks to the above duality and transport techniques one solves the problem of  uniqueness for hierarchies without appealing to Feynman graph expansion, nor board game argument nor any multilinear estimates of any kind.  The only thing that counts is that the initial value problem of type \eqref{hartree}  satisfies the uniqueness property for its weak solutions on some natural functional spaces.   This  reduces the problem (iii)  to the investigation of an initial value problem like NLS or Hartree and usually there is an  abundant literature on these questions for various nonlinear  PDEs. We believe
that the approach we suggest here is quite natural. In fact, one can argue that the mean field limit of quantum states \eqref{qstate} are probability distributions $\mu_t$ satisfying a continuity (Liouville)  equation because of conservation laws. On the other side, the $k$-point correlation functions \eqref{corr-funct} of the quantum system  should  converge at least formally  towards  the classical correlation functions of the same  probability distribution $\mu_t$.   Hence, the Liouville and hierarchy equations provide a dynamical statistical descriptions of the same classical system and so their are equivalent. Furthermore, working with Liouville equation is more advantageous since we do not need to control all the moments with respect to the probability distribution $\mu_t$ and this explains in some sense why we get rid of the combinatorial  and the nonlinear problems that surround the hierarchy equations.   We will describe  our main results in more details in the following subsections.

\bigskip
\noindent
%contribution:
In a broader perspective,  our main purpose in this article is the study of the relationships between an initial value problem like the NLS and Hartree equations and two distinguished descriptions of its statistical dynamics. Indeed, to adopt a statistical point of view starting from a Cauchy problem, there are at least two distinguished visions. The first consists on writing a Liouville equation, as in finite dimension, while the second consists on writing a hierarchy equation. The letter point of view requires that the initial value problem is invariant with respect to the gauge group $U(1)$  while the first is more natural and follows the original spirit of statistical mechanics. Our main contribution here is essentially the clarification of how the uniqueness and well posedness properties of each the above formulations  are interconnected to the others in full generality.
More precisely, the  main results of this article  can be summarized as follows:
\begin{itemize}
\item  {\it Duality}:  We prove  in full generality a natural duality  between hierarchy equations of type \eqref{GPH-cubic} and Liouville equations of the from \eqref{int-liouville}.
\item {\it Uniqueness and existence principle}:   We  establish  a general principle  saying:
  \begin{itemize}
\item [(i)]  Uniqueness for a hierarchy equation holds whenever the related initial value problem satisfies
a uniqueness property  for its weak solutions.
\item [(ii)]   Existence of solution to an $U(1)$-invariant  initial value problem  implies the existence of solutions
 for the related  hierarchy  equation of type  \eqref{GPH-cubic}.
\end{itemize}
\item {\it Applications}: We provide in Subsection \ref{sub.sec.highresult} several examples focused around the NLS and Hartree equations.  In particular, our work lifts straightforwardly to hierarchy equations \eqref{GPH-cubic} or \eqref{int-hier} the  landmark results of \cite{MR1383498,MR2474179,MR2361505,MR3056755,MR1992354}  on \emph{unconditional}  uniqueness  for  NLS equations.   Furthermore, we  formulate \emph{conditional} uniqueness results
for  solutions of hierarchy equations in the critical and subcritical cases. We also provide a counter-example showing  that uniqueness fails for a critical hierarchy equation  if a conditional  assumption is not assumed.
\item {\it Transport techniques}: We show that transport techniques are very powerful tools in solving the questions of uniqueness and well-posedness for general hierarchy equations.  In particular, we staidly push forward in this article
the attempt to build a unified  statistical theory of nonlinear PDEs which  sprang up in \cite{MR3379490}  and continued in \cite{MR3721874}.  We believe that such transport  techniques are very useful and would have fruitful applications  in various fields as hydrodynamics, nonlinear dispersive PDEs, integrable systems and quantum field theory (see e.g. \cite{MR3721874,MR3737034,MR3255099}).
\end{itemize}

\subsection{Preliminaries}
\label{fram}
We introduce below some notations and the general framework  that we shall follow throughout the article. This is particularly useful to  state clearly our main results in the next subsection.

\bigskip
\noindent
\emph{Notations}:  Let $\mathfrak{H}$ be a separable  Hilbert space. We use   $\mathscr{L}^k(\mathfrak{H})$, $1\leq k\leq \infty$, to denote the Schatten classes. In particular,  $\mathscr{L}^1(\mathfrak{H})$ and $\mathscr{L}^\infty(\mathfrak{H})$ are   the spaces of trace class and compact operators respectively. Two natural topologies
over $\mathscr{L}^1(\mathfrak{H})$  will be often used, namely the norm topology $||\cdot||_{ \mathscr{L}^1(\mathfrak{H})}$ and the weak-$*$ topology. The latter is inherited from the duality  $  \mathscr{L}^1(\mathfrak{H})= \mathscr{L}^\infty(\mathfrak{H})^*$  and leads to the following sequential convergence:
$$
B_n\overset{ *}{ \rightharpoonup} B  \quad\hbox{  if and only if } \quad \lim_{n\to\infty}\Tr[B_n \,K]=\Tr[B \, K], \quad \hbox{
for any } \; K\in  \mathscr{L}^\infty(\mathfrak{H}).
$$

\medskip
%\noindent
% probability measures
Let $X_r=B_{\mathfrak{H}}(0,r)$ denotes the closed ball  of radius $r>0$ in $\mathfrak{H}$.
The spaces of Borel probability measures on $\mathfrak{H}$ or on $X_r$ will be denoted by  $\mathfrak{P}(\mathfrak{H})$ and $\mathfrak{P}(X_r)$ respectively. Obviously,
a measure $\mu\in\mathfrak{P}(X_r)$ is  also a measure in $\mathfrak{P}(\mathfrak{H})$ which  concentrates on $X_r$, i.e.~$\mu(X_r)=1$. Consider
$$
U(1):=\{e^{i\theta}, \theta\in\R\}
$$
  to be the circle group. We say that a measure
$\mu\in \mathfrak{P}(\mathfrak{H})$ is $U(1)$-invariant if for any  bounded Borel function $\varphi:\mathfrak{H}\to \R$ and for any $\theta\in\R$,
$$
\int_{\mathfrak{H}} \varphi(e^{i\theta} x) \;d\mu(x) =
\int_{\mathfrak{H}} \varphi( x) \;d\mu(x)\,.
$$
The following set of probability measures on $\mathfrak{H}$ will be important  latter on,
\begin{equation}
\label{invmea}
\mathscr{M}(\mathfrak{H}):= \{\mu\in \mathfrak{P}(X_1): \mu \textit{ is  }  U(1)\textit{-invariant}\}\,.
\end{equation}
There is two  natural topologies on   $\mathfrak{P}(\mathfrak{H})$  that we shall systematically use, namely the strong and weak narrow convergence topologies.   We say that a sequence $(\mu_i)_{i\in\N}$ in $\mathfrak{P}(\mathfrak{H})$ converges weakly narrowly to  a   $\mu\in \mathfrak{P}(\mathfrak{H})$ if:
\begin{eqnarray}
\label{defnarroww}
\mu_i\rightharpoonup \mu & \Longleftrightarrow &\big(\int_{\mathfrak{H}} f \,d\mu_i \to \int_{\mathfrak{H}} f \,d\mu, \quad \forall f\in \mathscr{C}_b(\mathfrak{H}_w)\,\big)\,,
\end{eqnarray}
and strongly narrowly if:
\begin{eqnarray}
\label{defnarrows}
 \mu_i\to\mu &\Longleftrightarrow &\big( \int_{\mathfrak{H}} f \,d\mu_i \to \int_{\mathfrak{H}} f \,d\mu, \quad\forall f\in \mathscr{C}_b(\mathfrak{H}_s) \,\big)\,,
\end{eqnarray}
where  $\mathscr{C}_b(\mathfrak{H}_w)$ and $\mathscr{C}_b(\mathfrak{H}_s)$ are  the spaces of bounded continuous  functions with respect to the weak and  norm topology of  the Hilbert space $\mathfrak{H}$ respectively. Recall that the weak topology in $\mathfrak{H}$ is metrizable  on bounded sets. This can be seen using for instance the following distance,
\begin{equation}
\label{dweak}
d_w(x,y):=\sqrt{\sum_{n\in\N}
\frac{1}{2^n}  \,|\langle x-y, e_n\rangle|^2}\,,
\end{equation}
where $(e_n)_{n\in\N}$ is an O.N.B of $\mathfrak{H}$. In the same way, we define the weak and strong
narrow convergence topology in $\mathfrak{P}(X_r)$ as the limits in \eqref{defnarroww}-\eqref{defnarrows} with the functional spaces in the right hand side replaced by  $\mathscr{C}_b(X_r,d_w)$ and $\mathscr{C}_b(X_r,||\cdot||_{\mathfrak{H}})$ respectively.

\bigskip
\noindent
\emph{Rigged Hilbert spaces}:
Consider  a {separable} Hilbert space $\zed$  and   a self-adjoint operator $A$ with a domain $D(A) \subset \zed$ verifying :
\begin{equation}
\label{condopA}
\exists c > 0 ~,~ A \geq c \, 1\,.
\end{equation}
Using the operator $A$ one can build a natural  scale of Hilbert spaces indexed by a parameter $\tau \in \R$. Indeed, consider for every $\tau\in\R$ the inner products :
\[
\forall x,y \in D(A^{\frac{\tau}{2}}) ~,\qquad \langle x,y \rangle_{\mathscr{Z}_{\tau}} := \langle A^{\tau/2} x,A^{\tau/2}y \rangle_{\zed}\,,
\]
and let $\mathscr{Z}_{\tau}$ denotes the completion of the pre-Hilbert space $(D(A^{\frac{\tau}{2}}), \langle \cdot,\cdot\rangle_{\mathscr{Z}_{\tau}})$. Then for any  $s,\sigma\geq 0$, one has the
 canonical continuous and dense embeddings,
$$
\mathscr{Z}_{s} \hookrightarrow \mathscr{Z}_{0} \hookrightarrow \mathscr{Z}_{-\sigma}\,.
$$
Remark that $\mathscr{Z}_{-\sigma}$ identifies with the dual space $\mathscr{Z}_{\sigma}^{'}$ of $\mathscr{Z}_{\sigma}$ with respect to the inner product of $\zed$.
A simple example is provided for instance by the Sobolev spaces with $\zeds = H^{s}(\mathbb{R}^{d})$ and
$\mathscr{Z}_{-s}= H^{-s}(\mathbb{R}^{d})$ for $s \geq 0$.

\bigskip
\noindent
\emph{Initial value problem}:
Consider  a (possibly)  non-autonomous vector field   ${v} : \mathbb{R} \times \zeds \to \mathscr{Z}_{-\sigma}$,  with $0 \leq s \leq \sigma$. Here the spaces $ \zeds$ and  $\mathscr{Z}_{-\sigma}$ are defined according to the above paragraph.   Then the initial value problem defined by the vector field $v$, is the following  differential equation valued in  $\mathscr{Z}_{-\sigma}$  and defined over an open time bounded  interval $I$ as,
\begin{equation}
\label{IVP}\tag{{\it ivp}}
  \left\{
    \begin{aligned}
    &\dot{u}(t) \ = {v}(t,u(t))\,,& \\
    &u(t_{0}) \ = x \in \zeds\,.& \\
    \end{aligned}
  \right.
\end{equation}
Here $t_0\in I$ is a fixed  initial time. We shall require the following assumption on the vector field,
\begin{equation}
\label{A0}\tag{A0}
\left\{
\begin{aligned}
&v  \text{ is }  \text{ Borel  and  bounded on bounded  sets },\\
&v  \text{ is } U(1)-\text{invariant }\,.
\end{aligned}
\right.
\end{equation}
The $U(1)$-invariance means that  $v(t,e^{i\theta} x)=e^{i\theta} v(t, x)$ for all $\theta\in \R$ and $(t,x)\in \R\times \zeds$. There are at least two distinct notions of solutions for the above initial value problem.

\begin{defn}
\label{w-ssol}
Consider the initial value problem  \eqref{IVP} with a vector field satisfying \eqref{A0}. Then:
\begin{itemize}
\item[(i)] A weak solution of \eqref{IVP} over $I$ is a function $u: t \in I \longrightarrow u(t)$ belonging to the space $L^{\infty}(I,\zeds) \cap W^{1,\infty}(I,\zedsi)$ satisfying \eqref{IVP} for a.e.~$t \in I$ and for some $t_{0} \in I$.
\item[(ii)] A strong solution of \eqref{IVP} over $I$ is a function $u: t \in I \longrightarrow u(t)$ belonging to the space $\mathscr C(I,\zeds) \cap \mathscr C^{1}(I,\zedsi)$ satisfying \eqref{IVP} for all $t \in I$ and for some $t_{0} \in I$.
\end{itemize}
\end{defn}
Here $\mathscr C(I,\mathfrak{H})$ and $\mathscr C^1(I,\mathfrak{H})$ are respectively the spaces of
continuous and $\mathscr{C}^1$-functions valued in a given Hilbert space $\mathfrak{H}$. While
$W^{1,p}(I,\zedsi)$, $1\leq p\leq\infty$, are   the Sobolev spaces of classes of functions in $L^p(I,\zedsi)$ with  distributional first derivatives in  $L^p(I,\zedsi)$. Recall that any $u\in W^{1,p}(I,\zedsi)$ is an absolutely continuous curve in $\zedsi$  with almost everywhere defined derivatives in $\zedsi$ satisfying $\dot u\in L^p(I,\zedsi)$. The initial value problem \eqref{IVP} makes sense in the space $L^\infty(I,\zeds)\cap W^{1,\infty}(I,\zedsi)$ since weak solutions of \eqref{IVP} are weakly continuous maps  $u:\bar I\to\mathscr Z_s$ which are differentiable almost everywhere on $I$. Furthermore, it is easy to check using the assumption \eqref{A0} that any function $u\in L^\infty(I,\zeds)$ satisfying the Duhamel formula,
\begin{equation}
\label{int-form}
u(t)=x+\int_{t_0}^t v(\tau,u(\tau)) ds\,, \text{ for  a.e. } t\in I\,,
\end{equation}
is a weak solution of \eqref{IVP}. Conversely, any weak solution $u$ of \eqref{IVP} satisfies \eqref{int-form}.   Similarly, if we assume that the vector field $v$ is continuous then strong solutions of \eqref{IVP} over $I$ are exactly continuous curves   in $\mathscr{C}(I,\zeds)$ satisfying the Duhamel formula \eqref{int-form} for all $t\in I$
  (see e.g.~\cite{MR2002047} for more details).

\bigskip
\noindent
\emph{Hierarchy equations}:
Consider a normal state or a density matrix $\varrho_n\in \mathscr{L}^1(\vee^n\zed)$
(i.e., $\varrho_n\geq 0$ and $ \Tr[\varrho_n]=1$). Then the $k$-particles {reduced density matrix} of  $\varrho_n$, for $0\leq k\leq n$, is by definition the unique non-negative trace-class operator denoted by $\varrho_n^{(k)}\in  \mathscr{L}^1(\vee^k\zed) $ and satisfying:
\begin{equation}
\label{redmat}
\Tr_{\otimes^n \zed} \big[\varrho_n  \, A\otimes 1^{(n-k)}\big]=\Tr_{\vee^k \zed}
\big[\varrho_n^{(k)}  \, A\big]\,, \quad \forall A \in  \mathscr{L}(\vee^k\zed)\,.
\end{equation}
 We call a \emph{symmetric hierarchy} (or simply a hierarchy) any subsequential limit of  $(\gamma_n^{(k)})_{k\in \N}$, with respect to the weak-$*$ topology in $
 \mathscr{L}^1(\vee^k\zed)$ . Moreover, we denote the  set of all these subsequential limits  by
 $\mathscr{H}(\zed)$.  In Section \ref{sec:01}, we prove  that   $\mathscr{H}(\zed)$
 is  a non trivial convex set admitting the following characterization:
\begin{equation}
\label{HKchar}
\mathscr{H}(\zed)=\left\{
\int_{X_1} |\varphi^{\otimes k} \rangle \langle \varphi^{\otimes k} | \;d\mu(\varphi)\,, \;\mu\in \mathscr{M}(\zed)\right\}\,,
\end{equation}
where $\mathscr{M}(\zed)$ is the set, given by \eqref{invmea}  with $\mathfrak{H}=\zed$,  and  constituted of probability measures which are $U(1)$-invariant and concentrated on the closed unit ball $ X_1:=B_{\zed}(0, 1)$ of $\zed$. In fact, for any $(\gamma^{(k)})_{k\in \N}\in \mathscr{H}(\zed)$ there exists a unique Borel probability measure $\mu$ on  $ X_1$  such that it  is  $U(1)$-invariant and satisfying for any $k\in\N$,
\begin{equation}
\label{pre-dfinet}
\gamma^{(k)}=\int_{X_1} |\varphi^{\otimes k} \rangle \langle \varphi^{\otimes k} | \;d\mu(\varphi)\,.
\end{equation}
Conversely, for any $\mu\in \mathscr{M}(\zed)$ the above expression defines a symmetric hierarchy (see Prop.~\ref{str.BEh}  and \ref{bij}).

In the following, we want to write a hierarchy equation that generalizes  the ones in \eqref{GPH-cubic} to any
nonlinearity or equivalently to any initial value problem \eqref{IVP} which is $U(1)$-invariant.
For that purpose, we introduce two operations  extending the actions  $B^\pm_{j,k}$
 given in  \eqref{Bjk-1}-\eqref{Bjk-2} to any initial value problem \eqref{IVP} satisfying \eqref{A0}.  Indeed, we  define for any  $k\in\N$  and   $j=1,\cdots,k$,  the following operations  on  $\gamma \in \mathscr{H}(\zed)$ as,
\begin{equation}
\begin{aligned}
\bullet\; C_{j,k}^{+} \gamma^{} &:= \displaystyle\int_{\zeds\cap X_1} \big| x^{\otimes k} \rangle \langle x^{\otimes j-1} \otimes v(t,x)\otimes x ^{\otimes k-j}\big| \;d\mu_{}(x)  \,,\\ \medskip
\bullet  \;C_{j,k}^{-} \gamma_{}^{} &:=  \displaystyle\int_{\zeds\cap X_1} \big| x^{\otimes j-1} \otimes v(t,x)\otimes x ^{\otimes k-j} \rangle \langle x^{\otimes k} \big| \;d\mu_{}(x) \,,
 \end{aligned}
\end{equation}
where $\gamma$ and $\mu$ are related according to the integral representation  \eqref{pre-dfinet}.

We call a \emph{(symmetric) hierarchy equation} related to the initial value problem \eqref{IVP} the following integral equation defined on a open bounded interval $I$,
\begin{equation}
\label{int-hier}
\forall t \in I ~,\qquad \gamma_{t}^{(k)} = \gamma_{t_{0}}^{(k)} + \int_{t_{0}}^{t} \sum_{j=1}^{k}  (C_{j,k}^{+} \gamma_{\tau}^{} +  C_{j,k}^{-} \gamma_{\tau}^{}) \;d\tau ~,~
\end{equation}
with an initial datum  $\gamma_{t_{0}}\in\mathscr{H}(\mathscr{Z}_0)$ for some $t_0\in I$. Of course, the latter
equation may not make sense and as usual some regularity is required on the solutions.
So, we will be interested only on solutions which are curves of symmetric hierarchies $t\in I\to \gamma_t\in \mathscr{H}(\mathscr{Z}_0)$  satisfying the following regularity assumption:
\begin{equation}
\left\{
\label{A2}\tag{A1}
\begin{array}{rl}
\bullet & \gamma_t\in\mathscr{H}(\mathscr{Z}_0)  \text{ for all } t\in I\,,\\[8pt]
\bullet & \exists R>0, \; \displaystyle ||(A^{s/2})^{\otimes k}\gamma_t^{(k)} (A^{s/2})^{\otimes k} ||_{ \mathscr{L}^1(\vee^k\mathscr{Z}_0)}\leq R^{2k}, \; \forall k\in\N, \forall t\in I\,,\\[8pt]
\bullet & \forall k\in\N, \; t\in I\to (A^{-\sigma/2})^{\otimes k}\gamma_t^{(k)}   (A^{-\sigma/2})^{\otimes k}  \;\text{ is weak-$*$ continuous in } \mathscr{L}_{}^1(\vee^k\mathscr{Z}_0)\,.
\end{array}
\right.
\end{equation}
Such assumption is actually quite natural since the first condition emerges from well justified physical constraints, while the second puts the regularity of  $\gamma_t$ at the same level as the one for the related initial value problem \eqref{IVP}  thus justifying the operations $C^\pm_{j,k}$; and the last is a very mild condition that ensures a rigorous meaning to the right hand side of \eqref{int-hier} as a Bochner integral in some Banach spaces (see  Section \ref{sec.equiv} for more details).

\bigskip
\noindent
\emph{Liouville equations}:
The Liouville equation, related to the initial value problem \eqref{IVP}, is given as in finite dimension by the formal expression,
\[
\partial_{t} \mu_{t} + \nabla^{T} ({v}\mu_{t}) = 0\,.
\]
However, since we are in infinite dimensional spaces we shall understand the above equation in a weak sense using a convenient   space of  cylindrical test functions defined below.   In particular,
we  shall use the   real structure of the Hilbert space $\mathscr Z_{-\sigma}$ and interpret $\nabla$ as a  real gradient and the subscript $^T$ as a transpose operation with respect to the real scalar product
of $ \mathscr Z_{-\sigma}$.  We refer the reader to \cite{MR3721874} for more details. %dire un peu plus...

\bigskip

Consider $\mathscr{Z}_{-\sigma}$ as a real Hilbert space $ \mathscr{Z}_{-\sigma,\R}$ endowed with the scalar product ${\rm Re}\langle \cdot, \cdot \rangle_{\mathscr{Z}_{-\sigma}}$ simply denoted by $\langle \cdot, \cdot \rangle_{\mathscr Z_{-\sigma,\R}}$.  Let $\mathbb{P}_n$ be the set of all projections $\pi:\mathscr Z_{_{-\sigma,\R}} \to  \R^n$  given by
\begin{equation}
\label{eq.pi}
\pi(x)=(\langle x, e_1\rangle_{\mathscr Z_{-\sigma},\R}, \cdots, \langle x, e_n\rangle_{\mathscr Z_{-\sigma},\R})\,,
\end{equation}
where $\{e_1,\cdots,e_n\}$ is any orthonormal family of $\mathscr Z_{-\sigma, \R}$.
We denote by
$\mathscr{C}_{0,cyl}^{\infty}(\mathscr Z_{-\sigma})$ the space of  functions  $\varphi=\psi\circ \pi$ with $\pi\in \mathbb{P}_n$ for some $n\in\N$ and $\psi\in \mathscr C_0^\infty(\R^n)$. In particular, one  checks that the gradient (or the $\R$-differential) of $\varphi$ is equal
to
$$
\nabla_{\mathscr Z_{-\sigma,\mathbb{R}}}\varphi=\pi^T\circ\nabla\psi\circ\pi,
$$
where $\pi^T:\R^n\to \mathscr{Z}_{-\sigma,\R} $ denotes the transpose map of $\pi$ given by
$$
\pi^T(x_1,\cdots,x_n)=\sum_{i=1}^n x_i \,e_i\,.
$$
Let $I$ be a bounded open interval, then we say that a function $\varphi : I \times \mathscr{Z}_{-\sigma}  \to \mathbb{R}$ belongs to $\mathscr C_{0,cyl}^{\infty}(I \times \mathscr{Z}_{-\sigma})$ if there exists, for some $n\in\N$, a projection $\pi \in \mathbb{P}_n$ and $\phi \in \mathscr C_{0}^{\infty}(I \times \R^n)$ such that:
\[
\forall (t,z) \in I \times\mathscr{Z}_{-\sigma} ~,\quad \varphi(t,z) = \phi(t,\pi(z)) \,.
\]
In particular, the functions of the form $\chi(.)\phi(.)$ where $\chi \in \mathscr C_{0}^{\infty}(I)$ and $\phi \in \mathscr C_{0,cyl}^{\infty}(\mathscr{Z}_{-\sigma})$ are in the space of cylindrical test functions $\mathscr C_{0,cyl}^{\infty}(I \times \mathscr{Z}_{-\sigma})$.

\bigskip

We consider  the  Liouville  equation, related to the initial value problem \eqref{IVP} and
 defined on a bounded open interval $I$, as the following integral equation,
\begin{equation}
\label{eq.transport}
      \displaystyle\int_{I}\int_{\mathscr Z_{-\sigma}}\partial_{t}\varphi(t,x)+\langle
  v(t,x),{\nabla_{\mathscr Z_{-\sigma,\mathbb{R}}}}\varphi(t,x)\rangle_{{\mathscr Z_{-\sigma,\mathbb{R}}}} \; d\mu_{t}(x)\,dt=0, \quad\forall
    \varphi \in \mathscr{C}_{0,cyl}^{\infty}(I \times \mathscr Z_{-\sigma})\,,
\end{equation}
where $\langle \cdot,\cdot\rangle_{\mathscr Z_{-\sigma,\mathbb{R}}}={\rm Re}
\langle \cdot,\cdot\rangle_{\mathscr Z_{-\sigma}}$, $\nabla_{\mathscr Z_{-\sigma,\mathbb{R}}}$ is the $\mathbb{R}$-gradient (or differential) in $\mathscr Z_{-\sigma,\R}$ and $t\in I\to \mu_t$ is a curve in the space of Borel probability measures $\mathfrak{P}(\mathscr{Z}_{-\sigma})$. The equation \eqref{eq.transport} is always supplemented by a prescribed initial condition $\mu_{t_0}\in \mathfrak{P}(\mathscr{Z}_{-\sigma})$ at a given time $t_0$.

We are interested on solutions of the Liouville equation \eqref{eq.transport} which satisfy the following assumption:
\begin{equation}
\left\{
\label{A1}\tag{A2}
\begin{array}{rl}
\bullet & \mu_t\in\mathscr{M}(\mathscr{Z}_0),  \text{ for all } t\in I.\\[8pt]
\bullet & \mu_t(B_{\mathscr{Z}_s}(0,R)) =1,  \text{ for some } R>0 \text{ and for all } t\in I.\\[8pt]
\bullet & t\in I\to\mu_t \text{ is weakly narrowly continuous in } \mathfrak{P}(\mathscr{Z}_{-\sigma}).
\end{array}
\right.
\end{equation}
Remark that Borel sets of $\mathscr Z_s$ are also  Borel set of $\mathscr Z_{-\sigma}$, see for instance \cite[Appendix]{MR3721874}. So, the assumption \eqref{A1} implies that the integral  with respect to $\mu_t$ in \eqref{eq.transport} is actually  over the closed ball $B_{\mathscr{Z}_s}(0,R)$. Furthermore,  the integrand \eqref{eq.transport} is bounded and the integration with respect to $\mu_t$ and $dt$ is well defined. The first and last requirements in \eqref{A1}  are quite natural as
they are justified by physical constraints and a mild time regularity. However, to make the Liouville equation rigourously make sense it is not necessary to assume the concentration condition $\mu_t(B_{\mathscr{Z}_s}(0,R)) =1$ which can be relaxed. This confers already a serious advantage to the Liouville equation \eqref{eq.transport} over the symmetric hierarchy equation \eqref{int-hier}.

\begin{remark}
\label{balls}
In the assumptions \eqref{A2} and \eqref{A1}, we explicitly asked respectively  for the following bounds to hold for all
$t\in I$,
$$
(\forall k\in\N ,\quad \Tr_{\vee^k \zed}[\gamma_t^{(k)}]\leq 1)  \quad \text{ and } \quad \mu_t(B_{\zed}(0, 1))=1\,.
$$
Such requirements are only made  to meet the important physical constraint \eqref{cont2} related to the  mean-field theory of Bose gases. From a pure mathematical point of view, one can relax these conditions by simply  assuming  for some $R'>0$,
\begin{equation}
\label{bdrk}
(\forall k\in\N ,\quad \Tr_{\vee^k \zed}[\gamma_t^{(k)}]\leq R'^k)  \quad \text{ and } \quad \mu_t(B_{\zed}(0, R'))=1\,.
\end{equation}
In this case, the second condition  in \eqref{A2} and respectively in  \eqref{A1} imply the  bounds \eqref{bdrk} with
$R'=\frac{R}{c}$ (where $c$ is the constant in the inequality \eqref{condopA}).
Our main results in Subsect.~\ref{sub.sec.highresult}  still hold true under this small modification. And this  explains why we do not require that the norm $||\cdot||_{\zed}$ is a conserved quantity for the initial value problem \eqref{IVP} since our results are also valid for non-conservative dynamical systems.
\end{remark}

\subsection{Highlighted results}
\label{sub.sec.highresult}
In this Subsection, we present our main contributions previously  discussed in the introduction; namely
duality between Liouville and hierarchy equations (Theorem \ref{sec.0.thm1}) and  uniqueness and existence principles (Theorem \ref{sec.0.thm2} and \ref{ext-2}). These results will be stated in full generality. Moreover, some appealing applications to the NLS and Hartree equations will  be discussed below.

\begin{thm}(Duality)
\label{sec.0.thm1}
Let ${v}: \mathbb{R} \times \zeds \mapsto \zedsi$ be a vector field satisfying \eqref{A0} and
$I$ a bounded open interval. Then  $t\in I\to\gamma_t=(\gamma_t^{(k)})_{k\in\N}$ is a solution of the symmetric hierarchy equation \eqref{int-hier} satisfying \eqref{A2} if and only if  $t\in I\to\mu_t$ is a solution of the Liouville equation \eqref{eq.transport}  satisfying \eqref{A1} with $\mu_t$ is related to
$\gamma_t$ according to
\begin{equation*}
\gamma_t^{(k)}=\int_{\zeds} |\varphi^{\otimes k} \rangle \langle \varphi^{\otimes k} | \;d\mu_t(\varphi)\,, \quad \forall k\in\N.
\end{equation*}
\end{thm}

Recall that the notion of weak solutions of the \eqref{IVP} is given in Definition \ref{w-ssol}.
\begin{thm}(Uniqueness principle)
\label{sec.0.thm2}
Let ${v}: \mathbb{R} \times \zeds \mapsto \zedsi$ be a vector field satisfying \eqref{A0}.  Then  uniqueness of weak solutions over a bounded open interval $I$ for the initial value problem
\eqref{IVP} implies the uniqueness of solutions over $I$ of the symmetric hierarchy equation \eqref{int-hier}
 satisfying the assumption \eqref{A2}.
\end{thm}

\begin{thm}(Existence principle)
\label{ext-2}
Let $v:\R\times \mathscr Z_s\to \mathscr Z_{-\sigma}$ a vector field satisfying \eqref{A0}  and  let $I$ be  a bounded open interval with $t_0\in I$ a fixed initial time. Assume that there exist a Borel
subset $\mathcal{A}$ of $\zeds$ and a Borel  map $\phi:\bar I\times \mathcal{A}\to \zeds$ which is bounded on bounded sets
and such that for any $x\in \mathcal{A}$ the curve $t\in  I \to \phi(t,x)$ is a weak solution of the initial value problem \eqref{IVP} satisfying $ \phi(t_0,x)=x$.  Furthermore, suppose that $||\phi(t,x)||_{\zed}=
||x||_{\zed}$ and $\phi(t,e^{i \theta} x)= e^{i \theta}
\phi(t, x)$ for any $x\in \mathcal{A}$, $t\in \bar I$ and $\theta\in\R$. Then for any  symmetric hierarchy $\gamma=(\gamma^{(k)})_{k\in\N}\in\mathscr{H}(\zed)$ satisfying:
\begin{eqnarray*}
&&\forall k\in\N, \quad  \gamma^{(k)}=\int_{\zeds} |\varphi^{\otimes k} \rangle \langle \varphi^{\otimes k} | \;d\nu(\varphi)\,, \quad \text{ with } \quad \nu(\mathcal{A})=1 , \\
&&\exists R>0, \;\; \displaystyle ||(A^{s/2})^{\otimes k}\gamma^{(k)} (A^{s/2})^{\otimes k} ||_{ \mathscr{L}^1(\vee^k\mathscr{Z}_0)}\leq R^{2k}, \; \forall k\in\N,
\end{eqnarray*}
there exists a solution $t\in I\to \gamma_t=(\gamma_t^{(k)})_{k\in\N}\in\mathscr{H}(\zed)$  of the hierarchy equation \eqref{int-hier}   verifying  the initial condition $\gamma_{t_0}=\gamma$ and the assumption   \eqref{A2}.
\end{thm}

It is worth noticing that there is a converse for Thm.~\ref{sec.0.thm2} and \ref{ext-2}. Actually, the converse to
Thm.~\ref{sec.0.thm2}  is easy to establish and says that the uniqueness of hierarchy solutions satisfying \eqref{A2} implies uniqueness for the weak solutions of the  initial value problem \eqref{IVP} modulo $U(1)$-gauge invariance. While, the converse of Thm.~\ref{ext-2} is more involved and will be considered elsewhere in connection with other applications.
\subsection{Examples}
\label{subsec.app}

Consider the following nonlinear Schr\"odinger  (NLS)  equation on $\R^d$,
\begin{equation}
\label{eq.ivp}
\left\{
\begin{aligned}
&i \partial_{t}u=-\Delta u+g(u)\\
&u_{|t=0}=u_{0}\,,
\end{aligned}
\right.
\end{equation}
where  $g:\C\to \C$ is a function  satisfying the assumptions
\begin{equation}
\left\{
\label{cond-nls}
\begin{aligned}
& \quad g\in\mathscr{C}^1(\R^2,\R^2), \;g(0)=0, \\
& \quad \exists \alpha\geq 0, \; \exists M>0, \;|g'(\xi)|\leq M |\xi|^\alpha, \quad  \forall \,|\xi|\geq 1,\\
& \quad g(e^{i\theta} x)= e^{i\theta} g( x),  \quad  \forall\, \theta\in\R\,, \forall\, x\in\C\,.
\end{aligned}
\right.
\end{equation}
Here $g'$ denotes the real derivative  of $g$ considered as a function from  $\R^2$ into itself. So that,
$g'(\xi)$  can be identified with $(\partial_{z}g(\xi), \partial_{\bar z} g(\xi))$ and  %$g'(\xi)\in \mathscr{M}_2(\R)$
$$
|g'(\xi)|=|\partial_{z}g(\xi)|+|\partial_{\bar z}g(\xi)|\,.
$$
One can  easily see that  the two first assertions in \eqref{cond-nls}  are equivalent to assuming
that the function $g$ splits into the sum of two function $g=g_1+g_2$ such that $ g_i\in\mathscr{C}^1(\R^2,\R^2)$, $g_i(0)=0$,  for  $i=1,2$ and satisfying additionally,
\begin{eqnarray*}
  |g_1(\xi)-g_1(\eta)| \leq  M \;|\xi-\eta| \,, \qquad \text{ and } \qquad
  |g_2(\xi)-g_2(\eta)| \leq  M \;|\xi-\eta| \max(|\xi|^\alpha,|\eta|^\alpha)\,,
\end{eqnarray*}
for any  $\xi,\eta\in\C$ (see e.g.~\cite{MR877998}).  Notice also that using Sobolev's inequalities, if for $0\leq s<
\frac d 2$ the power $\alpha$ in \eqref{cond-nls} verifies
\begin{equation}
\label{alpha}
0\leq \alpha\leq \frac{d+2s}{d-2s}\,,
\end{equation}
then there exists $\sigma\geq 0$ such that  the nonlinearity in the NLS equation \eqref{eq.ivp},
\begin{equation}
\label{nonlin}
\begin{aligned}
G: H^s(\R^d)&\longrightarrow & H^{-\sigma}(\R^d)\\
u&\longrightarrow & g(u)
\end{aligned}
\end{equation}
is continuous and bounded on bounded sets.  Moreover,  for any $s\geq \frac{d}{2}$ and $\alpha\geq 0$
the map $G:H^s(\R^d)\to L^2(\R^d)$ is also continuous and bounded on bounded sets. This in particular implies that the initial value problem
\eqref{eq.ivp}  make sense in the spaces $H^s(\R^d)$.
A mild solution $u$  of the NLS equation \eqref{eq.ivp}, defined over a bounded open interval $0\in I$,  is a function $u\in L^\infty(I, H^s(\R^d))$ satisfying
for any $t\in I$,
\begin{equation}
\label{nls-int}
u(t)=\mathcal{U}(t)u(0)-i\int_0^t \mathcal{U}(t-\tau) g(u(\tau)) \; d\tau\,,
\end{equation}
where $\mathcal{U}(t)=e^{it \Delta}$.  Since the nonlinearity $G$ is bounded on bounded sets one concludes that
$u\in W^{1,\infty}(I,H^{-\sigma}(\R^d))$. Consequently, if one considers $\tilde u:=\mathcal{U}(-t)u$ then
$u\in L^\infty(I, H^s(\R^d))$ is  a mild solution of  \eqref{eq.ivp} satisfying \eqref{nls-int}  if and only if $\tilde u\in L^\infty(I, H^s(\R^d)) \cap W^{1,\infty}(I, H^{-\sigma}(\R^d))$ is a weak solution, in the sense of Def.~\ref{w-ssol}, of the initial value problem \eqref{IVP} with the vector field $v:\R\times H^s(\R^d)\to H^{-\sigma}(\R^d)$ defined by
\begin{equation}
\label{nls-v}
v(t,x):=-i\, \mathcal{U}(-t) \;g(\mathcal{U}(t) x)\,.
\end{equation}
So, one notices that   the vector field $v$ satisfies  the assumption \eqref{A0} whenever  $g$ verifies  \eqref{cond-nls}  and $\alpha$ is such that \eqref{alpha} holds true if $0\leq s<\frac{d}{2}$ or $\alpha\geq 0$   if $s\geq \frac{d}{2}$.  Hence, one can consider the related hierarchy equation \eqref{int-hier} with this vector field  $v$, given in \eqref{nls-v},  and apply Theorem \ref{sec.0.thm2} with
\begin{equation}
\label{spec-fram}
A=-\Delta+1, \qquad \zeds=H^s(\R^d), \qquad  \zedsi=H^{-\sigma}(\R^d) .
\end{equation}
 Since there is a multitude of results on unconditional uniqueness for  NLS and Hartree equations, so one can lift them straightforwardly  to  the corresponding  hierarchy equations  \eqref{int-hier}  using Thm.~\ref{sec.0.thm2}  (see e.g.~\cite{MR3447005,MR1383498,MR2361505,MR3056755,MR2474179,MR2002047,MR1992354,MR1055532,MR952091}).

We will not try to detail all the uniqueness  results that can be lifted through the uniqueness principle of Thm.~\ref{sec.0.thm2}, but instead we will illustrate it by a few remarkable examples given below.   Although Thm.~\ref{sec.0.thm2} seems only suitable for using   unconditional uniqueness results in spaces of type $L^\infty(I, H^s(\R^d))$  for  the initial value problem \eqref{IVP}, it is not difficult to see that its proof allows also to lift conditional uniqueness results for the  \eqref{IVP}  to the related hierarchy equations \eqref{int-hier}.  However, to avoid intricate statements  in this case, we refrain from giving an abstract conditional uniqueness result for hierarchy equations and prefer to treat some  significant examples below.

It seems that  in the literature the only studied hierarchy equations are related to nonlinearities  given by $g(u)=|u|^\alpha u$ with $\alpha=2,4$ or $g(u)=V*|u|^2 u$. So, here in particular we show how to define a general  symmetric  hierarchy equations for any consistent\footnote{We mean that $g$ defines a  Borel map $G$ as in  \eqref{nonlin} for some $s,\sigma\geq 0$ and it is bounded on bounded sets.}   nonlinearity which is $U(1)$-invariant.
Our formulation of the hierarchy equations in \eqref{int-hier}   is
equivalent to the usual writing in the special cases mentioned before. But there are two slight differences, which are worth highlighting: Firstly,
we prefer to  work in the interaction representation while in the literature integral equations are used and secondly we use trace-class operators while in the literature  kernel representation is preferred. In Subsection \ref{reg-issue}, the above equivalence between these two writing is explained.

\bigskip
\noindent
\emph{ Unconditional results}:
One of the remarkable results on unconditional uniqueness for  NLS equations \eqref{eq.ivp} is due to Kato in \cite{MR1383498}. It says that any two mild solutions
in $L^\infty(I, H^s(\R^d))$   of  \eqref{eq.ivp}  with the same initial condition coincide under some assumptions on the dimension $d\geq 1$, the power $\alpha\geq 0$ and the Sobolev scale $s\geq 0$.
Such a result has the following implication on the related hierarchies.

\begin{prop}[Kato]
\label{uniq-kato}
Let $g:\C\to\C$ a function satisfying \eqref{cond-nls}  and take $s\geq 0$.  Consider  $v:\R\times H^s(\R^d) \to H^{-\sigma}(\R^d)$  to be the vector field  given by \eqref{nls-v}. Then any two solutions  of the  hierarchy equation \eqref{int-hier}  satisfying \eqref{A2}, with the same initial condition, coincide in the following cases:
 \begin{itemize}
 \item [(i)] $s\geq \frac{d}{2}$.
 \item [(ii)] $d\geq 2$, $0\leq s <\frac{d}{2}$ and $\alpha< \min\{\frac{4}{d-2s},\frac{2s+2}{d-2s}\}$.
 \item [(iii)] $d=1$, $0\leq s<\frac{1}{2}$ and $ \alpha\leq \frac{1+2s}{1-2s}$.
 \end{itemize}
\end{prop}
There are several improvement (see e.g. \cite{MR2361505,MR3056755,MR2474179,MR2002047,MR1992354,MR1055532,MR952091}) of the unconditional uniqueness  result of Kato when the nonlinearity is specified as
\begin{equation}
\label{nonlin-p}
g(u)=\pm |u|^\alpha u\,.
\end{equation}
We summarize a straightforward consequence of some of these uniqueness  results on the hierarchy equations
 \eqref{int-hier}.  We respectively  refer to \cite[Theorem 1.5]{MR1992354}, \cite[Theorem 1.1]{MR2361505} and
\cite[Theorem 1.5]{MR3056755}.
\begin{prop}[]
\label{uniq-power}
Consider  $v:\R\times H^s(\R^d) \to H^{-\sigma}(\R^d)$  to be the vector field  defined in  \eqref{nls-v} with the nonlinearity given by \eqref{nonlin-p}. Then any two solutions  of the  corresponding hierarchy equation \eqref{int-hier}  satisfying \eqref{A2}, with the same initial condition coincide in the following cases:
 \begin{itemize}
 \item (Furioli $\&$ Terraneo):
 \begin{eqnarray*}
  && 3\leq d\leq 5, \quad 0<s<1, \quad  \alpha \leq \frac{d+2-2s}{d-2s},
\\\text{ and }  &&\max\{1, \frac{2s}{d-2s}\}<\alpha<\min\{\frac{4}{d-2s},\frac{ d+2s}{d-2s},\frac{ 4s+2}{d-2s}\}.
 \end{eqnarray*}
 \item (Rogers):
 \begin{equation*}
d\geq3,\quad 0\leq s\leq 1, \quad
\frac{2+2s}{d-2s}\leq\alpha<\min\{\frac{2+4s-\frac{4s}{d}}{d-2s},\
\frac{4}{d-2s}\},
\end{equation*}
 \item (Han $\&$ Fang):
 \begin{eqnarray*}
&&d=2,   \quad 0<s<1,   \quad \alpha=\frac{2+2s}{2-2s} ,\\
 \text{ or } && d=3, \quad  \frac{1}{4}<s<\frac{1}{2},  \quad \alpha=\frac{3+2s}{3-2s}.
 \end{eqnarray*}
 \end{itemize}
\end{prop}
In particular, for cubic nonlinearity $\alpha=2$, we recover uniqueness for the hierarchy equation \eqref{GPH-cubic}
proved in \cite{MR3395127} in the cases $d=1,2$ for any  $s\geq\frac{d}{6}$ and  $d\geq 3$ for any $s>\frac{d-2}{2}$.  While for the  quartic  nonlinearity $\alpha=3$, we obtain unconditional uniqueness for hierarchy equations  in the cases:  $d=1,2$ for any  $s\geq\frac{d}{4}$ and  $d\geq 3$ for any $s>\frac{3d-4}{6}$.

\bigskip
\noindent
\emph{ Conditional results}:
In this paragraph we provide  two further uniqueness results of conditional type for hierarchy equations \eqref{int-hier}. Conditional here means that we require additional  implicit conditions on hierarchy solutions in the spirit of nonlinear dispersive equations.  Indeed,  consider a mapping
\begin{equation}
\label{unc-g}
\begin{aligned}
&g:L^2(\R^d)\cap L^r(\R^d)\longrightarrow  L^{r'}(\R^d)\,,&\\
\text{ with  } & 2\leq r <\frac{2d}{d-2} \;\text{  if } \; d\geq 2 \;\text{ or  }\; 2\leq r \leq \infty \text{ if } d=1 \quad(\text{ Here } \; \frac{1}{r'}+\frac{1}{r}=1); &
\end{aligned}
\end{equation}
 and assume the existence of a constant $\alpha>0$ such that for any $M>0$ there exists $C(M)<\infty$  satisfying:
\begin{equation}
\label{unc.eq.1}
||g(u_1)-g(u_2)||_{L^{r'}(\R^d)}\leq C(M) \; \bigg( ||u_1||_{L^r(\R^d)}^\alpha+ ||u_2||_{L^r(\R^d)}^\alpha\bigg)\;
 ||u_1-u_2||_{L^r(\R^d)}\,,
\end{equation}
for any $u_1,u_2\in L^2(\R^d)\cap L^r(\R^d) $ such that $||u_1||_{L^2(\R^d)}\leq M$ and $||u_2||_{L^2(\R^d)}\leq M$.

\medskip
We are interested on the NLS equation \eqref{eq.ivp}, with $g(u)$  given as before, and its corresponding
hierarchy equation \eqref{int-hier}.  To fit our general setting for hierarchy equations  \eqref{int-hier} we need the following  observations: Firstly, one can easily extend the mapping $g:L^2(\R^d)\cap L^r(\R^d)\to  L^{r'}(\R^d)$ to a  Borel map  $G:L^2(\R^d)\to H^{-1}(\R^d)$ (see Lemma \ref{ext-borel-field} in the Appendix). Secondly, the results which we state below are independent from the choice of the extension.

\begin{prop}
\label{uniq-unc1}
Let $g$ be a $U(1)$-invariant mapping satisfying \eqref{unc-g}-\eqref{unc.eq.1} and such that:
$$
\frac{2}{q}:=d\,(\frac{1}{2}-\frac{1}{r})<\frac{2}{\alpha+2}\,.
$$
Consider the  Borel vector field  $v$ given by
\begin{equation}
\label{vect-field}
v(t,x):=-i \,\mathcal{U}(-t) \,G(\mathcal{U}(t) x)
\end{equation}
with $\mathcal{U}(t)=e^{it \Delta}$ and $G$ is any  Borel extension of the mapping $g$.  Let $I$ a bounded open interval and $t\in I\to \gamma_t,\tilde\gamma_t \in\mathscr{H}(L^2(\R^d))$ two solutions of the hierarchy equation \eqref{int-hier}  associated to $v$ and satisfying the assumption \eqref{A2} (with $s=0$, $\sigma=1$ and  $A=-\Delta+1$).    Furthermore, suppose  that
\begin{equation}
\label{cond-Stri}
\int_{I} \int_{L^2} ||\,\mathcal{U}(t) \,x||_{L^r}^q\; d\mu_t(x)\,dt<\infty \quad \text{ and } \quad
\int_{I} \int_{L^2} ||\,\mathcal{U}(t)\, x||_{L^r}^q\; d\tilde\mu_t(x)\,dt<\infty\,,
\end{equation}
where $\mu_t$ and $\tilde\mu_t$ are the two Borel probability measures verifying  for all $k\in\N$,
\begin{equation}
\label{cond-eq2}
\gamma_t^{(k)}=\int_{L^2} |x^{\otimes k} \rangle \langle x^{\otimes k} | \;d\mu_t(x)\, \quad \text{ and } \quad
\tilde\gamma_t^{(k)}=\int_{L^2} |x^{\otimes k} \rangle \langle x^{\otimes k} | \;d\tilde\mu_t(x)\,.
\end{equation}
Then   $\gamma_{t_0}=\tilde\gamma_{t_0}$ for some $t_0\in I$ implies that  $\gamma_{t}=\tilde\gamma_{t}$ for all $t\in I$.
\end{prop}
In contrast with the previous examples for unconditional uniqueness, in this case  we do not  assume that the
vector field $v$ is bounded on bounded sets. Nevertheless, the hierarchy equation \eqref{int-hier} still make sense for solutions $t\in I\to \gamma_t$  satisfying \eqref{A2} and \eqref{cond-Stri}.  In fact, one easily checks that
$$
\int_I \int_{L^2} ||v(t,x)||_{H^{-1}(\R^d)} \;d\nu_t(x) \,dt<\infty \,,
$$
where  $\nu_t$ is either $\mu_t$ or $\tilde\mu_t$ given in the right hand side of \eqref{cond-eq2}. So, the related Liouville equation \eqref{int-liouville}   makes sense even though the vector field $v$ may not be bounded on bounded sets. Consequently, using \eqref{est-vect1} one notices that  the hierarchy equation \eqref{int-hier} is meaningful  since
$$
\int_I ||(A^{-1/2})^{\otimes k}\gamma^{(k)}  \,C^{\pm}_{j,k} \gamma_t \,(A^{-1/2})^{\otimes k}||_{\mathscr{L}^1(L^2(\R^{dk}))}  \,dt<\infty \,,
$$
for any $k\in\N$ and $1\leq j\leq k$.  Notice also that the conditional uniqueness result given in Prop.~\ref{uniq-unc1}  holds for nonlinearity of the type
$$
g(u)=V u+f(\cdot,u(\cdot))+(W*|u|^2) \,u\,,
$$
where $V\in L^\delta(\R^d)+L^\infty(\R^d)$ for some $\delta\geq 1, \delta>\frac{d}{2}$,
$W\in L^\beta(\R^d)+L^\infty(\R^d)$ for some $\beta\geq 1, \beta>\frac{d}{2}$ and  $f$ satisfying  $f(x,0)=0$
for all $x\in\R^d$ and
$$
|f(x,z)-f(x,\tilde z)|\leq C (1+|z|+|\tilde z|)^\alpha \,|z-\tilde z|\,,
$$
for some $\alpha<\frac{4}{d}$ (see \cite[Corollary 4.6.5]{MR2002047}).

\medskip
We can also prove the following uniqueness result in the $L^2$-critical case. The proofs of Prop.~\ref{uniq-unc1} and \ref{uniq-unc2} are sketched  in Subsection \ref{sub.sec.uniq-ex}.
\begin{prop}
\label{uniq-unc2}
Take $g(u)=\lambda |u|^\alpha u$ with $\alpha=\frac{4}{d}$ and $\lambda\in\C$. Then the conclusion of Prop.~\ref{uniq-unc1}
holds true also in this case if we assume \eqref{cond-Stri} for $q=\alpha+2$.
\end{prop}

\bigskip
\noindent
\emph{Counter-example:}  Finally, we borrow a nice counter-example from the work of Win-Tsutsumi \cite{MR2474179} for the $L^2$-critical NLS equation (i.e.,
$g(u)=- |u|^\alpha u$ with $\alpha=\frac{4}{d}$)  in order to show that Prop.~\ref{uniq-unc2} breaks down if we
remove the assumption \eqref{cond-Stri}.  Indeed, consider the function
$$
u(t,x)=\frac{1}{(2t)^{d/2}} \, e^{i |x|^2} \, e^{i/t} \, \phi(\frac{x}{2t})\,,
$$
where $\phi$ is a solution of the nonlinear elliptic equation
$$
-\Delta\phi+\phi-\phi^{1+4/d}=0, \quad \phi>0, \quad \phi\in H^1(\R^d)\,.
$$
Then one checks that $t\in\R\setminus\{0\}\to u(t):=u(t,\cdot)\in L^2(\R^d)$ is  continuous and $u(t)$ converges weakly
to $0$ when $t\to 0$. So that $t\to u(t)\in L^2(\R^d)$ is a weakly continuous solution of the NLS equation \eqref{eq.ivp} satisfying $u(0)=0$. Consequently,  the curve $t\to \gamma_t\in\mathscr{H}(L^2(\R^d))$ given
for any $k\in\N$ by
$$
\gamma_t^{(k)}=\int_{L^2} |x^{\otimes k} \rangle \langle x^{\otimes k} | \;d\mu_t(x)\,,
$$
such that  $\mu_t= \delta_{\mathcal{U}(-t) u(t)}$, satisfies \eqref{A2},   $\gamma_0=0$ and the hierarchy equation \eqref{int-hier} with the corresponding vector field \eqref{vect-field}.  So, we conclude  that the hierarchy equation \eqref{int-hier} lacks uniqueness since $\gamma_t\neq 0$ if $t\neq 0$  while  the null hierarchy is also a solution.

\bigskip
\noindent
%\emph{Outline of the article:}

\section{Structure of symmetric hierarchies}
\label{sec:01}
According to the fundamental postulates of  quantum mechanics, a {normal state} of a system of $n$-quantum particles is described by a \emph{density operator}, which is a non-negative trace class  operator $\varrho_n$ of trace one acting on a tensor product of $n$-Hilbert spaces
$\mathfrak{H}_1\otimes\cdots\otimes \mathfrak{H}_n$.  Moreover, when the particles are bosons and identical, the system should obey  the Bose-Einstein statistics.
This means that the Hilbert spaces are identical $\mathfrak{H}=\mathfrak{H}_1=\cdots= \mathfrak{H}_n$  and the multiple-particle state $\varrho_n$ satisfies the following  symmetry,
\begin{equation}
\label{eq.1}
U_\pi \,\varrho_n =\varrho_n  \,,
\end{equation}
where $U_\pi$ is the operator defined for any  permutation $\pi$ as
$$
U_\pi f_1\otimes \cdots  \otimes f_n:=  f_{\pi(1)}\otimes\cdots \otimes f_{\pi(n)}\,.
$$
In particular, the following operator
$$
S_n :=\frac{1}{n!} \sum_{\pi} U_\pi
$$
defines an orthogonal projection on $\mathfrak{H}^{\otimes n}$ and the relation \eqref{eq.1} yields
$$
S_n \varrho_n=\varrho_n S_n = \varrho_n.
$$
This means that   a normal  state of a system of  $n$-identical  bosons is in fact a density  operator on the $n$-symmetric tensor product
$$
\vee^n\mathfrak{H}:=S_n\,\mathfrak{H}^{\otimes n}\, .
$$

\subsection{Symmetric hierarchies}
Consider  a sequence of normal states $(\varrho_n)_{n\in \N}\in \Pi_{n\in \N} \mathscr{L}^1(\vee^n\mathfrak{H})$. Then  for each fixed $k$ the $k$-particle reduced density matrices $(\varrho_n^{(k)})_{n\geq k}$, defined according to \eqref{redmat}, form a sequence of trace-class operators in the same space $\mathscr{L}^1(\vee^k\mathfrak{H})$. Hence, one can study the weak-$*$ convergence of  $(\varrho_n^{(k)})_{n\geq k}$. The following result asserts the existence of subsequential limits given by the same subsequence for all $k\in\N$.

\begin{lem}
\label{lem.1}
For any sequence of normal states  $(\varrho_n)_{n\in \N}\in \Pi_{n\in \N} \mathscr{L}^1(\vee^n\mathfrak{H})$ there exists a subsequence $(\varrho_{n_i})_{i\in \N}$ and a sequence
$(\gamma^{(k)})_{k\in \N} \in  \Pi_{k\in \N} \mathscr{L}^1(\vee^k\mathfrak{H})$
such that for any $k\in \N$,
\begin{equation}
\label{eq.sec1.1}
\lim_{i\to \infty} \Tr_{\vee^k \mathfrak{H}} \big[\varrho_{n_i}^{(k)}  \, A\big]=\Tr_{\vee^k \mathfrak{H}}
\big[\gamma^{(k)}  \, A\big]\,, \quad \forall A \in  \mathscr{L}^\infty(\vee^k\mathfrak{H})\,.
\end{equation}
\end{lem}
\begin{proof}
This follows by  Banach-Alaoglu theorem on  $  \mathscr{L}^1(\vee^n\mathfrak{H})= \mathscr{L}^\infty(\vee^n\mathfrak{H})^*$  together with a diagonal extraction argument.
\end{proof}

The subsequential limits provided by Lemma \ref{lem.1} are in fact the main analyzed quantities   in the study of  the mean field theory of Bose gases. As explained in the introduction, the many-body Schr\"odinger dynamics \eqref{eq.hn} yields the Gross-Pitaevskii and Hartree  hierarchies  \eqref{GPH-cubic}  in the mean-field limit. Furthermore, the solutions  $(\gamma^{(k)}(t))_{k\in\N}$ are actually the kernels of the subsequential limits of reduced density matrices originated from the quantum states $(\varrho_n(t))_{n\in\N}$ given by \eqref{qstate}.  So, this means that the physically relevant solutions of the  Gross-Pitaevskii and Hartree  hierarchies  \eqref{GPH-cubic}  are not any arbitrary sequences  $(\gamma^{(k)}(t))_{k\in\N} \in \Pi_{k\in\N}\mathscr{L}^1(\vee^k \mathfrak{H})$ of trace-class operators
but they have a more specific structure inherited from the relation \eqref{eq.sec1.1}.  So, this  motivated the following definition.

\begin{defn}
\label{sec1.def1}
We call a { symmetric hierarchy} any sequence $(\gamma^{(k)})_{k\in \N}$ \, $ \in  \Pi_{k\in \N} \mathscr{L}^1(\vee^k\mathfrak{H})$ such that there exists a sequence  of normal states $(\varrho_n)_{n\in \N}\in \Pi_{n\in \N} \mathscr{L}^1(\vee^n\mathfrak{H})$ and a subsequence $ (\varrho_{n_i})_{i\in \N}$ satisfying \eqref{eq.sec1.1} for all $k\in \N$. The set of all  symmetric hierarchies will be denoted by $\mathscr{H}(\mathfrak{H})$.
\end{defn}

The terminology of \emph{ symmetric hierarchy} is probably uncommon in the literature, nevertheless it seems justified on one hand by  the Bose-Einstein symmetry  and on the other one by the fact that these subsequential limits $(\gamma^{(k)})_{k\in \N}$  are countable collections  of linked trace-class operators  (i.e.,  $\gamma^{(k+1)}$ and $\gamma^{(k)}$ are related in some sense) and hence deserving the name of hierarchies.

\medskip
Some simple consequences may be derived from the Definition \ref{sec1.def1}.
In particular, the weak-$*$ convergence yields  that any element of  $(\gamma^{(k)})_{k\in \N}\in \mathscr{H}(\mathfrak{H})$ satisfies the constraint \eqref{cont2}, i.e.:
\begin{eqnarray*}
0 \leq \gamma^{(k)} \leq 1, \quad \forall k\in \N\,.
\end{eqnarray*}
One also can see that we have the following simple example:
\begin{itemize}
\item   Let $\varrho_n=|\varphi_n^{\otimes n} \rangle \langle \varphi_n^{\otimes n} | $ such that $ || \varphi_n||=1$ and $\varphi_n  \rightharpoonup  \varphi$ (with $|| \varphi||\leq 1$ ). Then $\varrho_n^{(k)} =|\varphi_n^{\otimes k} \rangle \langle \varphi_n^{\otimes k} | \overset{*}{\rightharpoonup}|\varphi^{\otimes k} \rangle \langle \varphi^{\otimes k} | $. Hence, for any $\varphi$ in the closed  unit ball of $\mathfrak{H}$,
\begin{equation}
\label{pure.points}
(|\varphi^{\otimes k} \rangle \langle \varphi^{\otimes k} |)_{k\in \N}\in \mathscr{H}(\mathfrak{H}).
\end{equation}
\end{itemize}
This shows for instance that $\mathscr{H}(\mathfrak{H})$ is a nontrivial  (uncountable) set. More fundamental  properties of $\mathscr{H}(\mathfrak{H})$ will be given in the next section.

\subsection{Integral representation}
\label{sec:02}
The set  of symmetric hierarchies $\mathscr{H}(\mathfrak{H})$ can be described as a  convex combinations of the simple elements in \eqref{pure.points}. Such a superposition principle is usually interpreted  as a non-commutative de Finetti theorem. Long ago St{\o}rmer in \cite{MR0241992} have proved for a different purpose, and using Choquet theory, a non-commutative de Finetti theorem for invariant states over some  $C^*$-algebras. Thereafter, Hudson and Moody in \cite{MR2108315,MR660367,MR0397421} specified St{\o}rmer's result to the framework of normal states and gave an integral representation for elements $(\gamma^{(k)})_{k\in \N}\in \mathscr{H}(\mathfrak{H})$  which satisfy $\Tr[\gamma^{(k)}]=1$ for all $k\in \N$. With a different point of view, Ammari and Nier have proved in \cite{MR2465733} an integral representation for all elements of $\mathscr{H}(\mathfrak{H})$ (see Prop.~\ref{str.BEh}) by appealing to Wigner measures. Subsequently, Lewin, Nam and Rougerie gave in \cite{MR3161107} an alternative proof. Thereafter, several other applications to this superposition principle appeared (see e.g.~\cite{MR3335056,MR3385343,MR3210237}).

In fact, each  symmetric hierarchy in $\mathscr{H}(\mathfrak{H})$ can be written as an integral, over a probability measure, of  elements in \eqref{pure.points}. Here we recall such a result and refer the reader for instance to  \cite[Proposition 1.1]{MR3506807}   for more details.

\begin{prop}
\label{str.BEh}
For each $(\gamma^{(k)})_{k\in \N}\in \mathscr{H}(\mathfrak{H})$ there exists a unique Borel probability
measure $\mu$ on the closed unit ball $ X_1:=B_\mathfrak{H}(0, 1)$ of $\mathfrak{H}$ such that it  is  $U(1)$-invariant and satisfies for any $k\in\N$,
\begin{equation}
\label{decomp}
\gamma^{(k)}=\int_{X_1} |\varphi^{\otimes k} \rangle \langle \varphi^{\otimes k} | \;d\mu(\varphi)\,.
\end{equation}
\end{prop}

\noindent
The following comments are useful:
\begin{itemize}
\item The Borel $\sigma$-algebras on $X_1$ equipped respectively with the norm or the weak topology coincide. Moreover,    $\mu$ can  be considered  as Borel probability measure on $\mathfrak{H}$ concentrated on the unit ball  $X_1$ (i.e., $\mu(X_1)=1$).
\item Actually, $\mu$ can be interpreted as the Wigner measure of the subsequence $(\varrho_{n_i})_{i\in \N} $  in  Lemma \ref{lem.1} that leads to the symmetric hierarchy $(\gamma^{(k)})_{k\in \N}\in \mathscr{H}(\mathfrak{H})$. For a short review on the relationship between symmetric hierarchies and Wigner measures,  we refer to
\cite[Section 3]{MR3506807}.
\item The integral  \eqref{decomp} is well defined in a weak sense and also as a Bochner integral in
$ \mathscr{L}^1(\vee^k\mathfrak{H})$ for each $k\in \N$.
\end{itemize}

\begin{prop}
\label{prop.BEh}
Let $(\gamma^{(k)})_{k\in \N}\in \mathscr{H}(\mathfrak{H})$  and  $\mu\in\mathscr{M}(\mathfrak{H})$ given by Prop.~\ref{str.BEh} and satisfying \eqref{decomp}. Then the following assertions are equivalent:
\begin{description}
\item (i)  The measure $\mu$ concentrated  on the unit sphere of $\mathfrak{H}$  (i.e., $\mu(S_{\mathfrak{H}}(0,1))=1$).
\item (ii) $ \Tr[ \gamma^{(k)}] =1$ for all  $k\in\N$ .
\item (iii) If  $(\varrho_n)_{n\in \N}$ is a sequence of normal states in $
\Pi_{n\in \N} \mathscr{L}^1(\vee^n\mathfrak{H})$ satisfying \eqref{eq.sec1.1} (i.e., there exits a subsequence
$(n_i)_{i\in\N}$ such that $\varrho_{n_i}^{(k)}\overset{ *}{ \rightharpoonup}  \gamma^{(k)}$ )  then
$(\varrho_{n_i}^{(k)})_{i\in \N}$ converges  to $ \gamma^{(k)}$ in the norm topology of $ \mathscr{L}^1(\vee^k\mathfrak{H})$ for all  $k\in\N$.
\item (iv) $ \Tr[ \gamma^{(k)}] =1$ for some $k\in\N$ .
\end{description}
\end{prop}
\begin{proof}
(i)$\Rightarrow$(ii): The integral representation \eqref{decomp} yields for all $k\in\N$,
$$
\Tr[\gamma^{(k)}]=\int_{X_1} ||\varphi||^{2k} \; d\mu(\varphi)\,.
$$
Hence, if  the measure $\mu$ concentrates on the unit sphere then $\Tr[\gamma^{(k)}]=1$ for all $k\in\N$.  \\
(ii)$\Rightarrow$(iii): Follows by a general property of spaces of trace-class operators called the Kadec-Klee property (KK*); and  which ensures that  weak-$*$ and norm convergence coincide on the unit sphere of  $\mathscr{L}^1(\vee^k\mathfrak{H})$ (see  Appendix \ref{KKstar} ).  Hence, (ii) converts
the weak-$*$ convergence of  $(\varrho_{n_i}^{(k)})_{i\in\N}$ towards $\gamma^{(k)}$  to a norm convergence. \\
(iii)$\Rightarrow$(iv): is trivial since $\Tr[\varrho_{n_i}^{(k)}]=1$ for all $i\in\N$.\\
(iv)$\Rightarrow$(i): For some $k\in\N$,
\begin{equation*}
\Tr[\gamma^{(k)}]=\int_{||\varphi||=1} 1 \; d\mu(\varphi)+
\int_{||\varphi||<1} ||\varphi||^{2k} \; d\mu(\varphi)=1\,.
\end{equation*}
So, this implies
\begin{equation*}
\int_{||\varphi||<1} 1- ||\varphi||^{2k}  \; d\mu(\varphi)=0
\end{equation*}
Taking any  $0< R<1$, we see that $0\leq (1-R^{2k})  \mu( B_\mathfrak{H}(0,R))\leq \int_{||\varphi||<1} 1- ||\varphi||^{2k}  \; d\mu(\varphi)$. Hence,  $\mu(B_\mathfrak{H}(0,R))=0$ and consequently the measure $\mu$ concentrates on the unit sphere of $\mathfrak{H}$.
\end{proof}

\begin{lem}
Let $(\gamma^{(k)})_{k\in \N}\in \mathscr{H}(\mathfrak{H})$  and  $\mu\in\mathscr{M}(\mathfrak{H})$ given by Prop.~\ref{str.BEh} and satisfying \eqref{decomp}.  Then $ \Tr_{k+1} [ \gamma^{(k+1)}] = \gamma^{(k)}$  for all $k\in\N$
if and only if $\mu(\{0\}\cup S_{\mathfrak{H}}(0,1))=1$.
\end{lem}
\begin{proof}
Using the integral representation \eqref{decomp}, one deduces
\begin{equation}
\label{eq.trac}
\int_{X_1} |\varphi^{\otimes k}\rangle\langle \varphi^{\otimes k}|  \;d\mu(\varphi)\,=\gamma^{(k)}=\Tr_{k+1} [ \gamma^{(k+1)}] =\int_{X_1} ||\varphi||^2\,
 |\varphi^{\otimes k}\rangle\langle \varphi^{\otimes k}|  \;d\mu(\varphi)\,   .
\end{equation}
Taking the trace in the latter identity, implies
$$
\int_{X_1}  ||\varphi||^{2k}\, (1- ||\varphi||^{2})  \;d\mu(\varphi)\,=0
$$
Hence, this shows that $\mu(\{0\}\cup S_{\mathfrak{H}}(0,1))=1$.  Conversely, if the measure $\mu$ concentrates on the origin and the unit sphere then   \eqref{eq.trac} holds true.
\end{proof}

%\subsection{Second definition}
It is useful to  identify the set of  symmetric hierarchies $\mathscr{H}(\mathfrak{H})$ in a more simpler way without appealing to convergence of subsequences as in  Definition \ref{sec1.def1}.   In fact, we can show that the two sets $\mathscr{H}(\mathfrak{H})$ and $\mathscr{M}(\mathfrak{H})$
are in one-to-one correspondence.

\begin{prop}
\label{bij}
The following mapping
\begin{eqnarray}
\Phi:\mathscr{M}(\mathfrak{H})&\rightarrow &  \mathscr{H}(\mathfrak{H})\\
\mu &\rightarrow & (\gamma^{(k)})_{k\in \N} =\left(\int_{X_1} |\varphi^{\otimes k} \rangle \langle \varphi^{\otimes k} | \;d\mu\right)_{k\in \N}  \,,
\end{eqnarray}
is a (bijective) one-to-one correspondence.
\end{prop}
\begin{proof}
It is enough to prove that for any $\mu\in \mathscr{M}(\mathfrak{H})$ the sequence $\Phi(\mu)$ is, according to the Definition \ref{sec1.def1}, a  symmetric hierarchy. \\
Let $\{e_n\}_{n\in \N}$ be an O.N.B of $\mathfrak{H}$. For any $\varphi\in X_1$, we have the
decomposition $  \varphi=\sum_{k=1}^\infty \varphi_k e_k$ . Consider the mapping
\begin{eqnarray*}
\Psi_n: B_\mathfrak{H}(0,1)& \to & S_\mathfrak{H}(0,1)\\
\varphi &\to & \Psi_n(\varphi):= \bigg(\sqrt{
1-\sum_{k\neq n} |\varphi_k|^2 } - \varphi_n\bigg) e_n+ \varphi\,.
\end{eqnarray*}
Remark that for every $\varphi\in B_\mathfrak{H}(0,1)$,
$$
||\Psi_n(\varphi)||^2= 1- \sum_{k\neq n} |\varphi_k|^2 +\sum_{k\neq n} |\varphi_k|^2=1\,.
$$
We claim that $\Psi_n$ is continuous for each $n\in\N$.  In fact, the following inequality holds
\begin{equation}
\label{contpsi}
\|\Psi_n(\varphi)-\Psi_n(\tilde\varphi)\|\leq \| \varphi-\tilde\varphi\|
+ \|  \sqrt{1-\sum_{k\neq n} |\varphi_k|^2 }-  \sqrt{
1-\sum_{k\neq n} |\tilde\varphi_k|^2 }\|\,,
\end{equation}
and if  $\varphi\to \tilde\varphi$ in $X_1$ then $ P_n (\varphi)\to P_n (\tilde\varphi)$ with $P_n$ is  the orthogonal projection over the subspace ${\rm Vect}\{e_k,k\neq n\} $. This shows that the right hand side  of \eqref{contpsi} converges to $0$ when  $\varphi\to \tilde\varphi$ at fixed $n$.  \\
Let $\mu\in \mathscr{M}(\mathfrak{H})$  and consider the sequence,
\begin{equation}
\label{eq.seq1}
\varrho_n=\int_{X_1}  |\Psi_n(\varphi)^{\otimes n} \rangle \langle \Psi_n(\varphi)^{\otimes n} | \;d\mu\,.
\end{equation}
Then one checks that  $(\varrho_n)_{n\in\N}\subset   \mathscr{L}^1(\vee^n\mathfrak{H})$ with
$\varrho_n\geq 0$ and $ \Tr[\varrho_n]=1$ since $\Psi_n(\varphi)\in S_\mathfrak{H}(0,1)$ for every
$\varphi\in B_{\mathfrak{H}}(0,1)$.  Notice that the function
$\varphi\to  |\Psi_n(\varphi)^{\otimes n} \rangle \langle \Psi_n(\varphi)^{\otimes n} |\in  \mathscr{L}^1(\vee^k\mathfrak{H})$ is continuous  and Bochner integrable.
Moreover,  we have the reduced density matrices
$$
\varrho_n^{(k)}= \int_{X_1}  |\Psi_n(\varphi)^{\otimes k} \rangle \langle \Psi_n(\varphi)^{\otimes k} |
\;d\mu\,.
$$
The point  is that $\Psi_n(\varphi)\underset{n\to\infty}{\rightharpoonup} \varphi$ for any
$\varphi\in B_\mathfrak{H}(0,1)$. Hence, one shows
$$
|\Psi_n(\varphi)^{\otimes k}\rangle \langle \Psi_n( \varphi)^{\otimes k} |\overset{*}{\rightharpoonup} |\varphi^{\otimes k}\rangle \langle \varphi^{\otimes k} |\,.
$$
when $n\to \infty$. Dominated convergence yields that for any $k\in \mathbb{N}$,
$$
\varrho_n^{(k)} \overset{*}{\rightharpoonup}
\left(\int_{X_1} |\varphi^{\otimes k} \rangle \langle \varphi^{\otimes k} | \;d\mu\right)_{k\in \N} .
$$
\end{proof}

\begin{defn}
Thus, one can naturally propose a second definition  for  symmetric hierarchies,
\begin{equation}
\mathscr{H}(\mathfrak{H})=\left\{\left(\int_{X_1} |\varphi^{\otimes k} \rangle \langle \varphi^{\otimes k} | \;d\mu\right)_{k\in \N} ,\; \mu\in \mathscr{M}(\mathfrak{H})\right\}\,.
\end{equation}
\end{defn}
The previous identification  yields the following simple  consequence.
\begin{cor}
$\mathscr{H}(\mathfrak{H})$ is  a convex subset of $\Pi_{k\in \N} \mathscr{L}^1(\vee^k\mathfrak{H})$.
\end{cor}
\begin{proof}
This follows from the fact that $\mathscr{M}(\mathfrak{H})$ is a convex set.
\end {proof}

\subsection{Topological isomorphism }
It is useful to endow the set   of  symmetric hierarchies $\mathscr{H}(\mathfrak{H})$ with two natural topologies. Remember that the Hilbert space $\mathfrak{H}$ is separable. Consequently, for all $k\in\N$ the Banach spaces
 $\mathscr{L}^\infty(\vee^k \mathfrak{H})$ are also separable. Moreover, $\mathscr{L}^1(\vee^k \mathfrak{H})=\mathscr{L}^\infty(\vee^k \mathfrak{H})^*$ are endowed with two distinguished  topologies, namely the weak-$*$ and the norm topology. Recall that the weak-$*$ topology is metrisable  on bounded sets of  $\mathscr{L}^1(\vee^k \mathfrak{H})$. For instance, the bounded set
$$
 \mathscr{L}_0^1(\vee^k\mathfrak{H}):=\{\gamma\in \mathscr{L}^1(\vee^k\mathfrak{H}),
 0\leq \gamma\leq 1\}\,,
$$
can be equipped with the following metric of weak-$*$ convergence,
\begin{eqnarray*}
\mathbf{d}^{(k)}(\gamma,\tilde\gamma):= \sum_{i\in\N}
\frac{1}{2^i}  \;\frac{|| K_i \,(\gamma-\tilde \gamma)||_{   \mathscr{L}^1(\vee^k\mathfrak{H})}}{
1+|| K_i \,(\gamma-\tilde \gamma)||_{   \mathscr{L}^1(\vee^k\mathfrak{H})}}\,,
\end{eqnarray*}
with $\{K_i\}_{i\in \N}$ is a dense  countable set in $ \mathscr{L}^\infty(\vee^k\mathfrak{H})$. So, the distance $\mathbf{d}^{(k)}$  induces the weak-$*$ topology on $\mathscr{L}_0^1(\vee^k\mathfrak{H})$.
Notice that
\begin{equation}
\label{eq.2}
 \mathscr{H}(\mathfrak{H})\subset\Pi_{k\in\N} \mathscr{L}_0^1(\vee^k\mathfrak{H})
 \subset\Pi_{k\in\N} \mathscr{L}^1(\vee^k\mathfrak{H})\,,
\end{equation}
and that the Cartesian product in  the right hand side of \eqref{eq.2} can be equipped with a product topology (with the weak-$*$ or  norm topology on each component  $ \mathscr{L}^1(\vee^k\mathfrak{H})$).
Thus, one can consider the set of  symmetric hierarchies  $\mathscr{H}(\mathfrak{H})$ as a metric space endowed with one of the two product distances $\mathbf{d}_w$ or $\mathbf{d}_s$ given below,
\begin{eqnarray}
\label{distw}
\mathbf{d}_w(\gamma,\tilde\gamma)&:=& \sum_{k\in\N}
\frac{1}{2^k}  \;\mathbf{d}^{(k)}\big( \gamma^{(k)},\tilde \gamma^{(k)}\big)\,,
\\
\label{dists}
\mathbf{d}_s(\gamma,\tilde\gamma)&:=& \sum_{k\in\N}
\frac{1}{2^k}  \,|| \gamma^{(k)}-\tilde \gamma^{(k)}||_{   \mathscr{L}^1(\vee^k\mathfrak{H})}\,,
\end{eqnarray}
for all $\gamma=(\gamma^{(k)})_{k\in \N}$ and $ \tilde\gamma=(\tilde\gamma^{(k)})_{k\in \N}$ in
$\mathscr{H}(\mathfrak{H})$.  So, $\mathbf{d}_w$ (resp.~$\mathbf{d}_s$) induces the product topology in $\mathscr{H}(\mathfrak{H})$ with the weak-$*$  (resp.~ norm)  topology in each component $\mathscr{L}^1(\vee^k\mathfrak{H})$.

\bigskip
Remember that the set   $\mathscr{M}(\mathfrak{H})$ is endowed with two distinguished topologies, namely the weak and strong narrow convergence topologies (see Sect. \ref{fram}).  It is well known that the weak and strong narrow topologies  on the set of Borel  probability measures $\mathfrak{P}(X_1)$, with $X_1=B_{\mathfrak{H}}(0,1)$, are metrisable and in particular  $\mathfrak{P}(X_1)$ is a compact  metric space when endowed with the weak narrow topology.

\begin{prop}
\label{prop.weakcv}
Let $(\gamma,(\gamma_j)_{j\in\N})$ and $(\mu,(\mu_j)_{j\in\N})$ two sequences respectively in  $\mathscr{H}(\mathfrak{H})$ and $\mathscr{M}(\mathfrak{H})$  with $\Phi(\mu_j)=\gamma_j$ and $\Phi(\mu)=\gamma$. Then $\mu_j\rightharpoonup \mu$ weakly narrowly  in  $\mathscr{M}(\mathfrak{H})$ if and only if for  all
$k\in\N$,
\begin{equation}
\label{convstar}
\gamma_j^{(k)}=\int_{X_1} |\varphi^{\otimes k} \rangle \langle \varphi^{\otimes k} | \;d\mu_j  \overset{*}{\rightharpoonup}  \int_{X_1} |\varphi^{\otimes k} \rangle \langle \varphi^{\otimes k} | \;d\mu= \gamma^{(k)}\,.
\end{equation}
\end{prop}
\begin{proof}
For any $A\in \mathscr{L}^\infty (
\vee^k\mathfrak{H})$,
$$
\Tr[\gamma_j^{(k)}\, A]=\int_{X_1} \langle\varphi^{\otimes k}, A  \varphi^{\otimes k} \rangle
 \;d\mu_j\,.
$$
Remark that  the function $\chi: X_1\to \C$ , $ \chi(\varphi)= \langle\varphi^{\otimes k}, A  \varphi^{\otimes k} \rangle  $ is bounded and continuous with respect to  the distance $d_w$ (of the weak topology in $X_1$). Hence,  the weakly narrow convergence $\mu_j\rightharpoonup \mu$  shows that for any $A\in \mathscr{L}^\infty (
\vee^k\mathfrak{H})$,
$$
\lim_{j} \Tr[\gamma_j^{(k)}\, A]=\int_{X_1} \langle\varphi^{\otimes k}, A  \varphi^{\otimes k} \rangle
 \;d\mu= \Tr[\gamma^{(k)}\, A]\,,
$$
 with $\gamma^{(k)}=\int_{X_1} |\varphi^{\otimes k} \rangle \langle \varphi^{\otimes k} | \;d\mu$.
This proves that $\gamma_j^{(k)} \overset{*}{\rightharpoonup} \gamma^{(k)}$.

\bigskip

The proof of the inverse statement is a bit more involved.
Suppose that a sequence  $(\gamma_j)_{j\in\N}$ converges to $\gamma\in \mathscr{H}(\mathfrak{H})$  with respect to the distance $\mathbf{d}_w$. This means that \eqref{convstar} holds  for any $k\in\N$. In particular, taking $A=  |\xi^{\otimes k} \rangle \langle \xi^{\otimes k} |\in\mathscr{L}^\infty(\vee^k\mathfrak{H})$,
\begin{equation}
\label{eq.3}
\lim_j\Tr[\gamma_j^{(k)}\, A]=\lim_j\int_{X_1} \langle\varphi^{\otimes k}, \xi^{\otimes k} \rangle \langle \xi^{\otimes k},  \varphi^{\otimes k} \rangle   \;d\mu_j= \int_{X_1}
 | \langle\varphi, \xi \rangle|^{2k} \;d\mu\,,
\end{equation}
with $\mu\in \mathscr{M}(\mathfrak{H})$ such that $\Phi(\mu)=\gamma$. So, the characteristic function of $\mu_j$ have the following absolutely convergent expansion,
\begin{eqnarray*}
\int_{X_1} e^{i{\rm Re} \langle\varphi, \xi \rangle_\mathfrak{H}} \,
d\mu_j&=& \sum_{k=0}^\infty \frac{i^k}{k!}  \,\int_{X_1} {\rm Re} \langle\varphi, \xi \rangle^k \; d\mu_j \\
&=& \sum_{k=0}^\infty \frac{(-1)^k}{4^k (2k)!}  \,\int_{X_1} | \langle\varphi, \xi \rangle|^{2k} \; d\mu_j,
\end{eqnarray*}
where the last equality is a direct consequence of the $U(1)$-invariance of $\mu_j$. So, the dominated
convergence and  \eqref{eq.3} yield the following convergence for the characteristic functions,
$$
\lim_j \int_{X_1} e^{i{\rm Re} \langle\varphi, \xi \rangle_\mathfrak{H}} \,
d\mu_j= \int_{X_1} e^{i{\rm Re} \langle\varphi, \xi \rangle_\mathfrak{H}} \,
d\mu\,.
$$
 According to Theorem \ref{thmA}, in the Appendix \ref{measure}, the sequence $(\mu_j)_{j\in\N}$ converges towards $\mu$ weakly narrowly.
\end{proof}

\begin{prop}
\label{prop.strcv}
Let $(\gamma,(\gamma_i)_{i\in\N})$ and $(\mu,(\mu_i)_{i\in\N})$ two sequences respectively in  $\mathscr{H}(\mathfrak{H})$ and $\mathscr{M}(\mathfrak{H})$  with $\Phi(\mu_i)=\gamma_i$ and $\Phi(\mu)=\gamma$. Then $\mu_i\to \mu$ strongly narrowly  in  $\mathscr{M}(\mathfrak{H})$ if and only if for  all
$k\in\N$,
$$
\gamma_i^{(k)}=\int_{X_1} |\varphi^{\otimes k} \rangle \langle \varphi^{\otimes k} | \;d\mu_i\, {\rightarrow}  \int_{X_1} |\varphi^{\otimes k} \rangle \langle \varphi^{\otimes k} | \;d\mu= \gamma^{(k)}\,,
$$
in the norm topology of  $\mathscr{L}^1(\vee^k\mathfrak{H})$.
\end{prop}
\begin{proof}
For any $A\in \mathscr{L} (
\vee^k\mathfrak{H})$,
$$
\Tr[\gamma_i^{(k)}\, A]=\int_{X_1} \langle\varphi^{\otimes k}, A  \varphi^{\otimes k} \rangle
 \;d\mu_i\,.
$$
In particular, this implies that $\gamma_i^{(k)} \overset{*}{\rightharpoonup}  \gamma^{(k)}$ and $ \Tr[\gamma_i^{(k)}]\to \Tr[\gamma_i^{(k)}]$ by taking $A=1$. Hence, appealing to the Kadec-Klee property (KK*) of the space $\mathscr{L}^1(\vee^k\mathfrak{H})$ (see Thm.~\ref{th.kadec} in Appendix \ref{KKstar}), one shows that $\gamma_i^{(k)} {\to}  \gamma^{(k)}$ in the  norm topology.

\bigskip
The inverse statement is proved with the help of Thm.~\ref{thmB}. Suppose that
for any $k\in \N$, $\gamma_i^{(k)} \to  \gamma^{(k)}$ in the norm topology and consider
the sequence of  operators
$$
A_N =\sum_{i=N}^\infty |e_i\rangle\langle e_i| \in \mathscr{L}(\mathfrak{H})\,,
$$
with $(e_i)_{i\in\N}$ is an O.N.B of the Hilbert space $\mathfrak{H}$.
Then, we have
\begin{equation}
\label{eq.4}
\Tr[ \gamma^{(1)}_n \, A_N]  =\int_{X_1}\, \sum_{i=N}^\infty | \langle\varphi, e_i\rangle |^2 \, d\mu_n
\underset{n\to \infty}{\longrightarrow } \int_{X_1}\, \sum_{i=N}^\infty | \langle\varphi, e_i\rangle |^2 \, d\mu\,,
\end{equation}
uniformly  in $N\in \N$ since for all $N\in\N$,
$$
\left| \Tr[ (\gamma_n^{(1)}- \gamma^{(1)}) A_N] \right|
\leq || \gamma_n^{(1)}- \gamma^{(1)}||_{\mathscr{ L}(\mathfrak{H})} \underset{n\to \infty}{\rightarrow} 0\,.
$$
The Chebyshev's inequality gives for any $\varepsilon>0$,
$$
\int_{X_1} \displaystyle 1_{\{\sum_{i=N}^\infty | \langle\varphi, e_i\rangle |^2 \geq \varepsilon\}} \, d\mu_n
\leq \frac{1}{\varepsilon} \int_{X_1} \sum_{i=N}^\infty | \langle\varphi, e_i\rangle |^2  \, d\mu_n \,.
$$
Therefore,  one deduces from the uniform convergence \eqref{eq.4}  and the above inequality the following statement,
$$
\lim_{N\to \infty } \sup_{n\in\N}  \int_{X_1} \displaystyle 1_{\{\sum_{i=N}^\infty | \langle\varphi, e_i\rangle |^2 \geq \varepsilon\}} \, d\mu_n
\leq  \frac{1}{\varepsilon} \lim_{N\to \infty } \sup_{n\in\N} \int_{X_1} \sum_{i=N}^\infty | \langle\varphi, e_i\rangle |^2  \, d\mu_n =0\,.
$$
Moreover, by Proposition \ref{prop.weakcv}, one already  knows that $\mu_n\rightharpoonup\mu$ weakly narrowly  in $\mathscr{M}(\mathfrak{H})$. Thus, applying Theorem \ref{thmB} in the Appendix \ref{measure}, one proves the convergence $\mu_n\to \mu$ with respect to the strong narrow topology in $\mathscr{M}(\mathfrak{H})$.
\end{proof}

The following corollary provides a useful   characterisation of the relevant  topologies on  the set of symmetric hierarchies.
\begin{cor}
\label{homeo}
The mapping
\begin{eqnarray*}
\Phi:(\mathscr{M}(\mathfrak{H}), \tau) &\rightarrow &  (\mathscr{H}(\mathfrak{H}), \mathbf{d})\\
\mu &\rightarrow & (\gamma^{(k)})_{k\in \N} =\left(\int_{X_1} |\varphi^{\otimes k} \rangle \langle \varphi^{\otimes k} | \;d\mu\right)_{k\in \N}  \,,
\end{eqnarray*}
defines an homeomorphism in the two cases:
\begin{description}
 \item (i) $\tau$ is the strong narrow convergence topology and $\mathbf{d}$ is the product distance  $\mathbf{d}_s$  defined in
 \eqref{dists}.
 \item(ii) $\tau$ is the weak narrow convergence topology   and $\mathbf{d}$ is the product distance  $\mathbf{d}_w$ defined in \eqref{distw}.
 \end{description}
\end{cor}
\begin{proof}
Follows by Propositions  \ref{prop.weakcv} - \ref{prop.strcv}.
\end{proof}

Remark that the above homeomorphism allows to use the  Krein-Milman and the Choquet-Bishop-de Leeuw  theorems on the convex set of symmetric hierarchies.

\section{The Liouville-hierarchy duality}
\label{sec.equiv}
 We discuss, in this section, the rigorous formulation of    the Liouville and the hierarchy equations and establish their equivalence in full generality as stated in Theorem \ref{sec.0.thm1}.  More precisely, we will  prove that any curve $t\to \mu_t$ in  $\mathscr{M}(\mathscr{Z}_0)$ satisfying \eqref{A2} and solving the Liouville equation \eqref{int-liouville} will give a curve   $t\to\Phi( \mu_t)=\gamma_t$ in $\mathscr{H}(\mathscr{Z}_0)$  satisfying \eqref{A1} and  solving the hierarchy equation \eqref{int-hier}  and vice versa.
Remember that $(\mathscr{Z}_0,\mathscr{Z}_s, \mathscr{Z}_{-\sigma})$ is the triple of spaces introduce in
Subsection \ref{fram} with $0\leq s\leq\sigma$; and  $\Phi$ is the  homeomorphism, in Corollary \ref{homeo},   relating symmetric hierarchies in $\mathscr{H}(\mathscr{Z}_0)$  to probability measures in $\mathscr{M}(\mathscr{Z}_0)$.

\subsection{Regularity issues}
\label{reg-issue}
We first emphasis a general property showing a correspondence between  the concentration property of  measures $\mu\in \mathscr{M}(\mathscr{Z}_0)$ and regularity of the  hierarchies $\gamma=\Phi(\mu)\in \mathscr{H}(\mathscr{Z}_0)$.

\begin{lem}
\label{concent}
Let $\mu\in\mathscr{M}(\mathscr{Z}_0)$ and $\gamma=(\gamma^{(k)})_{k\in \N}\in \mathscr{H}(\mathscr{Z}_0)$ such that
$\Phi(\mu)=\gamma$. Then for any $s\geq 0$ and $R>0$:
\begin{equation}
\label{eq.5}
\mu(B_{\mathscr{Z}_s}(0,R)) =1 \Leftrightarrow \left( \Tr[ (A^{s/2})^{\otimes k} \, \gamma^{(k)} \, (A^{s/2})^{\otimes k} ] \leq R^{2k}, \forall k\in\N \right)\,.
\end{equation}
\end{lem}
\begin{proof}
The equivalence is a consequence of the identity,
$$
\Tr[ (A^{s/2})^{\otimes k} \, \gamma^{(k)} \, (A^{s/2})^{\otimes k} ] =\int_{X_1} ||\varphi||_{\mathscr{Z}_s}^{2k} \, d\mu\,.
$$
If the measure $\mu$ is concentrated  on the closed ball $B_{\mathscr{Z}_s}(0,R)$  then the inequalities in the right hand side of \eqref{eq.5} hold true. Conversely, the right hand side of
\eqref{eq.5} yields for all $k\in\N$,
$$
\int_{X_1} ||\varphi||_{\mathscr{Z}_s}^{2k} \, d\mu\,\leq R^{2k}\,,
$$
which in turn gives, by  Chebyshev's inequality, the concentration of the measure $\mu$ on $
B_{\mathscr{Z}_s}(0,R)$.
\end{proof}

   An important ingredient  in the problems of well-posedness and uniqueness of the Liouville and hierarchy equations is the regularity of the solutions with respect to time.    The following proposition identifies  the  relevant notions of regularity that we shall use.

\begin{prop}
\label{regcurv} Let $I$ be an interval and consider two curves  $t\in I\to\mu_t\in\mathscr{M}(\mathscr{Z}_0)$ and $t\in I\to\gamma_t=(\gamma_t^{(k)})_{k\in \N}\in \mathscr{H}(\mathscr{Z}_0)$ such that $\Phi(\mu_t)=\gamma_t$ for all $t\in I$. Assume that for some
$s\geq 0$ and $R>0$,  $\mu_t(B_{\mathscr{Z}_s}(0,R)) =1$ for all $t\in I$. Then for all $\tau\in (-\infty,s]$,
\begin{enumerate}
\item $t\in I\to\mu_t\in\mathfrak{P}(\mathscr{Z}_\tau)$ weakly narrowly continuous if and only if
$ t\to (A^{\tau/2})^{\otimes k} \, \gamma_t^{(k)} \, (A^{\tau/2})^{\otimes k} $ is continuous with respect to the
weak-$*$ topology in $\mathscr{L}^1(\vee^k\mathscr{Z}_0)$ for all $k\in\N$.
\item   $t\in I\to\mu_t\in\mathfrak{P}(\mathscr{Z}_\tau)$ strongly narrowly continuous if and only if
$ t\to (A^{\tau/2})^{\otimes k} \, \gamma_t^{(k)} \, (A^{\tau/2})^{\otimes k} $ is continuous with respect to the
norm topology in $\mathscr{L}^1(\vee^k\mathscr{Z}_0)$ for all $k\in\N$.
\end{enumerate}
\end{prop}
\begin{proof}
The proof follows by the same arguments as in Propositions \ref{prop.strcv}-\ref{prop.weakcv}.
\end{proof}

As a consequence of the above observations, we have the following equivalence between the two main assumptions
\eqref{A2} and \eqref{A1}.
\begin{lem}
\label{a1a2}
A curve $t\in I\to \gamma_t\in \mathscr{H}(\mathscr{Z}_0)$, defined over an interval $I$, satisfies the
assumption  \eqref{A2}  if and only if the curve $t\in I\to \mu_t=\Phi^{-1}(\gamma_t)\in \mathscr{M}(\mathscr{Z}_0)$ satisfies \eqref{A1}.
\end{lem}
\begin{proof}
The homeomorphism $\Phi$ in Corollary \ref{homeo} with Lemma \ref{concent} and
Proposition \ref{regcurv} give the equivalence  between the two assumptions \eqref{A1} and \eqref{A2}.
\end{proof}

\bigskip
In the  two next paragraphs we rigorously justify that the Liouville and the symmetric hierarchy equations are meaningful under  the assumptions \eqref{A0}, \eqref{A2} and \eqref{A1}.

\medskip
\emph{Liouville equations}:
In the Liouville equation   \eqref{eq.transport}, one  presumes that the integral with respect to time  is  well defined.
The following Lemma guaranties this property using \eqref{A0}  and  \eqref{A1}.
\begin{lem}
Let $v:\R\times \zeds\to\zedsi$ a vector field satisfying \eqref{A0} and  $t\in I\to \mu_t\in \mathscr{M}(\mathscr{Z}_0)$  a curve satisfying  \eqref{A1}. Then for any $\varphi\in\mathscr C_{0,cyl}^{\infty}(I \times \mathscr{Z}_{-\sigma})$  the map,
\begin{eqnarray}
\label{mes-time}
t\in I &\longrightarrow& \int_{\zeds}  {\mathrm Re} \langle v(t,x), \nabla \varphi(t,x)\rangle_{\zedsi} \;d\mu_t  \,,
\end{eqnarray}
belongs to  $L^\infty(I,dt)$ when $I$ is a bounded open interval.
\end{lem}
\begin{proof}
Since $t\in I\to \mu_t\in \mathfrak{P}(\zedsi)$ is weakly narrowly continuous then it is a Borel family  of probability measures  (i.e., for any Borel set  $B$ of $ \zedsi$ the map $t\to \mu_t(B)$ is measurable).
Thanks to \eqref{A0} and \eqref{A1}, a simple bound gives
$$
 \int_{\zeds}  \bigg| \langle v(t,x), \nabla \varphi(t,x)\rangle_{\zedsi} \bigg|\;d\mu_t \leq
  c \int_{\zeds}  \|v(t,x)\|_{\zedsi}\;d\mu_t<\infty\,.
$$
\end{proof}

\bigskip
\emph{Symmetric hierarchy equations}:
 In order to give a rigorous meaning to the symmetric  hierarchy equation given  formally in  \eqref{int-hier}, we introduce the following  Hilbert rigging  $\mathcal{D}_\alpha^{(k)} \subset \vee^k\mathscr{Z}_0\subset \mathcal{D}_{-\alpha}^{(k)}$, for $\alpha>0$, as in the preliminary Subsect.~\ref{fram},  with $\mathcal{D}_\alpha^{(k)}=D((A^{\alpha/2})^{\otimes k})$, equipped with its graph norm and $\mathcal{D}_{-\alpha}^{(k)}$ identifies with the  dual of the  latter space with respect to the inner product of $ \vee^k\mathscr{Z}_0$. Moreover,  consider the following two  Banach spaces:
\begin{eqnarray}
\label{Ls}
\mathscr{L}_s^1(\vee^k\mathscr{Z}_0)&:=&
\big\{
T\in \mathscr{L}( \mathcal{D}_{-s}^{(k)}, \mathcal{D}_{s}^{(k)}), \, (A^{s/2})^{\otimes k} \, T \, (A^{s/2})^{\otimes k}
\in \mathscr{L}^1(\vee^k\mathscr{Z}_0)\big\}\,,\\
\mathscr{L}_{-\sigma}^1(\vee^k\mathscr{Z}_0)&:=&
\big\{
T\in \mathscr{L}(\mathcal{D}_{\sigma}^{(k)},\mathcal{D}_{-\sigma}^{(k)}), \, (A^{-\sigma/2})^{\otimes k} \, T \, (A^{-\sigma/2})^{\otimes k}
\in \mathscr{L}^1(\vee^k\mathscr{Z}_0)\big\}\,,
\end{eqnarray}
endowed respectively with the norms,
\begin{eqnarray*}
||T||_{ \mathscr{L}_s^1(\vee^k\mathscr{Z}_0)} &:=&|| (A^{s/2})^{\otimes k} \, T \, (A^{s/2})^{\otimes k}||_{ \mathscr{L}^1(\vee^k\mathscr{Z}_0)}\,, \\
||T||_{ \mathscr{L}_{-\sigma}^1(\vee^k\mathscr{Z}_0)} &:=&|| (A^{-\sigma/2})^{\otimes k} \, T \, (A^{-\sigma/2})^{\otimes k}||_{ \mathscr{L}^1(\vee^k\mathscr{Z}_0)}\,.
\end{eqnarray*}
For  a curve of  symmetric hierarchies, $t\in I\to \gamma_t\in \mathscr{H}(\mathscr{Z}_0)$, satisfying the assumption  \eqref{A2}  such that  $\mu_t=\Phi^{-1}(\gamma_t)\in \mathscr{M}(\mathscr{Z}_0)$, we have defined in Section \ref{fram}  the following operations on  $\gamma_t$,
\begin{equation}
\label{defBjk}
\begin{aligned}
\bullet\;\; C_{j,k}^{+} \gamma_t^{} &:= \displaystyle\int_{\zeds} \big| x^{\otimes k} \rangle \langle x^{\otimes j-1} \otimes v(t,x)\otimes x ^{\otimes k-j}\big| \;d\mu_{t}(x)  \,,\\ \medskip
\bullet  \;\;C_{j,k}^{-} \gamma_{t}^{} &:=  \displaystyle\int_{\zeds} \big| x^{\otimes j-1} \otimes v(t,x)\otimes x ^{\otimes k-j} \rangle \langle x^{\otimes k} \big| \;d\mu_{t}(x) \,,
 \end{aligned}
\end{equation}
for any  $t\in I$, $k\in\N$  and   $j=1,\cdots,k$.  Here the projectors in the above integrals,  $$
P=\big| x^{\otimes k} \rangle \langle x^{\otimes j-1} \otimes v(t,x)\otimes x ^{\otimes k-j}\big| \quad \text{ and } \quad Q=\big| x^{\otimes j-1} \otimes v(t,x)\otimes x ^{\otimes k-j} \rangle \langle x^{\otimes k} \big|\,,
$$
are well defined operators in $\mathscr{L}(\mathscr{D}_{\sigma}^{(k)},\mathscr{D}_{-\sigma}^{(k)})$, respectively acting as follows for each fixed $x\in \mathscr{Z}_s$ and for any $\phi\in\mathscr{D}_{\sigma}^{(k)}$,
$$
P( \phi)= \langle x^{\otimes j-1} \otimes v(t,x)\otimes x ^{\otimes k-j},\phi\rangle_{\vee^k
\mathscr{Z}_0}\; x^{\otimes k} \; \text{ and } \;
Q( \phi)= \langle x^{\otimes k}, \phi\rangle_{\vee^k
\mathscr{Z}_0}\; \; x^{\otimes j-1} \otimes v(t,x)\otimes x ^{\otimes k-j} \,.
$$

\begin{lem}
\label{opBjk}
Let $v:\R\times \zeds\to\zedsi$ a vector field satisfying \eqref{A0} and  $t\in I\to \gamma_t\in \mathscr{H}(\mathscr{Z}_0)$  a curve satisfying
 \eqref{A2}. Then for any $t\in I$, $k\in\N$  and   $j\in\{1,\cdots,k\}$,  the operations \eqref{defBjk},
\begin{eqnarray*}
C_{j,k}^\pm: \mathscr{H}(\mathscr{Z}_0) &\longrightarrow & \Pi_{k\in \N}  \mathscr{L}_{-\sigma}^1(\vee^k\mathscr{Z}_0)\\
(\gamma_t^{(k)})_{k\in\N} &\longrightarrow& \bigg(C_{j,k}^\pm \gamma_t^{} \bigg)_{k\in\N}\,,
\end{eqnarray*}
are well defined  as Bochner integrals in    $\mathscr{L}_{-\sigma}^1(\vee^k\mathscr{Z}_0)$ with respect to $\mu_t$.
\end{lem}
\begin{proof}
One can show that the following  maps,
\begin{eqnarray*}
x\in \mathscr{Z}_s \to  | x^{\otimes k} \rangle \langle x^{\otimes j-1} \otimes v(t,x)
\otimes x^{\otimes k-j} |\in  \mathscr{L}_{-\sigma}^1(\vee^k\mathscr{Z}_0),\\ \noalign{\medskip}
x\in \mathscr{Z}_s \to  |  x^{\otimes j-1} \otimes v(t,x)
\otimes x^{\otimes k-j}\rangle \langle x^{\otimes k}  |\in
 \mathscr{L}_{-\sigma}^1(\vee^k\mathscr{Z}_0),
\end{eqnarray*}
are weakly measurable. Since $\mathscr{L}_{-\sigma}^1(\vee^k\mathscr{Z}_0)$ is a separable Banach space then by Pettis theorem the above maps are $\mu_t$-Bochner measurable. Moreover, using assumption \eqref{A0},\eqref{A2} and Lemma \ref{a1a2},  one shows
\begin{eqnarray}
\label{est-vect1}
\int_{\mathscr{Z}_s} \bigg\|   | x^{\otimes k} \rangle \langle x^{\otimes j-1} \otimes v(t,x)
\otimes x^{\otimes k-j} |\bigg\|_{ \mathscr{L}_{-\sigma}^1(\vee^k\mathscr{Z}_0)} \, d\mu_t
&\leq&  \int_{\mathscr{Z}_s} \| x\|_{\zed}^{2 k-1} \, \|v(t,x)\|_{\mathscr{Z}_{-\sigma}}  \, d\mu_t
\\ \label{est-vect2}
&\leq& M \,,
\end{eqnarray}
for some finite constant $M>0$. So, we see that the operations \eqref{defBjk} are well defined
as Bochner integrals in   $\mathscr{L}_{-\sigma}^1(\vee^k\mathscr{Z}_0)$ with respect to $\mu_t$.
%these integrals \eqref{defBjk} can also be understood  in a weak sense in $\mathscr{D}_{\sigma}^{(k)}$.
\end{proof}

\begin{lem}
\label{intBjk}
Let $v:\R\times \zeds\to\zedsi$ a vector field satisfying \eqref{A0} and $t\in I\to \gamma_t\in \mathscr{H}(\mathscr{Z}_0)$ be a curve satisfying
 \eqref{A2}. Then, for any  $k\in\N$  and   $j\in\{1,\cdots,k\}$, the map $t\in I\to  C_{j,k}^\pm \gamma_t \in \mathscr{L}_{-\sigma}^1(\vee^k\mathscr{Z}_0)$ is Bochner integrable  with respect to the Lebesgue measure over the bounded interval $I$.
\end{lem}
\begin{proof}
The map $t\to  (A^{-\sigma/2})^{\otimes k} C_{j,k}^\pm \gamma_t (A^{-\sigma/2})^{\otimes k}  \in \mathscr{L}^1(\vee^k\mathscr{Z}_0)$ is weakly measurable. So, by Pettis theorem $t\to  C_{j,k}^\pm \gamma_t \in \mathscr{L}_{-\sigma}^1(\vee^k\mathscr{Z}_0)$  is Bochner measurable. Moreover, one easily checks using \eqref{est-vect1}-\eqref{est-vect2},  (for $t>t_0$),
\begin{eqnarray*}
\int_{t_0}^t ||  C_{j,k}^\pm \gamma_\tau||_{ \mathscr{L}_{-\sigma}^1(\vee^k\mathscr{Z}_0)}  d\tau
&\leq & c_1 \int_{t_0}^t   \int_{\mathscr{Z}_s} \| x\|_{\zed}^{2 k-1} \, \|v(\tau,x)\|_{\mathscr{Z}_{-\sigma}}  \, d\mu_\tau d\tau \\
&\leq & c_2  \, |t-t_0|\,,
\end{eqnarray*}
for some constants $c_1,c_2$.
\end{proof}
Hence, under the assumption \eqref{A0} and the a priori condition \eqref{A2}  on solutions of the  symmetric hierarchy equation \eqref{int-hier}, the following  integral  make sense,
\begin{equation}
\label{hier}
 t\in I \to  \int_{t_{0}}^{t} \sum_{j=1}^{k}  (C_{j,k}^{+} \gamma_{\tau}^{} +  C_{j,k}^{-} \gamma_{\tau}^{}) \;d\tau \,\in \mathscr{L}_{-\sigma}^1(\vee^k\mathscr{Z}_0)\,,
\end{equation}
and it is continuous with respect to time. So, the hierarchy equation is meaningful and it is consistent with the
requirement  \eqref{A2}.

\bigskip
\emph{Kernel representation of hierarchy equations}:
In this paragraph we show that both the Gross-Pitaevskii  and the Hartree  hierarchies in \eqref{GPH-cubic} are particular  cases of the abstract symmetric hierarchy equation given in \eqref{int-hier} corresponding respectively  to the nonlinearities
$g(u)=|u|^2 u$ and $g(u)=V*|u|^2 u$ with the vector field $v$ given according to  \eqref{nls-v}  with the choice  \eqref{spec-fram} and such that the mapping \eqref{nonlin} is continuous and bounded on bounded sets. The hierarchy equations   \eqref{GPH-cubic} are usually understood via the corresponding integral equation,
  \begin{equation}
  \label{hier-kernel-int}
  \gamma^{(k)}(t)= \mathcal{U}^{(k)}(t)\gamma^{(k)}(0)-i \int_{0}^t  \mathcal{U}^{(k)}(t-\tau) \,
  B_k\gamma^{(k+1)}(\tau) \,d\tau\,,
  \end{equation}
where $\gamma^{(k)}(t)\in L^2(\R^{dk}\times \R^{dk})$ satisfying \eqref{cont1}-\eqref{cont2}
are kernels of non-negative trace class operators and $\mathcal{U}^{(k)}(t)$ is the one-parameter group,
$$
\mathcal{U}^{(k)}(t)=\displaystyle \prod_{j=1}^{k}  \;e^{i t (\Delta_{x_j}-\Delta_{x'_j})}\,.
$$
For the equation \eqref{hier-kernel-int}  to make sense, one usually looks for solutions satisfying for some $R>0$ and  for all $k\in\N$ and  all $t\in I$ the estimate,
$$
\sup_{t\in I}  \bigg\|\prod_{j=1}^k (-\Delta_{x_j}+1)^{s/2} (-\Delta_{x'_j}+1)^{s/2}\,\gamma^{(k)}(t) \bigg\|_{ \mathscr{L}^1(
L^2(\R^{dk}))}\leq R^{2k}\,.
$$
This is exactly the second assumption in  \eqref{A2}, with the choice \eqref{spec-fram}, expressed in terms of trace-class operators.  Moreover, if we consider  $\tilde \gamma^{(k)}(t):= \mathcal{U}^{(k)}(t)\gamma^{(k)}(t)$ then the integral equation \eqref{hier-kernel-int} gives,
\begin{equation}
\label{equiv-hier}
 \tilde\gamma^{(k)}(t)= \tilde\gamma^{(k)}(0)+ \int_{0}^t  \,  \sum_{j=1}^{k}  (C_{j,k}^{+} \tilde\gamma_{\tau}^{} +  C_{j,k}^{-} \tilde\gamma_{\tau}^{}) \;d\tau\,,
\end{equation}
where $C_{j,k}^{\pm}$ are the operations defined in \eqref{defBjk}. In fact, a simple computation yields
\begin{eqnarray*}
\langle \psi, C^+_{j,k}\tilde \gamma(\tau) \,\phi\rangle_{L^2(\R^{dk})}&= &
\int_{H^1(\R^d)}
\langle\psi  | x^{\otimes k} \rangle \langle x^{\otimes j-1} \otimes v(\tau,x)\otimes x ^{\otimes k-j}| \phi\rangle \;d\tilde\mu_{\tau}(x)
\\
&=& -i
\int_{H^1(\R^d)}
\langle \mathcal{U}(\tau)^{\otimes k}\psi  |(\mathcal{U}(\tau)x)^{\otimes k} \rangle \times
\\
&&\hspace{.6in}\langle (\mathcal{U}(\tau)x)^{\otimes j-1} \otimes g( \mathcal{U}(\tau)x)\otimes (\mathcal{U}(\tau)x) ^{\otimes k-j}|  \mathcal{U}(\tau)^{\otimes k}\phi\rangle \;d\tilde\mu_{\tau}(x)
\\
&=& -i
\int_{H^1(\R^d)}
\langle \mathcal{U}(\tau)^{\otimes k}\psi  |y^{\otimes k} \rangle \langle y^{\otimes j-1} \otimes g(y)\otimes y ^{\otimes k-j}|  \mathcal{U}(\tau)^{\otimes k}\phi\rangle \;d\mu_{\tau}(y)\,.
\end{eqnarray*}
Here, $\tilde\mu_t=\Phi^{-1}((\tilde \gamma^{(k)})_{k\in\N}) $ and $\mu_t=\Phi^{-1}(( \gamma^{(k)})_{k\in\N})$ where
$\Phi$ is the homeomorphism of Corollary \ref{homeo}.
On the other hand, using \eqref{Bjk-1}-\eqref{Bjk-2} one obtains
\begin{eqnarray*}
\langle \psi, \mathcal{U}^{(k)}(-\tau)B^+_{j,k} \gamma^{(k+1)}(\tau) \,\phi\rangle_{L^2(\R^{dk})}&= &
\langle \mathcal{U}(\tau)^{\otimes k}\psi, B^+_{j,k} \gamma^{(k+1)}(\tau) \,\mathcal{U}(\tau)^{\otimes k}\phi\rangle_{L^2(\R^{dk})}\,,
\end{eqnarray*}
and a simple computation gives the following kernel-operator identification
$$
B^+_{j,k} \gamma^{(k+1)}(\tau) \equiv \int_{H^1(\R^d)}
\big|y^{\otimes k} \rangle \langle y^{\otimes j-1} \otimes g(y)\otimes y ^{\otimes k-j}\big|   \;d\mu_{\tau}(y)
\in\mathscr{L}_{-\sigma}^1(L_s^2(\R^{dk}))\,.
$$
So that, one proves
$$
-i\;\mathcal{U}^{(k)}(-\tau)B^+_{j,k} \gamma^{(k+1)}(\tau) \equiv C^+_{j,k}\tilde \gamma(\tau) \,.
$$
This shows the equivalence between  the two formulations  \eqref{equiv-hier} and \eqref{hier-kernel-int} .

\subsection{Characteristic equation}

In the sequel, we prove that the Liouville equation \eqref{eq.transport}  is equivalent to another simpler  formulation in terms of characteristic functions. This is based on the elementary fact that probability  measures
can be characterized  by their characteristic functions.

\begin{prop}
\label{lm.cha}
Let $v:\R\times \zeds\to\zedsi$ a vector field satisfying \eqref{A0} and  $(\mu_{t})_{t \in I}$ a curve in $\mathscr{M}(\mathscr{Z}_0)$ satisfying the assumption  \eqref{A1}. Then the following assertions are equivalent:
\begin{itemize}
\item[(i)] $(\mu_{t})_{t \in I}$ is a solution of the Liouville equation \eqref{eq.transport}.
\item[(ii)] $(\mu_{t})_{t \in I}$ solves the following characteristic equation, i.e.: $\forall t \in I ~,~ \forall y \in \mathscr{Z}_{\sigma} $,
\begin{equation}
\label{K}
%\tag{{\bf{Charachteristic}}
  \mu_{t}(e^{2i \pi \Re \langle y,. \rangle_{\zed}})
  =  \mu_{t_{0}}(e^{2i \pi \Re \langle y,. \rangle_{\zed}})
  + 2i\pi \int_{t_{0}}^{t} \mu_{\tau}\big( e^{2i\pi
   \Re \langle y,x \rangle_{\zed}} \;\Re \langle v(\tau,x);y \rangle_{\mathscr{Z}_{0}}\big)\,d\tau\,,
\end{equation}
\end{itemize}
where we have used the notation $\mu_{t}(e^{2i \pi Re \langle y,. \rangle_{\zed}}) =
 \displaystyle\int_{\zeds} e^{2i\pi \Re \langle y,x \rangle_{\zed}} \,d\mu_{t}(x)$.
\end{prop}

\begin{proof}
We suppose that $(\mu_{t})_{t \in I}$ is a solution of the Liouville equation \eqref{eq.transport} satisfying the assumption \eqref{A1}. Consider a test function $\varphi(t,x) = \chi(t) \varphi_{m}(x)$, with $\chi \in \mathscr C_{0}^{\infty}(I)$ and $\varphi_{m}$ is given by:
\[
\varphi_{m}(x) = \cos(2\pi \Re \langle z,x \rangle_{\zedsi}) \;\psi(\frac{\Re \langle z,x \rangle_{\zedsi}}{m}) ~,
\]
for some $z \in \zedsi$ fixed and $ \psi \in \mathscr C_{0}^{\infty}(\mathbb{R}) $ such that
$0\leq \psi \leq 1$  and equal to $1$ in a neighbourhood of $0$.  So that, the functions $\varphi_{m}$ converges pointwisely  to $\cos(2\pi \Re \langle z,. \rangle_{\zedsi})$ when $m$ tend to $+ \infty$.  As $(\mu_{t})_{t \in I}$ satisfies the Liouville equation \eqref{eq.transport} and $\varphi\in \mathscr C_{0,cyl}^{\infty}(I\times \zedsi)$, one can write:
\[
 \int_{I}\int_{\zeds} \chi^{'}(t)\varphi_{m}(x) + \Re \langle v(t,x); \nabla \varphi_{m}(x) \rangle_{\zedsi}
 \;\chi(t)\, d\mu_{t}(x)dt=0.
 \]
A simple computation yields the gradient of $\varphi_{m}$,
\begin{eqnarray*}
\nabla \varphi_{m}(x) &=& -2\pi \sin(2\pi \Re \langle z,x \rangle_{\zedsi})\psi(\frac{\Re \langle z,x \rangle_{\zedsi}}{m})\cdot z+ \\&&\cos(2\pi \Re\langle z,x \rangle_{\zedsi}) \frac{1}{m} \,\psi^{'}( \frac{\Re \langle z,x \rangle_{\zedsi}}{m}) \cdot z\in \zedsi.
\end{eqnarray*}
Then, if we replace this writing of $\nabla \varphi_{m}$ in the previous integral, we have, using the dominated convergence theorem and the Fubini's theorem,
\[
 \int_{I} \chi^{'}(t)\int_{\zeds} \cos(2\pi \Re\langle z,x \rangle_{\zedsi}) d\mu_{t}(x)dt = \int_{I} \chi(t)\int_{\zeds}  2\pi \sin(2\pi \Re \langle z,x \rangle_{\zedsi}) \Re \langle v(t,x); z \rangle_{\zedsi} d\mu_t(x)dt .
 \]
In the same way (with $\varphi_{m}(x) = \sin(2\pi \Re \langle z,x \rangle_{\zedsi}) \psi( \frac{\Re \langle z,x \rangle_{\zedsi}}{m})$), we have a similar identity:
\[
 \int_{I} \chi^{'}(t)\int_{\zeds} \sin(2\pi \Re\langle z,x \rangle_{\zedsi}) d\mu_{t}(x)dt =- \int_{I} \chi(t)\int_{\zeds}  2\pi \cos(2\pi \Re \langle z,x \rangle_{\zedsi}) \Re \langle {v}(t,x); z \rangle_{\zedsi} d\mu_tdt .
 \]
And if one sums the first integral with $i$ times the second one, one obtains:
\[
  \int_{I} \chi^{'}(t)\int_{\zeds} e^{2i\pi \Re \langle z,x \rangle_{\zedsi}} d\mu_{t}(x)dt = -2i\pi \int_{I} \chi(t)\int_{\zeds} \Re \langle {v}(t,x); z \rangle_{\zedsi} e^{2i \pi \Re \langle x,z \rangle_{\zedsi}}  d\mu_t(x)dt.
  \]
Set
\[
m(t) := \int_{\zeds} e^{2i\pi \Re \langle z,x \rangle_{\zedsi}} d\mu_{t}(x).
\]
Then the previous equation becomes, in a distributional sense,
\[ \frac{d}{dt} m(t) = 2i \pi \int_{\zeds} \Re \langle{v}(t,x) ; z \rangle_{\zedsi} e^{2i\pi \Re \langle z,x \rangle_{\zedsi}}  d\mu_{t}(x).  \]
Since  $m  \in L^{1}(I,dt)$ and $m{'} \in L^{1}(I,dt)$, then $m \in W^{1,1}(I,\mathbb{R})$. This proves that $m$ is absolutely continuous over $I$ and so the fundamental theorem of analysis holds true for $m$, i.e,
\[
\forall (t,t_{0}) \in I^{2} ~,~ m(t) = m(t_{0}) + \int_{t_{0}}^{t}m{'}(s) ds.
\]
Rewriting the latter  equality,  one essentially obtains the characteristic equation for all $ z \in \zedsi$:
\begin{eqnarray*}
\int_{\zeds} e^{2i\pi \Re \langle z,x \rangle_{\zedsi}} d\mu_{t}(x) &=& \int_{\zeds} e^{2i\pi \Re \langle z,x \rangle_{\zedsi}} d\mu_{t_{0}}(x) +\\&& 2i\pi \int_{t_{0}}^{t} \int_{\zeds} \Re \langle {v}(\tau,x) ; z \rangle_{\zedsi}  e^{2i\pi \Re \langle z,x \rangle_{\zedsi}} d\mu_{t}(x)dt.
\end{eqnarray*}
The scalar product   $\langle z,x \rangle_{\zedsi} = \langle A^{-\sigma}z,x \rangle_{\zed}$. So that, if we set $y = A^{-\sigma}z \in \mathscr{Z}_{\sigma}$, we have:
\begin{equation*}
\int_{\zeds} e^{2i\pi \Re \langle y,x \rangle_{\zed}} d\mu_{t}(x) = \int_{\zeds} e^{2i\pi \Re \langle y,x \rangle_{\zed}} d\mu_{t_{0}}(x) + 2i\pi \int_{t_{0}}^{t} \int_{\zeds} \Re \langle {v}(\tau,x) ; y \rangle_{\zed}  e^{2i\pi \Re \langle y,x \rangle_{\zed}} d\mu_{\tau}(x)d\tau.
\end{equation*}

\bigskip
 Conversely, suppose that the measures $(\mu_{t})_{t \in I}$ satisfy the characteristic equation \eqref{K}. Let $\psi \in \mathscr C_{0,cyl}^{\infty}(\zedsi)$,  then we can write $\psi(x) \equiv \phi(p(x))$ where $p$ is an orthogonal  projection  on a finite dimensional  subspace of $\zedsi$ and $\phi \in \mathscr C_{0}^{\infty}(p(\zedsi))$.  As $\psi$ is real-valued, one can write using
 inverse Fourier transform,
\begin{equation}
\label{fourier-inv}
 \psi(x) = \int_{p(\zedsi)} \cos(2\pi \Re \langle z;x \rangle_{\zedsi}) \mathscr{F}_{R}(\psi)(z) + \sin(2\pi \Re \langle z;x \rangle_{\zedsi})\mathscr{F}_{I}(\psi)(z) dL(z)\,,
\end{equation}
where $dL$ denotes the Lebesgue mesure on $p(\zedsi)$ and
\begin{eqnarray*}
\mathscr{F}_{R}(\psi)(z) &= & \int_{p(\zedsi)} \cos(2\pi \Re \langle z;x \rangle_{\zedsi}) \;\psi(x) \;dL(x),
\\
\mathscr{F}_{I}(\psi)(z) &= &\int_{p(\zedsi)} \sin(2\pi \Re \langle z;x \rangle_{\zedsi}) \;\psi(x) \;dL(x)\,.
\end{eqnarray*}
Splitting  the characteristic equation \eqref{K} into real and imaginary part, yields:
\begin{eqnarray*}
A =  \int_{\zedsi} \cos (2\pi \Re \langle z;x \rangle_{\zedsi}) d\mu_{t}(x) &=&  \int_{\zedsi} \cos (2\pi \Re \langle z;x \rangle_{\zedsi}) d\mu_{t_{0}}(x) + \\
&& \hspace{-.8in}\int_{t_{0}}^{t} \int_{\zedsi} \Re \langle v(\tau,x);z \rangle_{\zedsi} (-2\pi \sin(2\pi \Re \langle z;x \rangle_{\zedsi})d\mu_{\tau}(x)d\tau,
\end{eqnarray*}
\begin{eqnarray*}
 B = \int_{\zedsi} \sin (2\pi \Re \langle z;x \rangle_{\zedsi}) d\mu_{t}(x) &= & \int_{\zedsi} \sin(2\pi \Re \langle z;x \rangle_{\zedsi}) d\mu_{t_{0}}(x) + \\
&&  \hspace{-.8in}\int_{t_{0}}^{t} \int_{\zedsi} \Re \langle v(\tau,x);z \rangle_{\zedsi} (2\pi \cos(2\pi \Re \langle z;x \rangle_{\zedsi})d\mu_{\tau}(x)d\tau.
\end{eqnarray*}
Now, a computation of $\displaystyle\int_{p(\zedsi)} (\mathscr{F}_{R}(\psi) \times A + \mathscr{F}_{I}(\psi) \times B) \,dL(z)$ gives:
\[
 \int_{\zedsi} \psi(x) d\mu_{t}(x) = \int_{\zedsi} \psi(x) d\mu_{t_{0}}(x) + \int_{t_{0}}^{t} \int_{\zedsi} \Re \langle v(\tau,x); \nabla \psi(x) \rangle_{\zedsi} d\mu_{\tau}(x)d\tau,
  \]
 with the last term in the above right hand side is obtained using the Fourier inverse formula,
 \begin{eqnarray*}
 \nabla\psi(x)[u]&=& \int_{p(\zedsi)} \bigg(2\pi \Re\langle u,z\rangle_{\zedsi} \; \cos(2\pi \Re \langle z;x \rangle_{\zedsi})\mathscr{F}_{I}(\psi)(z)  -\\
 &&\hspace{1in}2\pi \Re\langle u,z\rangle_{\zedsi} \;  \sin(2\pi \Re \langle z;x \rangle_{\zedsi}) \mathscr{F}_{R}(\psi)(z) \bigg)\,dL(z)\,,
 \end{eqnarray*}
 for any $x,u\in p(\zedsi)$.  This equality shows that, in the distributional sense, we have
\[
\frac{d}{dt} \int_{\zedsi} \varphi(x) d\mu_{t}(x) = \int_{\zedsi} \Re \langle v(t,x); \nabla \varphi(x) \rangle_{\zedsi} \;d\mu_{t}(x).
 \]
 Multiplying by $\chi(t)$, with $\chi\in \mathscr C_{0}^{\infty}(I)$, the latter equality and integrating by part the right hand side,    one  concludes that $(\mu_{t})_{t \in I}$ satisfies the Liouville equation \eqref{eq.transport}  for any test functions   of the form  $\varphi(t,x)\equiv\chi(t) \psi(x)$ with $\psi \in \mathscr C_{0,cyl}^{\infty}(\zedsi)$. Then using a standard density argument one obtains the Liouville  equations \eqref{eq.transport}  for any $\varphi\in
\mathscr C_{0,cyl}^{\infty}(I\times \zedsi)$.
\end{proof}

\subsection{Duality}
Once we have defined the characteristic equation \eqref{K} and proved its equivalence with the Liouville equation, we proceed to the proof of the duality  between the hierarchy equations \eqref{hier} and the Liouville equations \eqref{eq.transport}  stated in Theorem  \ref{sec.0.thm1}.

\begin{prop}
\label{thm.main}
Let ${v}: \mathbb{R} \times \zeds \mapsto \zedsi$ be a vector field satisfying \eqref{A0} and  $t \in I \to \mu_{t}$ a curve in $\in \mathscr{M}(\zed)$ verifying \eqref{A1} and solving the Liouville equation \eqref{eq.transport}. Then $t\in I\to\gamma_t=\Phi(\mu_{t})$ is a curve in
$\mathscr{H}(\mathscr{\zed})$ satisfying \eqref{A2} and solving the symmetric hierarchy equation \eqref{int-hier}.
\end{prop}
\begin{proof}
Remember that Lemma \ref{a1a2} says that the assumptions  \eqref{A1}  and \eqref{A2} are equivalent. So, it is enough to prove that $t\to \gamma_t=\Phi(\mu_{t})$ solves the hierarchy equation \eqref{int-hier}, i.e.,
\[
\forall t \in I ~,~ \gamma_{t}^{(k)} = \gamma_{t_{0}}^{(k)} + \int_{t_{0}}^{t} \sum_{j=1}^{k}  (C_{j,k}^{+} \gamma_{\tau}^{} +  C_{j,k}^{-} \gamma_{\tau}^{}) d\tau \;\in \mathscr{L}^1_{-\sigma}(\vee^k\zed) ~.~
\]
According to Proposition \ref{lm.cha},  $(\mu_{t})_{t \in I}$  satisfies the characteristic equation \eqref{K}. So, for any $y \in \mathscr{Z}_{\sigma}$, we have:
\[
\frac{d}{dt} \mu_{t}(e^{2i\pi \Re \langle y;x \rangle_{\zed}}) = 2i \pi \displaystyle\int_{\zeds} \Re \langle v(t,x); y \rangle_{\zed}\, e^{2i\pi \Re \langle y;x \rangle_{\zed}} d\mu_{t}(x) ~,~ ~a.e.~ t \in I.  \]

We  use the analyticity of the function $ \lambda \to  \mu_{t}(e^{2i\pi \Re \langle y;x \rangle_{\zed}. \lambda})$ in order to obtain equalities between different order derivatives. Indeed, since $\mu_{t}(B_{\zed}(0,1))=1$, there exists a constant $C >0$ such that
\[
\displaystyle\int_{\zeds} |  \Re \langle y;x \rangle_{\zed} |^{n} d\mu_{t}(x) \leq || y ||_{\mathscr{Z}_{0}}^{n} \displaystyle\int_{\zeds} ||x||_{\zed}^{n} d\mu_{t}(x) \leq C^{n}.
\]
One also obtains a similar estimate,
\begin{eqnarray*}
 \int_{I} \int_{\zeds} | \Re \langle v(t,x) ; y \rangle_{\zed} |.| \Re \langle y,x \rangle_{\zed} |^{n} d\mu_{t}(x)dt &\leq& \int_{I}\int_{\zeds} ||v(t,x)||_{\zedsi}\,||x||_{\zed}^{n} d\mu_{t}(x)dt \times ||y||_{\mathscr{Z}_{\sigma}}^{n+1}\,, \\
&\leq& ||y||_{\mathscr{Z}_{\sigma}}^{n+1} \int_{I} \int_{\zeds} ||v(t,x)||_{\zedsi}d\mu_{t}(x)dt \leq C^{n+2}\,,
\end{eqnarray*}
since $I$ is bounded, $v$ satisfies \eqref{A0} and $\mu_t$ concentrates on $B_{\zeds}(0,R)$ for some constant  $R>0$ independent of time. Hence, with these inequalities, we can write, since $(\mu_{t})_{t \in I}$ satisfies \eqref{K}:
\begin{eqnarray*}
\int_{\zeds} \sum_{n=0}^{\infty} \frac{(2i\pi \Re \langle y,x \rangle_{\zed})^{n}\lambda^{n}}{n!}d\mu_{t}(x) &=& \int_{\zeds} \sum_{n=0}^{\infty} \frac{(2i\pi \Re \langle y,x \rangle_{\zed})^{n}\lambda^{n}}{n!}d\mu_{t_{0}}(x) +
\\
&& \int_{t_{0}}^{t} \int_{\zeds} \Re \langle v(\tau,x) ; \lambda y \rangle_{\zed} \sum_{n = 0}^{\infty} \frac{(2i\pi \Re \langle y,x \rangle_{\zed})^{n}\lambda^{n}}{n!} d\mu_{\tau}(x)dt\,.
\end{eqnarray*}
Part by part, this gives the equality:
\begin{eqnarray*}
\int_{\zeds} \frac{(2i\pi)^{n}}{n!}(\Re \langle y,x \rangle _{\zed})^{n} d\mu_{t}(x) &=& \int_{\zeds}  \frac{(2i\pi)^{n}}{n!}(\Re \langle y,x \rangle_{\zed} )^{n} d\mu_{t_{0}}(x) +
\\
 &&2i\pi \int_{t_{0}}^{t} \int_{\zeds} \Re \langle v(\tau,x);y \rangle_{\zed} \frac{(2i\pi)^{n-1}}{(n-1)!} (\Re \langle y,x
 \rangle_{\zed} )^{n-1} d\mu_{\tau}(x)d\tau
\end{eqnarray*}
Writing
$$
(\Re \langle y,x \rangle_{\zed} )^{n} = \frac{1}{2^{n}} \sum_{k=0}^{n} C_{n}^{k} \langle y,x \rangle_{\zed}^{k} \langle x,y \rangle_{\zed}^{n-k}\,,
$$
and noticing that the $U(1)$-invariance of the measures $\mu_t$ yields for any $0\leq k\leq n$ except $k=n/2$,
$$
\int_{\zeds} \langle y,x \rangle^{k} \langle x,y \rangle^{n-k} \;d\mu_{t}(x) = 0\,;
$$
then one concludes that in the case $n=2k$,
\[
\int_{\zeds} (\Re \langle y,x \rangle)^{n} \;d\mu_{t}(x) = \frac{1}{2^{2k}} \int_{\zeds} C_{2k}^{k} |
\langle y,x \rangle |^{2k} d\mu_{t}(x)\,.
\]
This gives
\begin{eqnarray*}
\frac{1}{2^{2k}} \int_{\zeds} C_{2k}^{k} | \langle y,x \rangle |^{2k} d\mu_{t}(x) &=& \frac{1}{2^{2k}} \int_{\zeds} C_{2k}^{k} | \langle y,x \rangle |^{2k} d\mu_{t_{0}}(x) +
\\
 &&\hspace{-.6in}\int_{t_{0}}^{t} \int_{\zeds} \Re \langle v(\tau,x) ; y \rangle_{\zed} \bigg( \frac{2k}{2^{2k-1}} \sum_{j=0}^{2k-1} C_{2k-1}^{j} \langle y,x \rangle^{j} \langle x,y \rangle^{2k-1-j}\bigg) d\mu_{\tau}(x)d\tau \,.
\end{eqnarray*}
The last term in  the above equality can also be simplified thanks the $U(1)$-invariance of the vector field
$v$ and the measures $\mu_t$. Indeed, if we develop
\[
\Re \langle v(\tau,x) ; y \rangle_{\zed} = \frac{1}{2} \langle v(\tau,x) , y \rangle_{\zed} +\frac{1}{2}  \langle y, v(\tau,x) \rangle_{\zed},
\]
and  write,
\begin{eqnarray}
\int_{\zeds} \Re \langle v(\tau,x) ; y \rangle_{\zed} \bigg( \sum_{j=0}^{2k-1} C_{2k-1}^{j} \langle y,x \rangle_{\zed}^{j} \langle x,y \rangle_{\zed}^{2k-1-j}\bigg) d\mu_{\tau}(x) &=& \nonumber\\ \label{equiv.eq1}
 &&\hspace{-2.4in} \frac{1}{2} \sum_{j=0}^{2k} C_{2k-1}^{j} \int_{\zeds} \langle v(\tau,x),y \rangle_{\zed} \langle x,y \rangle_{\zed}^{2k-1-j}\langle y,x \rangle_{\zed}^{j}d\mu_{\tau}(x)
\\ \label{equiv.eq2}
 &&\hspace{-2.4in}+ \frac{1}{2} \sum_{j=0}^{2k} C_{2k-1}^{j} \int_{\zeds} \langle y,v(\tau,x) \rangle_{\zed} \langle x,y \rangle_{\zed}^{2k-1-j}\langle y,x \rangle_{\zed}^{j}d\mu_{\tau}(x)\,,
\end{eqnarray}
then using the gauge invariance, one notices that  the sum in \eqref{equiv.eq1} reduces to $j=k$ while
the sum \eqref{equiv.eq2} reduces to $j=k-1$. So that, we  have:
\[
 \int_{\zeds} \Re \langle {v}(\tau,x) ; y \rangle_{\zed} ( \sum_{j=0}^{2k-1} C_{2k-1}^{j} \langle y,x \rangle_{\zed}^{j} \langle x,y \rangle_{\zed}^{2k-1-j}) d\mu_{\tau}(x) =
 \]
 \[
  \frac{1}{2} C_{2k-1}^{k} \int_{\zeds} \langle {v}(\tau,x) ; y \rangle_{\zed}  \langle x,y \rangle_{\zed}^{k-1} \langle y,x \rangle_{\zed}^{k} d\mu_{\tau}(x) +  \frac{1}{2} C_{2k-1}^{k-1} \int_{\zeds} \langle y, {v}(\tau,x)  \rangle_{\zed} \langle x,y \rangle_{\zed}^{k} \langle y,x \rangle_{\zed}^{k-1} d\mu_{\tau}(x)\,.
 \]
And then, we can finally write:
\begin{eqnarray*}
\int_{\zeds} | \langle y,x \rangle |^{2k} d\mu_{t}(x) &= & \int_{\zeds} | \langle y,x \rangle |^{2k} d\mu_{t_{0}}(x) +  \\
&& \int_{t_{0}}^{t} \int_{\zeds}\bigg( \frac{2k C_{2k-1}^{k}}{C_{2k}^{k}} \langle y^{\otimes k};x^{\otimes k} \rangle_{\zed} \langle {v}(\tau,x) \otimes x^{\otimes (k-1)}; y^{\otimes k} \rangle_{\zed} +\\&&\hspace{.7in} \frac{2k C_{2k-1}^{k-1}}{C_{2k}^{k}} \langle y^{\otimes k}; {v}(\tau,x) \otimes x^{\otimes (k-1)} \rangle_{\zed} \langle x^{\otimes k}; y^{\otimes k} \rangle_{\zed}\bigg) \;d\mu_{\tau}(x)dt
\end{eqnarray*}
Checking that
\[
\frac{2k C_{2k-1}^{k}}{C_{2k}^{k}} = k \qquad \text{ and } \qquad  \frac{2k C_{2k-1}^{k-1}}{C_{2k}^{k}} = k\,,
\]
and using the integral representation of $\gamma_t$ given by $\gamma_t=\Phi(\mu_{t})$, the last equality can be written as,
\begin{equation}
\label{equiv.eq4}
\begin{aligned}
 \langle y^{\otimes k}, \gamma_{t}^{(k)} y^{\otimes k} \rangle_{\zed} &=\langle y^{\otimes k}, \gamma_{t_{0}}^{(k)} y^{\otimes k} \rangle _{\zed}+ \\
 &\int_{t_{0}}^{t} \sum_{j=1}^{k} \int_{\zeds} \bigg(  \langle y^{\otimes k} | x^{\otimes (j-1)} \otimes {v}(\tau,x)
 \otimes x^{\otimes (k-j)}\rangle_{\zed}  \langle x^{\otimes k} | y^{\otimes k}\rangle_{\zed}+\\& \hspace{.9in} \langle y^{\otimes k} | x^{\otimes k} \rangle_{\zed} \langle x^{\otimes (j-1)} \otimes {v}(\tau,x)\otimes x^{\otimes (k-j)} | y^{\otimes k} \rangle_{\zed} \bigg) \;d\mu_{\tau}(x) d\tau
 \end{aligned}
\end{equation}
So,  one recognizes the operations $C_{j,k}^\pm$ given by \eqref{defBjk},
\begin{equation}
\label{equiv.eq3}
 \langle y^{\otimes k}, \gamma_{t}^{(k)} \,y^{\otimes k} \rangle_{\zed} = \langle y^{\otimes k}, \big( \gamma_{t_{0}}^{(k)}  +\int_{t_{0}}^{t} \sum_{j=1}^{k} \int_{\zeds} C_{j,k}^{+} \gamma_{\tau}^{} + C_{j,k}^{-} \gamma_{\tau}^{} d\mu_{\tau}(x)d\tau \big)  \,y^{\otimes k} \rangle_{\zed} \,.
\end{equation}
Notice that the assumptions \eqref{A2}-\eqref{A1} are satisfied  and $y\in\mathscr{Z}_\sigma$,  so all the above calculations are well justified. To conclude, one just have to use a  polarization formula  plus a standard density argument. Indeed, the identity \eqref{equiv.eq3} extends to  $\langle \eta^{\otimes k} , \gamma_{t}^{(k)} y^{\otimes k} \rangle$ for all $\eta, y\in\mathscr{Z}_{\sigma}$ using
\[
\langle \eta^{\otimes k} , \gamma_{t}^{(k)} y^{\otimes k} \rangle = \int_{0}^{1} \int_{0}^{1} \langle (e^{2i\pi \theta} \eta + e^{2i\pi \varphi}y)^{\otimes k} , \gamma_{t}^{(k)} (e^{2i\pi \theta} \eta + e^{2i\pi \varphi}y)^{\otimes k} \rangle e^{2i\pi	(k\theta-k\varphi)} d\theta d\varphi \,.
\]
Now, since ${\rm Vect}\{\eta^{\otimes k}, \eta\in \mathscr{Z}_\sigma\} $ is dense subspace of  $\mathcal{D}_{\sigma}^{(k)}= D((A^{\sigma/2})^{\otimes k})$ endowed with the graph norm,   one obtains the hierarchy equation
\eqref{int-hier} as an integral equation valued in $\mathscr{L}^1_{-\sigma}(\vee^k\zed)$.
\end{proof}

\begin{prop}
\label{thm.main2}
Let ${v}: \mathbb{R} \times \zeds \mapsto \zedsi$ be a vector field satisfying \eqref{A0} and $t\in I\to\gamma_t$ is a curve in $\mathscr{H}(\mathscr{\zed})$ satisfying \eqref{A2} and solving the symmetric hierarchy equation \eqref{hier}. Then  $t \in I \to \mu_{t}=\Phi^{-1}(\gamma_{t})$ is a curve in $\in \mathscr{M}(\zed)$ verifying \eqref{A1} and solving the Liouville equation \eqref{eq.transport}.
\end{prop}

\begin{proof}
Thanks to the concentration of $\mu_{t}$ over a bounded set of $\zeds$ (i.e., $\mu_t(B_{\zeds}(0,R))=1$), the following series expansion of $  \displaystyle\int_{\zedsi} e^{2i\pi \Re \langle x,z\rangle_{\zedsi}} d\mu_{t}(x) $ holds true,
\[
\displaystyle\int_{\zedsi} e^{2i\pi \Re \langle x,z \rangle_{\zedsi}} d\mu_{t}(x) = \sum_{k \geq 0} \frac{(2i\pi)^{k}}{k!} \sum_{j=0}^{k} \frac{C_{j}^{k}}{2^{k}} \displaystyle\int_{\zedsi} \langle x,z \rangle_{\zedsi}^{j} \langle z,x \rangle_{\zedsi}^{k-j} d\mu_{t}(x) ~.~
\]
Indeed,  we have the estimate,
\[
\displaystyle\int_{\zedsi} |\langle x,z \rangle_{\zedsi}^{2k}| d\mu_{t}(x) \leq ||z||_{\zedsi}^{2k} \displaystyle\int_{\zedsi} C^{2k} ||x||_{\zeds}^{2k} d\mu_{t}(x)
\leq  ||z||_{\zedsi}^{2k}  C^{2k} R^{2k}\,.
\]
Taking into account the $U(1)$- invariance of the measures $\mu_{t}$, one shows
\[
\displaystyle\int_{\zedsi} e^{2i \pi\Re \langle  z,x \rangle_{\zedsi}} d\mu_{t}(x) = \sum_{k \geq 0} \frac{(-1)^{k}\pi^{2k}}{(k!)^{2}} \displaystyle\int_{\zedsi} |\langle z,x \rangle_{\zedsi}^{2k}| d\mu_{t}(x) \,.
\]
Taking $y=A^{-\sigma/2} z\in \mathscr{Z}_{\sigma}$ hence
 $\langle z, x\rangle_{\zedsi}=\langle y, x\rangle_{\zed}$,  and consequently
\[
\displaystyle\int_{\zedsi} e^{2i \pi\Re \langle  z,x \rangle_{\zedsi}} d\mu_{t}(x) = \sum_{k \geq 0} \frac{(-1)^{k}\pi^{2k}}{(k!)^{2}}  \,\langle y^{\otimes k}, \gamma_{t}^{(k)} \,y^{\otimes k} \rangle_{\zed}  \,,
\]
with $\gamma_{t}$  satisfying the symmetric hierarchy equation \eqref{int-hier} and having the integral representation given by $\gamma_{t}=\Phi(\mu_{t})$. This leads us to the equality \eqref{equiv.eq4}.  As the computation of the last proof can be done reversely, we conclude that the the set of measures $(\mu_{t})_{t \in I}$  satisfies the characteristic equation  \eqref{K}  and therefore the Liouville \eqref{eq.transport} according to Prop.~\ref{lm.cha}.
\end{proof}

\section{Uniqueness and existence principles }

As explained in the introduction the duality between the hierarchy and the  Liouville  equations allows us to benefit from the recent advances in measure transportation  theory, see e.g. \cite{MR2400257,MR2439520,
MR2129498,MR2335089} .  In particular, the questions of uniqueness for continuity equations in finite dimension
are by now quite well understood through either the method of characteristics or a superposition principle   \cite{MR2668627}. The latter approach is very powerful  and it is based in a sort of probabilistic representation
of solutions of a continuity equation.   In particular, the two first authors consider in \cite{MR3721874}  the question of well-posedness of general Liouville equations related to  nonlinear  PDEs. In this section, we will  improve
the results of \cite{MR3721874} .

\subsection{Probabilistic representation}
\label{subsec.probrep}
We recall a  powerful  probabilistic representation for solutions of the Liouville's equation \eqref{eq.transport} given in Prop.~\ref{prob-rep} below and proved in a previous work of the two first authors in \cite{MR3721874}.  This result was inspired and stimulated by the ones proved  in finite dimension  in \cite[Theorem 4.1]{MR2335089} and \cite[Theorem 8.2.1 and 8.3.2]{MR2129498} and the one in infinite dimension proved in \cite[Proposition C.2]{MR3379490}.

\bigskip
Recall that $(\zeds,\zed,\zedsi)$  is the triple of spaces introduced in Section \ref{fram} with $0\leq s\leq \sigma$ and $I$ is always considered as  bounded open interval.  We denote by
\begin{equation}
\label{eq.X}
\mathfrak X=\zedsi\times  \mathscr{C}(\bar I, \zedsi) \,,
\end{equation}
and endow such a product space with the   following norm
\begin{equation}
\label{normX}
||(x,\varphi)||_{\mathfrak X}= ||x||_{\mathscr{Z}_{-\sigma}}+\sup_{t\in \bar I}||\varphi(t)||_{\mathscr{Z}_{-\sigma}}\, .
\end{equation}
 Here $\zedsi$ is  considered as a real Hilbert space. For each $t \in I$, we define the evaluation map over the space $\mathfrak X$  as,
$$
e_{t}:(x,\varphi) \in \mathfrak X \longmapsto \varphi(t) \in \zedsi\,.
$$

\begin{prop}
\label{prob-rep}
 Let $v:\R\times \mathscr Z_s\to \mathscr Z_{-\sigma}$ be  a Borel  vector field such that $v$ is bounded on bounded sets. Let $t\in I\to\mu_{t}\in \mathfrak{P}(\mathscr Z_s)$ be a weakly narrowly continuous solution in $\mathfrak{P}(\mathscr Z_{-\sigma})$ of the Liouville equation \eqref{eq.transport} defined on an open bounded interval $I$. Then there exists  $\eta$ a Borel probability measure, on the space $(\mathfrak X, ||\cdot||_{\mathfrak X})$ given in \eqref{eq.X}, satisfying:
\begin{enumerate}[label=\textnormal{(\roman*)}]
  \item  \label{concen-item1}$\eta$ is concentrated on the set of points $(x,\gamma)\in \mathfrak X$ such that the curves  $\gamma\in   W^{1,1}(I, \mathscr Z_{-\sigma})$  and are solutions of the initial value problem $\dot\gamma(t)= v(t,\gamma(t))$  for a.e. $t\in I$ and $\gamma(t)\in \mathscr Z_s$ for a.e. $t\in I$ with $\gamma(t_0)=x\in \mathscr Z_s$ for some fixed  $t_0\in I$.
 \item \label{concen-item2} $\mu_t=(e_{t})_{\sharp}\eta$ for any $t\in I$.
\end{enumerate}
\end{prop}

\begin{remark}
Proposition \ref{prob-rep} is proved in \cite[Prop.~4.1]{MR3721874}. However, some slight differences between the two statements may catch the reader's attention.  So, we explain how the Prop.~\ref{prob-rep} is a straightforward reformulation of \cite[Prop.~4.1]{MR3721874}. In fact, there are two points:
\begin{itemize}
\item In \cite[Prop.~4.1]{MR3721874}, an abstract  rigged Hilbet space $(\mathscr{Z}_{1},\zed,\mathscr{Z}'_{1})$ is used. With the framework here we are allowed to take $\mathscr{Z}_{1}\equiv \mathscr{Z}_{s}$,
$\mathscr{Z}_{0}\equiv \mathscr{Z}_{\frac{s-\sigma}{2}}$ so that $\mathscr{Z}'_{1}$ identifies with $\zedsi$.
\item The space $\mathfrak X$ in  \cite[Prop.~4.1]{MR3721874} is equipped with a different norm  from the one used here. However, it is easy to see that  the Borel sets of $\mathfrak X$ are the same for both norms (see  Lemma \ref{tribu} in the Appendix C).
\end{itemize}
\end{remark}

As one can see below the existence of such measure $\eta$ has important implications. In particular, the existence
of well defined flow for the initial value problem \eqref{IVP}.  For this we have first to establish some  properties of the measure $\eta$.   Consider the set
\begin{equation}
\label{spa.infty}
 \mathfrak{L}^\infty(\bar I,\mathscr{Z}_{s})=\{u\in \mathscr{C}(\bar I, \zedsi) :  \sup_{t\in \bar I} ||u(t)||_{\zeds}<\infty\}\,.
\end{equation}

\begin{lem}
\label{mes-lem1}
Assume the same assumptions as in Prop.~\ref{prob-rep} and suppose that the curve $t\in I\to\mu_{t}\in \mathfrak{P}(\mathscr Z_s)$ satisfies \eqref{A1}. Then
$$
\mathcal{F}_{t_0}:=\left\{ (x,\gamma)\in  \mathscr{Z}_s\times \mathfrak{L}^\infty(\bar I,\mathscr{Z}_{s}); \,\gamma(t)=x+
\int_{t_0}^t v(s,\gamma(s)) \, ds, \,\forall t\in \bar I\right\},
$$
is a Borel subset   of $\mathfrak{X}$ satisfying $\eta( \mathcal{F}_{t_0})=1$ where  $\eta$ is the probability measure given in Prop \ref{prob-rep}.
\end{lem}
\begin{proof}
We first prove that $ \mathfrak{L}^\infty(\bar I,\mathscr{Z}_{s})$ is  a Borel subset of the space $\mathscr{C}(\bar I, \zedsi)$ endowed with the norm of the uniform convergence,
$$
||u||_{\mathscr{C}(\bar I, \mathscr{Z}_{-\sigma})}=\sup_{t\in \bar I}||u(t)||_{\mathscr{Z}_{-\sigma}}\,.
$$
Indeed, the map
\begin{eqnarray*}
\phi_n: \mathscr{C}(\bar I, \zedsi) &\longrightarrow & \R\\
u&\longrightarrow & \sup_{t\in\bar I} || A^{\frac{s+\sigma}{2}} (1+\frac{A}{n})^{-\frac{s+\sigma}{2}} u(t)||_{\zedsi}\,
\end{eqnarray*}
is clearly continuous (since $\phi_n$ defines an equivalent norm on  $\mathscr{C}(\bar I, \zedsi)$) and converges, as $n\to \infty$, to
\begin{equation}
\phi(u)=
\left\{
\begin{aligned}
&\infty & \text{ if  } u\notin \mathfrak{L}^\infty(\bar I,\mathscr{Z}_{s})\\
 &\sup_{t\in\bar I} || u(t)||_{\zeds} & \text{ if  } u\in \mathfrak{L}^\infty(\bar I,\mathscr{Z}_{s})\,.
\end{aligned}
\right.
\end{equation}
Since $\phi$ is measurable, the subsets
$$
\mathfrak{L}^\infty(\bar I,\mathscr{Z}_{s})=\phi^{-1}(\R) \qquad \text{ and } \qquad \mathfrak{L}_m^\infty(\bar I,\mathscr{Z}_{s}):=\phi^{-1}([0,m]),
$$
are Borel. Furthermore,  with a similar argument, one  also proves that $\zeds$ is a Borel subset of $\zedsi$
 (see e.g.~\cite[Appendix]{MR3721874}).    Hence,  $\zeds\times \mathfrak{L}_m^\infty(\bar I, \zeds)$ is a
 Borel subset of the space $\mathfrak X$ endowed with the norm \eqref{normX}. And consequently,
 the Borel  $\sigma$-algebra   of $(\zeds\times\mathfrak{L}_m^\infty(\bar I, \zeds), ||\cdot||_{\mathfrak X})$
 coincides with the $\sigma$-algebra   of all Borel sets of $(\mathfrak X,||\cdot||_{\mathfrak X})$ contained in  $\zeds\times\mathfrak{L}_m^\infty(\bar I, \zeds)$.

\bigskip
\noindent
Now, we claim that the map $\psi_m: \zeds\times \mathfrak{L}_m^\infty(\bar I, \zeds) \longrightarrow  \R$ defined by
\begin{eqnarray*}
\psi_m(x,u)= \sup_{t\in\bar I}
||u(t)-x-\int_{t_0}^t v(\tau,u(\tau)) d\tau||_{\zedsi}\,
\end{eqnarray*}
is measurable. In fact, we have the following composition of measurable maps
\[
\begin{array}{ccccccc}
[t_0,t]\times \mathfrak{L}_m^\infty(\bar I, \zeds) & \overset{(1)}{\longrightarrow} & \bar I\times \zeds & \overset{(2)}{\longrightarrow}  &  \zedsi
 & \overset{(3)}{\longrightarrow} & \R \\
(\tau,u) & \longrightarrow & (\tau,u(\tau)) & \longrightarrow &  v(\tau,u(\tau)) & \longrightarrow &  \Re\langle v(\tau,u(\tau)), y\rangle_{\zedsi} \,
\end{array}
\]
where  $(2)$ is measurable by \eqref{A0},  $(3)$  is continuous for any $y\in\zedsi$ and $(1)$ is also measurable
since $\zeds$ is a Borel subset of $\zedsi$ and (1) is continuous if it is considered as a mapping into $I\times\zedsi$.  Applying Lemma \ref{classmo} in the Appendix C, one concludes that the following mappings  are measurable for  any $t\in \bar I$,
\begin{eqnarray*}
 \mathfrak{L}_m^\infty(\bar I, \zeds)  &\longrightarrow & \R\\
u&\longrightarrow & \int_{t_0}^t  \Re \langle v(\tau,u(\tau)),y\rangle_{\zedsi}\, d\tau \,.
\end{eqnarray*}
Since $\zedsi$ is a separable Hilbert space then  by Pettis theorem,
weak measurability and strong measurability coincide; and this implies that the  mappings
\begin{eqnarray*}
 \mathfrak{L}_m^\infty(\bar I, \zeds)  &\longrightarrow & \zedsi\\
u&\longrightarrow & \int_{t_0}^t   v(\tau,u(\tau))\, d\tau \,.
\end{eqnarray*}
are actually measurable for any $t\in \bar I$. Notice that the latter integrand is Bochner integrable thanks to the assumption \eqref{A0} and the fact that $u(\cdot)$ is a bounded function valued in $\zeds$.
So, combining this with the continuity of the mappings,
\[
\begin{array}{ccc}
\zeds\times \mathfrak{L}_m^\infty(\bar I, \zeds) & {\longrightarrow} & \zedsi \\
(x,u) & \longrightarrow & u(t)-x \,,
\end{array}
\]
 then one concludes that
\[
\begin{array}{ccc}
\zeds\times \mathfrak{L}_m^\infty(\bar I, \zeds) & \longrightarrow & \R \\
(x,u) & \longrightarrow &  \displaystyle\sup_{t\in\mathbb{Q}\cap I}||u(t)-x-\int_{t_0}^t  v(\tau,u(\tau)) \, d\tau||_{\zedsi}\,
\end{array}
\]
is measurable. Using the assumption  \eqref{A0}, one shows that the curve $t\to u(t)-x-\int_{t_0}^t  v(\tau,u(\tau)) \, d\tau\in\zedsi$ is continuous for any fixed $x\in\zeds$ and  $u\in  \mathfrak{L}_m^\infty(\bar I, \zeds)
\subset \mathscr{C}(\bar I, \zedsi)$.  Hence,
$$
\sup_{t\in\mathbb{Q}\cap I}||u(t)-x-\int_{t_0}^t  v(\tau,u(\tau)) \, d\tau||_{\zedsi}=\sup_{t\in\bar I}||u(t)-x-\int_{t_0}^t  v(\tau,u(\tau)) \, d\tau||_{\zedsi}\,,
$$
and  therefore
$$
\mathcal{F}_{t_0}=\bigcup_{m\in\N}\psi_m^{-1}(\{0\})
$$
is a Borel subset of $\mathfrak X$. Furthermore, using  Prop.~\ref{prob-rep}-\ref{concen-item2}, one proves  for all $t\in I$,  $k\in\N$ and $M\geq R$,
$$
\int_{\mathscr{Z}_{-\sigma}} 1_{B_{\zeds}(0,M)}(x) \, ||x||^{2k}_{\mathscr{Z}_s} \,d\mu_t(x)=\int_{\mathfrak{X}}
 1_{B_{\zeds}(0,M)}(u(t)) \, ||u(t)||^{2k}_{\mathscr{Z}_s} \,d\eta(x,u)\leq R^{2k}\,,
$$
since the  function $\varphi:\zedsi\to \R$, $\varphi(x)=1_{B_{\zeds}(0,M)}(x) \, ||x||^{2k}_{\mathscr{Z}_s} $ is Borel and bounded. So, by Fatou's lemma and assumption \eqref{A1}, letting $M\to\infty$ yields :
$$
\int_{\mathfrak{X}}   ||u(t)||^{2k}_{\mathscr{Z}_s} \,d\eta(x,u)\leq R^{2k}\,.
$$
On the other hand, by H\"older's inequality
$$
\int_{\mathfrak{X}} ||u(\cdot)||_{L^{2k}(I, \mathscr{Z}_s)} \,d\eta(x,u)\leq
\left(\int_I \int_{\mathfrak{X}}  ||u(t)||^{2k}_{\mathscr{Z}_s}   \,d\eta(x,u) \,dt\right)^{1/2k} \leq
|I|^{1/2k} R\,.
$$
Again by Fatou's lemma, letting $k\to \infty$ gives
$$
\int_{\mathfrak{X}} ||u(\cdot)||_{L^{\infty}(I, \mathscr{Z}_s)} \,d\eta(x,u)\leq R\,.
$$
So the norm $ ||u(\cdot)||_{L^{\infty}(I, \mathscr{Z}_s)}$  is  finite  for $\eta$-a.e.~$(x,u) \in \mathfrak{X}$.  Combining this fact with Prop.~\ref{prob-rep}-\ref{concen-item1} one concludes that there exists an  $\eta$-negligible set  $\mathcal{N}$ such that
$$
\mathcal{N}^c\subset \mathcal{F}_{t_0} \quad \text{ and } \quad \eta(\mathcal N)=0\,.
$$
Notice that if $u(\cdot)$ is a solution of \eqref{IVP} with  $u(\cdot) \in L^{\infty}(I, \mathscr{Z}_s) \cap W^{1,1}(I,\zedsi)$ then $u(\cdot)$ satisfies the integral equation \eqref{int-form}  and
 $u(\cdot)\in L^{\infty}(I, \mathscr{Z}_s)\cap W^{1,\infty}(I,\zedsi)$ (i.e. $u(\cdot)$ is a weak solution according to Definition \ref{w-ssol}). In particular, $u(\cdot)$ belongs to $\mathfrak{L}^\infty(\bar I,\mathscr{Z}_{s})$ and it is weakly continuous in $\zeds$. \\
Finally since $\mathcal{F}_{t_0}$ is measurable, then $\eta( \mathcal{F}_{t_0})=1$.

\end{proof}

\begin{lem}
\label{mes-lem2}
Assume the same assumptions as in Prop.~\ref{prob-rep} and suppose that uniqueness of weak solution for the \eqref{IVP} holds true. Then
$$
\mathcal{G}_{t_0}:=\left\{ x\in  \mathscr{Z}_s : \exists \gamma\in \mathfrak{L}^\infty(\bar I,\mathscr{Z}_{s})
\text{ s.t. } (x,\gamma)\in \mathcal{F}_{t_0}\right\},
$$
is a Borel subset   of $ \mathscr{Z}_s$.
\end{lem}
\begin{proof}
Recall the following known result in measure theory \cite[Thm.~3.9]{MR0226684}. Let  $X_1,X_2$  two complete metric spaces and $E_1\subset X_1$,  $E_2\subset X_2$ two subsets such that $E_1$ is Borel.  If  $\varphi: E_1\to X_2$ is a measurable one-to-one map such that $\varphi(E_1)=E_2$; then $E_2$ is a Borel subset of $X_2$.
Using  such argument,  one shows the claimed result. Indeed, consider
$X_1=(\mathfrak X,||\cdot||_{\mathfrak X})$, $X_2= (\zedsi, ||\cdot||_{\zedsi})$  two complete normed spaces and $\varphi$ given by
\[
\begin{array}{cccc}
 \varphi : & \mathfrak X  & \longrightarrow & \zedsi\\
                &   (x,u)            & \longrightarrow  & x\,.
\end{array}
\]
Then, clearly $\varphi$ is a  continuous map. Hence by Lemma \ref{mes-lem1},  its restriction
$\varphi_{| \mathcal{F}_{t_0}}:\mathcal{F}_{t_0}\to \zedsi$ is a measurable map. Now, the uniqueness  hypothesis for weak solutions
of the initial value problem \eqref{IVP},  shows that $\varphi_{| \mathcal{F}_{t_0}}$ is one-to-one. Hence, the set
$$
\mathcal{G}_{t_0}= \varphi_{|\mathcal{F}_{t_0}}(\mathcal{F}_{t_0}),
$$
is Borel in $\zedsi$. Since $\mathcal{G}_{t_0}\subset \zeds$ then it is also a Borel subset of $\zeds$.
\end{proof}

\begin{prop}
\label{flotpp}
Assume the same assumptions as in Prop.~\ref{prob-rep} and suppose that uniqueness of weak solutions of  the \eqref{IVP} holds true. Then the map
\begin{eqnarray*}
\phi(t,t_0): \mathcal{G}_{t_0} &\to & \mathscr{Z}_s \\
x &\to & u(t)
\end{eqnarray*}
where  $u(\cdot)$ is the unique curve in $ \mathfrak{L}^\infty(\bar I,\mathscr{Z}_{s})\cap
W^{1,\infty} (I,  \mathscr{Z}_{-\sigma}) $ satisfying $u(t)=x+
\int_{t_0}^t v(\tau,u(\tau)) \, d\tau,$ for all $ t\in \bar I$; is a well defined Borel map.
\end{prop}
\begin{proof}
Examining the proof of Lemma \ref{mes-lem2}, one notices that the inverse  map
\[
\begin{array}{cccc}
 \varphi^{-1} : & \mathcal{G}_{t_0}  & \longrightarrow & \mathcal{F}_{t_0}\\
                &       x    & \longrightarrow  & (x,u)
\end{array}
\]
is well defined and measurable. Moreover, the function $u$  such that $\varphi^{-1}(x)=(x,u)$ is the unique
weak solution of the initial value problem \eqref{IVP} satisfying $u(t_0)=x$.  The following composition
\[
\begin{array}{cccccc}
 \phi(t,t_0): & \mathcal{G}_{t_0}  & \overset{\varphi^{-1}}{\longrightarrow} & \mathcal{F}_{t_0}
  &\overset{e_t}{\longrightarrow} & \zedsi\\
                &       x    & \longrightarrow  & (x,u) & \longrightarrow  & u(t)
\end{array}
\]
yields a well defined measurable  map. Since $\phi(t,t_0)(\mathcal{G}_{t_0} )\subset \zeds$   and $\zeds$ is a Borel subset of $\zedsi$, then the claimed statement is proved.
\end{proof}

\subsection{Existence and uniqueness of solutions}
\label{sub.sec.uniq-ex}

\begin{thm}
\label{sec.0.thm3}
Let ${v}: \mathbb{R} \times \zeds \mapsto \zedsi$ be a vector field satisfying \eqref{A0}.  Then  uniqueness of weak solutions over a bounded open interval $I$ for the initial value problem
\eqref{IVP} implies the uniqueness of solutions over $I$ of the Liouville  equation \eqref{eq.transport}
 satisfying the assumption \eqref{A1}.
\end{thm}

{\bf Proof of Thm.~\ref{sec.0.thm2} and \ref{sec.0.thm3}}:
Thanks to the duality between hierarchy equations and Liouville equations, one only needs  to prove Thm.~\ref{sec.0.thm3}. Indeed, assume the assumptions stated in Thm.~\ref{sec.0.thm3} and suppose that we have two curves $t\in I\to \mu_t\in\mathfrak P(\zeds)$ and   $t\in I\to \nu_t\in\mathfrak P(\zeds)$ both satisfying \eqref{A2} and such that $\mu_{t_0}=\nu_{t_0}$ for some $t_0\in I$.  Then applying Prop. \ref{prob-rep}, one gets the existence of two probability measures $\eta_1$ and  $\eta_2$ on the space $\mathfrak X$ satisfying  respectively \ref{concen-item1}-\ref{concen-item2}. So that, for any bounded Borel function $f:\zeds \to \R$, we have
$$
\int_{\zeds} f(x) \, d\mu_t(x)= \int_{\mathfrak X} f(u(t)) \, d\eta_1(x,u) = \int_{\mathcal F_{t_0}} f(\phi(t,t_0)(x)) \, d\eta_{1}(x,u)= \int_{\mathcal G_{t_0}} f\circ\phi(t,t_0)(x) \, d\mu_{t_0}(x)\,.
$$
Recall that  $\mathcal F_{t_0}$, $\mathcal G_{t_0}$ and $\phi(\cdot,\cdot)$  are given respectively in Lemma
\ref{mes-lem1}, Lemma \ref{mes-lem2} and Proposition \ref{flotpp}.  Moreover,  in the last identities
we have used the fact that $(e_t)_{\sharp}\eta_1=\mu_t$, the concentration property  $\eta_1(\mathcal F_{t_0})=1$  in Lemma \ref{mes-lem1} and the measurability of the map $\phi(t,t_0)$  in Lemma \ref{mes-lem2}.  So, one concludes for any
$t\in I$,
$$
\mu_t= \phi(t,t_0)_{\sharp}\mu_{t_0}\,.
$$
Repeating the same argument for $\nu_t$ yields the same result so that for any $t\in I$,
$$
\mu_t= \phi(t,t_0)_{\sharp}\mu_{t_0}=\phi(t,t_0)_{\sharp}\nu_{t_0}=\nu_t\,.
$$
$\hfill\square$

\bigskip
\noindent
We state an  existence result of solutions for the Liouville equation \eqref{eq.transport}.
\begin{prop}
\label{ext-1}
Let $v:\R\times \mathscr Z_s\to \mathscr Z_{-\sigma}$ a Borel vector field which is bounded on bounded sets  and  let $I$ be  a bounded open interval with $t_0\in I$ a fixed initial time. Assume that there exists a Borel
set $\mathcal{A}$ of $\zeds$ and a Borel  map $\phi:\bar I\times \mathcal{A}\to \zeds$ which is bounded on bounded sets
and such that for any $x\in \mathcal{A}$ the curve $t\in \bar I \to \phi(t,x)$ is a weak solution of the initial value problem \eqref{IVP} satisfying $ \phi(t_0,x)=x$.  Then for any Borel probability measure $\nu\in\mathfrak{P}(\zeds)$, such that $\nu$ is concentrated on $ \mathcal{A}$ and on a bounded subset of $\zeds$, there exists a solution $t\in I\to \mu_t$ to the Liouville equation
\eqref{int-liouville} given by
\begin{equation}
\label{mes-propag}
\mu_t=\phi(t,\cdot)_{\sharp}\nu\,,
\end{equation}
satisfying $\mu_{t_0}=\nu$. Furthermore,  $t\in I\to \mu_t$ is strongly narrowly continuous in $\mathfrak{P}(\zedsi)$.
\end{prop}
\begin{proof}
Since the map $\phi$ is Borel then  $\phi(t, \cdot): \mathcal{A}\to \zeds$ is  also Borel. So, one can define
$\mu_t$ according to \eqref{mes-propag} as a Borel probability measure on $\zeds$ or $\zedsi$. Moreover, for any bounded continuous  function $f:\zedsi\to \R$, one easily checks that
$$
t\in I\longrightarrow \int_{\zedsi} f(x) \,d\mu_t=  \int_{\mathcal{A}} f(\phi(t,x)) \,d\nu\,,
$$
is continuous. So, the curve $t\in I\to \mu_t$ is strongly (weakly) narrowly continuous in $\mathfrak{P}(\zedsi)$ and satisfies $\mu_{t_0}=\nu$.  It still to prove that  $t\in I\to \mu_t$ is a solution of the Liouville equation \eqref{int-liouville}.
Let $\psi\in \mathscr{C}_0^\infty(\R^n)$ and $(e_1,\cdots,e_n)$ an orthonormal family in $\zedsi$. Here  $\zedsi
\equiv \mathscr{Z}_{-\sigma,\R}$ is considered as real Hilbert space. We use the notation in \eqref{eq.pi},
\begin{equation*}
\pi(x)=(\langle x, e_1\rangle_{\mathscr Z_{-\sigma},\R}, \cdots, \langle x, e_n\rangle_{\mathscr Z_{-\sigma},\R})\,,
\end{equation*}
so that  $\varphi(x)=\psi(\pi(x))\in\mathscr{C}_{0,cyl}^\infty(\zedsi)$. Then a simple computation yields
\begin{eqnarray*}
\frac{d}{dt} \int_{\zeds} \varphi(x) \, d\mu_t(x) %&=&\frac{d}{dt} \int_{\mathcal{A}} \psi\big(\langle \phi(t,x), e_1\rangle_{\zedsi,\R},\cdots , \langle \phi(t,x)x, e_n\rangle_{\zedsi,\R} \big)
%\, d\nu(x) \\
&=& \int_{\mathcal{A}} \sum_{j=1}^n \frac{d}{dt} \langle  \phi(t,x), e_j\rangle_{\zedsi,\R} \;
\partial^j \psi\big(\langle  \phi(t,x), e_1\rangle_{\zedsi,\R},\cdots,  \langle  \phi(t,x), e_n\rangle_{\zedsi,\R} \big)
\, d\nu  \\
%&=& \int_{\mathcal{A}}  \langle  v(t,\phi(t,x)), \nabla\psi\big(\pi(\phi(t,x)))\rangle_{\zedsi,\R}  \, d\nu(x)
% \\
&=& \int_{\zeds}  \langle  v(t,y), \nabla\varphi(y)\rangle_{\zedsi,\R}  \, d\mu_t(y)\,.
\end{eqnarray*}
The last equalities follow using two arguments: First, $t\to \phi(t,x)$ is a weak solution  of the initial value problem
\eqref{IVP} for any $x\in\mathcal{A}$  and it is an  absolutely continuous curve in $\zedsi$.
Second,  $\nu$ is concentrated on a bounded set of $\zeds$ and $\phi$ and $v$ are bounded on bounded sets. So,
one can use dominated convergence in order to switch between  time derivatives and integration.  A standard density argument gives the Liouville equation \eqref{int-liouville} with the measures $(\mu_t)_{t\in I}$.
\end{proof}

{\bf Proof of Thm.~\ref{ext-2}}:
Again we use the duality between hierarchy equations and Liouville equations  proved in Prop.~\ref{thm.main} and \ref{thm.main2}. Recall that we have also proved a duality between the assumptions \eqref{A2}  and \eqref{A1} in Lemma \ref{a1a2}. So, for any $\gamma\in \mathscr{H}(\zed)$ satisfying the hypothesis of Thm.~\ref{ext-2} there exists, by Prop.~\ref{str.BEh}, a probability measure $\nu\in\mathscr{M}(\zed)$ such that
\begin{equation*}
\gamma^{(k)}=\int_{\zeds} |\varphi^{\otimes k} \rangle \langle \varphi^{\otimes k} | \;d\nu(\varphi)\,, \quad\forall k\in\N\,.
\end{equation*}
Moreover, according to  Lemma \ref{concent}  the measure $\nu$ is concentrated on a bounded subset of $\zeds$. Applying Prop.~\ref{ext-1} then there exits a solution $t \in I \to \mu_t\in\mathfrak{P}(\zeds)$ of the Liouville equation \eqref{eq.transport} satisfying the initial condition $\mu_{t_0}=\nu$. Notice also that since the map $\phi$ transform bounded sets on bounded sets of $\zeds$, then $\mu_t$ concentrates on a ball $B_{\zeds}(0,R)$  for all $t\in I$.
Taking for any $k\in\N$,
$$
\gamma_t^{(k)}=\int_{\zeds} |\varphi^{\otimes k} \rangle \langle \varphi^{\otimes k} | \;d\mu_t(\varphi)\,,
$$
then one easily checks that $t \in I \to \gamma_t=(\gamma_t^{(k)})_{k\in\N}$ satisfies  \eqref{A2} and solves the
hierarchy equation \eqref{int-hier}  according to Prop.~\ref{thm.main} since  $t \in I \to \mu_t$ solves the Liouville equation.  $\hfill\square$

\medskip
{\bf Proof of Prop~\ref{uniq-unc1} and \ref{uniq-unc2}}:
We can not use directly  Thm.~\ref{sec.0.thm2} because the results in  Prop~\ref{uniq-unc1} and \ref{uniq-unc2}
are of conditional type. However, the proof is quite similar and uses Proposition \ref{prob-rep} as well and follows the
same  scheme as in Thm.~\ref{sec.0.thm2}.  So, we just indicate the main point in the proof.
Notice that  the vector field $v$ given in \eqref{vect-field}  may not be  bounded on bounded sets of $L^2(\R^d)$. Nevertheless, the Proposition \ref{prob-rep}  (or \cite[Prop.~4.1]{MR3721874}) still holds true  under the following  weaker assumption,
$$
\int_{I} \int_{L^2} || v(t,x) ||_{H^{-1}(\R^d)} \;d\mu_t(x) \,dt <\infty\,.
$$
The above inequality can be checked by simply using \eqref{unc.eq.1} and \eqref{cond-Stri}. Hence, there exists a probability measure $\eta$ over the space $\mathfrak{X}$, defined  as in \eqref{eq.X}, which concentrates on the solutions $u\in   L^\infty(I, L^2(\R^d))\cap W^{1,1}(I, H^{-1}(\R^d))$ of the NLS equation \eqref{eq.ivp} written in the interaction representation, i.e.,
$$
u(t)=x+\int_{0}^t  v(\tau,u(\tau)) \, d\tau\,,
$$
where $v$ is the Borel vector field given in \eqref{vect-field}. The additional assertion $u\in   L^\infty(I, L^2(\R^d))$
is deduced from the  assumption  \eqref{A1} with $s=0$ and $A=-\Delta+1$ satisfied by $(\mu_t)_{t\in I}$ solution of the corresponding Liouville equation.   Moreover, the requirement  \eqref{cond-Stri} with Prop.~\ref{prob-rep}-\ref{concen-item2}  yield
$$
\int_{I} \int_{L^2} ||\,\mathcal{U}(t) \,x||_{L^r}^q\; d\mu_t(x)\,dt =\int_{\mathfrak{X}} \int_I  ||\,\mathcal{U}(t) \,u(t)||_{L^r}^q \,dt\,d\eta(x,u)<\infty\,.
$$
Hence, one notices that $\mathcal{U}(\cdot)u(\cdot) \in L^q(I, L^r(\R^d))$ for $\eta$-a.e.~$(x,u)\in \mathfrak{X}$.   Applying Strichartz's estimate to the Duhamel formula
\begin{equation}
\label{pro.duha}
\mathcal{U}(t)u(t)=\mathcal{U}(t)x-i\int_{0}^t  \mathcal{U}(t-\tau) \,G(\mathcal{U}(\tau) u(\tau)) \, d\tau\,,
\end{equation}
one concludes that $\mathcal{U}(\cdot)u(\cdot)\in \mathscr{C}(\bar I, L^2(\R^d))\cap  L^q(I, L^r(\R^d))$ for $\eta$-a.e.~$(x,u)\in \mathfrak{X}$.
So now using the result of Tsutsumi \cite{MR915266} or more precisely \cite[Theorem 4.6.1]{MR2002047} one
can complete the proof as in Prop.~\ref{sec.0.thm3} and  obtains the
uniqueness for the Liouville solutions satisfying \eqref{A1} and \eqref{cond-Stri}.   Consequently, the duality result of Thm.~\ref{sec.0.thm2}  gives the  uniqueness for the corresponding hierarchy equation.
$\hfill\square$

\appendix
\begin{center}
{\bf Appendix}
\end{center}
In this appendix we recall a few useful known results concerning the topology of the space of trace class operators
and  the space of probability measures as well as some useful arguments for measurable sets and maps.
\section{Kadec-Klee property }
\label{KKstar}
A dual Banach space $(E,||\cdot||)$ have the Kadec-Klee property (KK*) if  any sequence
$(x_n)_{n\in\N}$ which converges with respect to the weak-$*$ topology
to a limit $x\in E$ with $\lim_n ||x_n||=||x||$, satisfies  $x_n\to x$ in the norm topology of $E$
(see e.g.~ \cite{MR943795,MR2154153}).

\begin{thm}
\label{th.kadec}
The space of trace class-operators $\mathscr{L}^1(\mathfrak{H})$, over a separable
Hilbert space $\mathfrak{H}$, satisfies the Kadec-Klee property (KK*).
\end{thm}

\section{Weak and strong narrow convergence }
\label{measure}

Let $\mathfrak{H}$ be a separable Hilbert space and $X_1$ its closed unit ball $B_\mathfrak{H}(0,1)$.
The set of Borel probability measures $\mathfrak{P}(X_1)$ can be endowed with the weak and strong narrow convergence topology defined according to \eqref{defnarroww}-\eqref{defnarrows}. Let $(e_i)_{i\in \N}$ be an O.N.B
of the Hilbert space $\mathfrak{H}$. We have the following  two useful results.

\begin{thm}
\label{thmA}
Let $(\mu,(\mu_j)_{j\in\N})$ be a sequence in $\mathfrak{P}(X_1)$. Then the  pointwise  convergence of  the characteristic functions,
$$
\hat\mu_j(y)=\int_{X_1} e^{-i {\rm Re}\langle x,y\rangle} \,d\mu_j \;\to \hat\mu(y)=\int_{X_1} e^{-i {\rm Re}\langle x,y\rangle} \,d\mu, \quad \forall y\in \mathfrak{H},
$$
implies the weak narrow converges of the  measures $\mu_j\rightharpoonup \mu$.
\end{thm}
\begin{proof}
Since the closed unit ball $X_1$ is compact and separable with respect to the weak topology (induced  for instance by the metric $d_w$ in \eqref{dweak}), then by Prokhorov's theorem the sequence $(\mu_j)_{j\in\N}$ is relatively sequentially compact in $\mathfrak{P}(X_1)$ with respect to  the weak narrow topology.
Moreover, $(\mu_j)_{j\in\N}$  admits a unique cluster point. Because otherwise, the characteristic functions
$\hat \mu_j$ would converge pointwisely  to two distinct  limits. So, the sequence $(\mu_j)_{j\in\N}$ should converge weakly narrowly to the measure $\mu$.
\end{proof}

\begin{thm}
\label{thmB}
Let $(\mu,(\mu_j)_{j\in\N})$ be a sequence in $\mathfrak{P}(X_1)$. Then
$$
\left(\mu_j\rightharpoonup \mu \text{ and } \quad \forall \varepsilon>0, \quad \lim_{N\to \infty } \sup_{j\in\N}
\int_{X_1}\; 1_{\{\sum_{i=N}^\infty |\langle \varphi, e_i\rangle|^2\geq \varepsilon\}} \; d\mu_j=0 \right) \;\Longleftrightarrow \;
\mu_j\to \mu\,.
$$
\end{thm}
\begin{proof} We refer for instance to \cite[Theorem 1]{MR1004337}.
\end{proof}

\section{Measurable maps and sets}
Let $I$ be a bounded open interval and consider the space $\mathfrak X=\zedsi\times \mathscr{C}(\bar I,\zedsi)$,
introduced in Subsection \ref{subsec.probrep}, and  endowed with the two norms,
$$
||(x,\varphi)||_{\mathfrak X}= ||x||_{\mathscr{Z}_{-\sigma}}+\sup_{t\in \bar I}||\varphi(t)||_{\mathscr{Z}_{-\sigma}}\,, \qquad ||(x,\varphi)||_{\mathfrak X_{w}}= ||x||_{\mathscr{Z}_{-\sigma,w}}+\sup_{t\in \bar I}||\varphi(t)||_{\mathscr{Z}_{-\sigma,w}}\,,
$$
where
$$
 ||x||^2_{\mathscr{Z}_{-\sigma,w}}=\sum_{n\in\N} \frac{1}{2^n} |\langle e_n, x\rangle_{\zedsi}|^2 \,,\qquad ||\varphi(t)||_{\mathscr{Z}_{-\sigma,w}}^2=\sum_{n\in\N} \frac{1}{2^n} |\langle e_n, \varphi(t)\rangle_{\zedsi}|^2\,,
$$
with $(e_n)_{n\in\N}$ is an  O.N.B.~of $\zedsi$.

\begin{lem}
\label{tribu}
The $\sigma$-algebras of Borel sets of $(\mathfrak X,||\cdot||_{\mathfrak X} )$  and  $(\mathfrak X, ||\cdot||_{\mathfrak X_{w}})$ coincide.
\end{lem}
\begin{proof}
Let $\imath: (\mathfrak X, ||\cdot||_{\mathfrak X}) \to (\mathfrak X, ||\cdot||_{\mathfrak X_{w}})$ be the identity map. It is clear that  $\imath$ is continuous and hence it is measurable. This implies  the following inclusion
of  $\sigma$-algebras,
$$\mathscr{T}(\mathfrak X, ||\cdot||_{\mathfrak X_{w}})\subset\mathscr{T}(\mathfrak X, ||\cdot||_{\mathfrak X})\,. $$
To show the inversion inclusion, one uses an approximation of identity,
\begin{eqnarray*}
\Psi_{\varepsilon}: (\mathfrak X, ||\cdot||_{\mathfrak X_{w}}) & \longrightarrow & (\mathfrak X, ||\cdot||_{\mathfrak X})\\
(x,\varphi) & \longrightarrow & \sum_{n\in\N} \frac{1}{1+n\varepsilon} \,
\big(\langle x, e_n\rangle_{\zedsi} e_n ; \langle \varphi, e_n\rangle_{\zedsi} e_n\big)\,,
\end{eqnarray*}
for some $\varepsilon>0$. It is clear that $\Psi_\varepsilon$ is continuous. Furthermore, for $x$ and $\varphi$ fixed
$$
\lim_{\varepsilon\to 0} \sum_{n\in\N} \frac{1}{1+n\varepsilon} \,\langle x, e_n\rangle_{\zedsi} e_n =x\quad
\text{ in } (\zedsi,||\cdot||_{\zedsi})\,;
$$
and
$$
\lim_{\varepsilon\to 0} \sum_{n\in\N} \frac{1}{1+n\varepsilon}  \langle \varphi, e_n\rangle_{\zedsi} e_n=\varphi
\quad
\text{ in } \mathscr{C}(\bar I,\zedsi)\,.
$$
The last limit follows using the fact that $\varphi(\bar I)$ is compact and the operator $\sum_{n\in\N} \frac{1}{1+n\varepsilon}  |e_n\rangle \langle e_n|$ converges strongly, as $\varepsilon\to 0$, to the identity on $\zedsi$.
Since one can show that $\Psi_\varepsilon$ converges pointwisely to the identity map $\imath^{-1}: (\mathfrak X, ||\cdot||_{\mathfrak X_{w}}) \to (\mathfrak X, ||\cdot||_{\mathfrak X})$, then the inverse inclusion holds true.

\end{proof}

\begin{lem}
\label{classmo}
Let $(M,d)$ be a metric space. Then for any  bounded measurable function  $f:[a,b]\times M\to \R$, the mapping
\begin{equation}
\label{mes.eq.1}
\begin{aligned}
M & \longrightarrow  \R\\
x  & \longrightarrow  \int_a^b f(\tau,x) \,d\tau
\end{aligned}
\quad\text{ is measurable}.
\end{equation}
\end{lem}
\begin{proof}
It suffices to  use the monotone class  theorem. Consider the set
$$
\mathscr{F}:=\{f:[a,b]\times M\to \R \text{ bounded measurable and satisfying } \eqref{mes.eq.1}\}\,.
$$
Then one  checks:
\begin{enumerate}
\item $\mathscr{F}$ is stable with respect to addition and scalar multiplication.
\item  $\mathscr{F}$ is stable with respect to monotone convergence, i.e.: if
$(f_n)_{n\in\N}$ is a bounded sequence of non-negative function in $\mathscr{F}$ such that
$(f_n)_{n\in\N}$  is a non-decreasing sequence of functions which converges pointwisely to  a function $f$, as $n\to\infty$, then $f\in\mathscr{F}$.
\item The set $\mathscr{A}:=\{ F\subset [a,b]\times M ,  F \text{ closed } \}$ is a $\pi$-system and for any $F\in \mathscr{A}$ the  indicator function $1_{F}$ belongs to $\mathscr{F}$.
\end{enumerate}
Consequently,  $\mathscr{F}$ contains all real-valued bounded measurable functions on  $[a,b]\times M$.
\end{proof}

The following observation is useful in Prop.~\ref{uniq-unc1}-\ref{uniq-unc2}   where we consider the uniqueness of hierarchy equations related to the  NLS equation \eqref{eq.ivp} with a nonlinearity $g$ given by \eqref{unc-g}.

\begin{lem}
\label{ext-borel-field}
The space $L^2(\R^d)\cap L^r(\R^d)$ is a Borel subset of $(L^2(\R^d),||\cdot||_{L^2(\R^d)})$. Furthermore, there exists a Borel extension $G:L^2(\R^d)\to H^{-1}(\R^d)$ of the  mapping $$g:L^2(\R^d)\cap L^r(\R^d)\to  L^{r'}(\R^d)$$ given  in  \eqref{unc-g} and satisfying \eqref{unc.eq.1}.
\end{lem}
\begin{proof}
Let $\varphi_n$ be a mollifier (i.e., $\varphi\in \mathscr{C}_0^\infty(\R^d)$, $\varphi\geq 0$, $\int_{\R^d} \varphi(x) dx=1$  and $\varphi_n(x):=n^d \varphi(nx)$). Then the mapping
\begin{eqnarray*}
T_n: L^2(\R^d) &\longrightarrow &  \R_+ \\
f & \longrightarrow & ||\varphi_n*f||_{L^r}
\end{eqnarray*}
is well defined and continuous thanks to  Young's inequality. Moreover, one easily checks that
\begin{equation*}
\lim_{n\to \infty} T_n(f)=\left\{
\begin{aligned}
||f||_{L^r}  &\quad\text{ if }  f\in L^2(\R^d)\cap L^r(\R^d)&\\
\infty   &\quad \text{ ifnot.} &
\end{aligned}
\right.
\end{equation*}
Hence, $L^2(\R^d)\cap L^r(\R^d)$ is a measurable subset of $L^2(\R^d)$. Consider the mapping
\begin{eqnarray*}
G: L^2(\R^d)& \longrightarrow & L^{r'}(\R^d)\subset H^{-1}(\R^d)  \\
 u &\longrightarrow & 1_{L^2\cap L^r}(u) \,g(u)\,.
\end{eqnarray*}
Then $G$ is  a Borel map extending $g$. Indeed, the sequence of mapping  $g_n$ defined as
\begin{eqnarray*}
g_n: L^2(\R^d)& \longrightarrow & H^{-1}(\R^d)  \\
 u &\longrightarrow & 1_{L^2\cap L^r}(u) \,g(\varphi_n*u)\,,
\end{eqnarray*}
are measurable since $u\to g(\varphi_n*u)$ are  continuous using the assumption  \eqref{unc.eq.1} and $\lim_{n\to\infty}g_n(u)=G(u)$ for any  $u\in L^2(\R^d)$.
\end{proof}

\end{document}